\documentclass[a4paper]{amsart}
\usepackage{amsmath,amsthm,amssymb, times, tikz-cd, cite}
\usepackage[british]{babel}
\usepackage[utf8]{inputenc}
\usepackage[normalem]{ulem}
\usepackage{enumerate}
\usepackage{extarrows}
\usepackage{comment}
\usepackage[pdfencoding=auto]{hyperref}

\title[Gluing asymptotically cylindrical Calabi-Yau threefolds]{Gluing and deformation of asymptotically cylindrical Calabi-Yau manifolds in complex dimension three}
\author{Tim Talbot}
\address{DPMMS\\Centre for Mathematical Sciences\\Wilberforce Road\\Cambridge\\CB3 0WB}
\email{T.J.Talbot@maths.cam.ac.uk}
\thanks{This work was supported by the UK Engineering and Physical Sciences Research Council (EPSRC) grant EP/H023348/1 for the University of Cambridge Centre for Doctoral Training, the Cambridge Centre for Analysis.}
\date{}

\theoremstyle{definition}
\newtheorem{defin}{Definition}[section]

\theoremstyle{plain}
\newtheorem{lem}[defin]{Lemma}
\newtheorem{prop}[defin]{Proposition}
\newtheorem{thm}[defin]{Theorem}

\newtheorem{cor}[defin]{Corollary}
\newtheorem*{thma}{Theorem A}
\newtheorem*{thmb}{Theorem B}
\newtheorem*{thmc}{Theorem C}
\newtheorem{claim}[defin]{Claim}
\theoremstyle{remark}
\newtheorem*{rmk}{Remark}

\newtheorem{numrmk}[defin]{Remark}
\numberwithin{equation}{section}

\newcommand{\R}{{\mathbb{R}}}
\renewcommand{\C}{{\mathbb{C}}}
\newcommand{\Z}{{\mathbb{Z}}}

\newcommand{\ddth}{\frac{\partial}{\partial \theta}}
\newcommand{\ddt}{\frac{\partial}{\partial t}}

\newcommand{\into}{{\hookrightarrow}}

\newcommand{\Kahler}{K\"ahler}
\newcommand{\ie}{i.e.\ }
\newcommand{\eg}{e.g.\ }

\newcommand{\Diff}{{\mathrm{Diff}}}
\newcommand{\Aut}{{\mathrm{Aut}}}

\newcommand{\abs}{\mathrm{abs}}

\newcommand{\bdd}{{\mathrm{bd}}}
\newcommand{\cpt}{{\mathrm{cpt}}} %compact parts of manifolds with ends
\newcommand{\id}{{\mathrm{id}}} %identity, considered as a function in a function space.

\newcommand{\la}{\langle}
\newcommand{\ra}{\rangle}

\renewcommand{\G}{{\mathcal{G}}} %moduli space of gluing data
\newcommand{\B}{{\mathcal{B}}} %base moduli space of gluing data
\newcommand{\M}{{\mathcal{M}}} %moduli spaces
\newcommand{\D}{{\Diff}} %spaces of diffeomorphisms
\renewcommand{\H}{\mathcal{H}}
\renewcommand{\L}{\mathcal{L}}

\DeclareMathOperator{\tr}{tr}
\DeclareMathOperator{\Vol}{Vol}
\DeclareMathOperator{\spn}{span}

\renewcommand{\Re}{\operatorname{Re}}
\renewcommand{\Im}{\operatorname{Im}}

\begin{document}
\begin{abstract}
We develop some consequences of the connection between Calabi-Yau structures and torsion-free $G_2$ structures on compact and asymptotically cylindrical six- and seven-dimensional manifolds. Firstly, we improve the known proof that matching asymptotically cylindrical Calabi-Yau threefolds can be glued. Secondly, we give an alternative proof that the moduli space of Calabi-Yau structures on a six-dimensional real manifold is smooth, and extend it to the asymptotically cylindrical case. Finally, we prove that the gluing map of Calabi-Yau threefolds, extended between these moduli spaces, is a local diffeomorphism: that is, that every deformation of a glued Calabi-Yau threefold arises from an essentially unique deformation of the asymptotically cylindrical pieces. 
\end{abstract}
\maketitle % makes  the title
\section{Introduction}
This paper is about Calabi-Yau threefolds, which we define as Riemannian manifolds with holonomy contained in $SU(3)$. These have been extensively studied: this condition immediately implies that the manifold is K\"ahler and Ricci-flat. Conversely, if we were interested in Ricci-flat K\"ahler manifolds, the Calabi conjecture, proved by Yau \cite{yaucalabiconjecture}, states that in the compact case it is essentially sufficient to consider compact K\"ahlerian manifolds whose canonical bundle has torsion first Chern class; it states that each such manifold has a unique Ricci-flat metric in each K\"ahler class. If, more strongly, we assume the canonical bundle is holomorphically trivial, so that there is a holomorphic volume form, then we have a parallel K\"ahler form and a parallel holomorphic volume form, so that the holonomy is contained in $SU(3)$. Such manifolds are interesting because of their restricted holonomy, because manifolds with restricted curvature are interesting in general, and because they are conjectured to be useful in physics: in certain forms of supersymmetric string theory, spacetime is conjectured to take the form $M \times K$ where $K$ is compact, Ricci-flat and K\"ahler. See \cite{strings}. 

As an auxiliary tool, we will use $G_2$ manifolds, that is those seven-dimensional manifolds with holonomy contained in the exceptional Lie group $G_2$. Since the subgroup of $G_2$ fixing a nonzero vector is isomorphic to $SU(3)$, the Riemannian product of a Calabi-Yau threefold and a general one-manifold is a $G_2$ manifold. The purpose of this paper is to use this correspondence to obtain results for Calabi-Yau threefolds from corresponding results for $G_2$ manifolds: because a small perturbation of a $G_2$ structures as a $3$-form is again a $G_2$ structure, the $G_2$ analysis is often easier than trying to work with Calabi-Yau structures directly.

Specifically, we shall study a gluing construction of compact Calabi-Yau threefolds given by taking asymptotically cylindrical Calabi-Yau threefolds and joining them to form a manifold with a neck. Combining geometric objects via manifolds with necks can also be done if the original objects were not asymptotically cylindrical. An important early example was the construction of self-dual conformal structures on four-manifolds, initiated by Floer\cite{floerprehistory}. In turn, such combinations lead to questions about which deformations of the glued structure arise from compatible deformations of the structures being glued and which, if any, arise from choices in the gluing. 

For instance, after showing that anti-self-dual connections on compact four-manifolds could be glued on a suitable connected sum, Donaldson and Kronheimer showed \cite[Theorem 7.2.63]{donaldsonkronheimer} that every deformation of an anti-self-dual connection on a four-manifold obtained by gluing can be obtained by either deformation of the connections being glued, variation of the length of the neck in gluing, or a change in the identification of the pieces, in an essentially unique way.

Later, Kovalev \cite{kovalevtwistedconnected} gave a gluing construction of $G_2$ manifolds from asymptotically cylindrical $G_2$ manifolds, in such a way that the holonomy was exactly $G_2$. The idea is that we can easily construct a $G_2$ structure with torsion, and that this torsion is small and can be removed by a perturbation argument. Removing small torsion was previously studied by Joyce \cite{joycebook} with a detailed result designed for use desingularising conical singularities. 

In the case of gluing asymptotically cylindrical $G_2$ structures, Nordstr\"om \cite{nordstromgluing} proved the analogous result on deformations: that every deformation of a $G_2$ structure obtained by gluing two matching asymptotically cylindrical $G_2$ structures is obtained, uniquely, by one of perturbing the glued structures, perturbing the identification of the two manifolds, or perturbing the length of the neck. 

We therefore seek a similar result in the Calabi-Yau case. We note that by work beginning with Tian--Yau and continued by Kovalev and Haskins--Hein--Nordstr\"om \cite{tianyauI, kovalevtwistedconnected, haskinsheinnordstrom}, asymptotically cylindrical Calabi-Yau manifolds do indeed exist. It is known that they can be glued by deformation in the sense of complex algebraic geometry. Using the smoothing results of Friedman \cite{friedman} and Kawamata--Namikawa \cite{kawanami}, Lee \cite{lee} showed that we can deform the singular space given by identifying the manifolds ``at infinity" to give a smooth manifold. This work leaves open the question of what the topology of the deformed space is and so whether it can be regarded as a gluing given by cutting off the structures to their asymptotically cylindrical limits, joining them to form a manifold with a long neck, and perturbing to maintain a Calabi-Yau structure. In a similar vein, it remains open whether all the different possible ways of deforming the singular space yield structures equivalent up to diffeomorphism, and so it's not clear that using these deformations gives a well-defined map on moduli spaces. We therefore will not consider this approach. 

We concentrate on the case $n=3$ because then we only have to introduce one extra dimension to reach seven dimensional $G_2$ manifolds. Similar arguments should work for $n<3$. If $n=2$, the obvious thing to do is to consider the Riemannian and holomorphic product of a Calabi-Yau twofold with a complex torus to obtain $SU(2)$ results from the $SU(3)$ results we prove in this paper. Hence, it would probably be necessary to consider the moduli space of Calabi-Yau structures on such a torus. In the same way, it should be possible given a ``Calabi-Yau curve" to construct a twofold by taking a product with a complex torus; however, as we need the moduli space of Calabi-Yau structures on a complex torus to apply this reduction, care would be required to avoid circular reasoning. 

By a Calabi-Yau structure, we mean a torsion-free $SU(n)$ structure. We regard both a Calabi-Yau and a torsion-free $G_2$ structure on a manifold as a set of differential forms satisfying appropriate algebraic and differential conditions; an asymptotically cylindrical such structure satisfies additional asymptotic conditions. We may then define moduli spaces $\M_{SU(3)}$ and $\M_{G_2}$ to be the quotient of structures on a given compact or asymptotically cylindrical manifold by the natural pullback action by (asymptotically cylindrical) diffeomorphisms. Similarly, by making everything invariant by a circle action, we may define a moduli space $\M_{G_2}^{S^1}$ of $S^1$-invariant torsion-free $G_2$ structures on a product $M \times S^1$. When $n=3$, we obtain
\begin{thma}
The Calabi-Yau moduli space $\M_{SU(3)}$ on the compact or asymptotically cylindrical manifold $M$ is a smooth manifold, and we have a diffeomorphism
\begin{equation}
\M_{SU(3)} \times \R_{>0} \times H^1(M, \R) \cong \M_{G_2}^{S^1}
\end{equation}
Also, $\M_{G_2}^{S^1}$ is locally diffeomorphic to open subsets of the $G_2$ moduli space (again, compact or asymptotically cylindrical) $\M_{G_2}$ on $M \times S^1$. 
\end{thma}
Theorem A is proved as Theorem \ref{mg2s1mfd} and Theorem \ref{maintheorema}. The $\R_{>0}$ and $H^1(M, \R)$ components give the length of the $S^1$ factor and in some sense its angle to the $M$ factor respectively. Note that it follows that $\M_{SU(3)}$ is finite-dimensional. Its dimension is easy to compute from the corresponding $G_2$ results: for instance, in the compact case, we get
\begin{equation}
\dim \M_{SU(3)}(M) = b^3(M \times S^1) - 1 - b^1(M) = b^3(M) + b^2(M) - b^1(M) - 1
\end{equation}

Given that we can glue Calabi-Yau structures (which we prove as Theorem \ref{su3structuregluing}), there is evidently a map from pairs of Calabi-Yau structures on a pair of manifolds to the glued structures. This map need not immediately induce a well-defined map of moduli spaces. However, once Theorem A is established, we can show
\begin{thmb}
The gluing map from the ``moduli space of asymptotically cylindrical $SU(3)$ gluing data" to the moduli space of $SU(3)$ structures on the glued manifold is well-defined, and is a local diffeomorphism.
\end{thmb}
The technical statement of this theorem is Theorem \ref{maintheoremb}, which follows a considerable amount of necessary preliminary work. The theorem says roughly that supposing a Calabi-Yau structure is given by gluing, any sufficiently close deformation of it is given by a gluing of small deformations of the original asymptotically cylindrical Calabi-Yau structures and gluing parameters, and moreover these structures and parameters are unique for any given deformation. The proof is fairly straightforward given Theorem A: using the moduli spaces from Theorem A, we define a moduli space of gluing data, by which we mean structures and necessary parameters, and an analysis of the structure of this moduli space and the gluing map shows that Theorem B follows from the corresponding result for $G_2$. 

Our approach to both of these thus rests on the result analogous to Theorem A at the level of structures.
\begin{thmc}
Suppose $M$ is six-dimensional and is compact or has an end. There is a homeomorphism between (asymptotically cylindrical) $G_2$ structures on $M \times S^1$ which are $S^1$-invariant and triples consisting of an $SU(3)$ structure on $M$, a positive number, and an (asymptotically translation invariant with appropriate limit) one-form on $M$, with respect to the smooth topologies on the spaces of differential forms and hence structures. 
\end{thmc}
The precise statement of this theorem is Theorem \ref{theoremc} below. We should say immediately that the compact case is essentially already known, following from work of Chan \cite[proof of Theorem 3.10]{chan} by allowing slightly more variation. The asymptotically cylindrical case contains some limited originality, but is not much harder. We mention this theorem at this point, despite the fact that it is less significant than Theorems A and B, as it illuminates the fundamental idea: that if we want to glue Calabi-Yau structures on threefolds, we can convert into $G_2$ structures and then convert back again. 

A slightly more subtle version of Theorem C would also show that given an $SU(3)$ structure with small torsion, then we can construct a $G_2$ structure with small torsion, perturb to remove torsion there, and then return to a Calabi-Yau structure. This idea was used by Chan\cite{chan} in the context of desingularising conical Calabi-Yau orbifolds. He showed that a ``nearly Calabi-Yau structure" can be deformed to a Calabi-Yau structure by passing through a $G_2$ structure on $M \times S^1$. We do not use Chan's concept of a nearly Calabi-Yau structure, but a similar analysis would apply to a nearly Calabi-Yau structure obtained by cutting off and gluing the $SU(3)$ structures $(\Omega_1, \omega_1)$, $(\Omega_2, \omega_2)$ directly, proving our Theorem \ref{su3structuregluing}.

The idea was also used by Doi and Yotsutani\cite{doiyotsutani} for gluing asymptotically cylindrical $SU(3)$ structures, in the simply connected case. In distinction to Chan's approach of gluing $SU(3)$ structures approximately first and then passing to $G_2$, they crossed with $S^1$ first and then glued the resulting $G_2$ structures using Kovalev's result, with no guarantee that the answer is well-adapted to this product. Then, using simple connectedness and the Cheeger-Gromoll splitting theorem, as well as classification theory for six-manifolds, they argued that the (simply connected) glued six-manifold is diffeomorphic to a six-dimensional factor of the universal cover of the glued seven-manifold; since the universal cover is a Riemannian product and also admits a $G_2$ structure, this factor must admit an $SU(3)$ structure. Our analysis simply removes the need for the classification theory and splitting (and consequently the necessary assumption of simple connectedness): in the simply connected case we can identify the universal cover concretely. 

The relation between $SU(3)$ and $G_2$ structures is also used in physics, for instance, by de la Ossa, Larfors, and Svanes \cite{delaossalarforssvanes}. By putting a $G_2$ structure on the product of their $SU(3)$ manifold and a half-line, they constrain the evolution of an $SU(3)$ structure, with torsion, (according to various physical hypotheses), and so study paths in the ``moduli space" of $SU(3)$ structures with torsion. They showed for instance that if the $G_2$ structure does not satisfy specific torsion conditions, then various components of the torsion of $SU(3)$ structures may become nonzero along these paths even if they were not nonzero to start with. 

None of these authors consider, however, whether similar arguments and Nordstr\"om’s theorem in the $G_2$ case \cite{nordstromgluing} enable us to say anything about the action of this gluing on the moduli space. The first step towards such a result is to set up the Calabi-Yau moduli space, and establish Theorem A. 

In the compact case, a moduli space of Calabi-Yau structures can be constructed using Hitchin's work\cite{hitchinthreeforms} on three-forms on six-dimensional compact manifolds. He showed that if $\Omega$ was a holomorphic volume form, then a small perturbation of $\Re \Omega$ was also the real part of a holomorphic volume form, and showed that locally each cohomology class of three-forms only contained one such real part. To get the full moduli space result, this has to be combined with some work on deformations of the K\"ahler form: this was carried out by Nordstr\"om in \cite{nordstromacyldeformations} to provide boundary values for deformations of asymptotically cylindrical $G_2$ manifolds. We do not use this approach to the Calabi-Yau moduli space directly, except when considering how this deformation theory of $G_2$ manifolds passes to the $S^1$-invariant setting. If we set up the Calabi-Yau moduli space in this way, it is relatively straightforward to show the remaining part of Theorem A. 

On the other hand, the asymptotically cylindrical Calabi-Yau moduli space has not been studied in the sense we will use. Various similar objects have been studied. 

Most of these studies have used the notion of a logarithmic deformation of a complex structure. By a result of Haskins--Hein--Nordstr\"om\cite[Theorem C]{haskinsheinnordstrom}, any asymptotically cylindrical Calabi-Yau manifold of dimension greater than 2 is given by removing a divisor from a suitable orbifold. Kawamata \cite{kawamatadeformations} studied deformations of complex manifolds compactifiable in a similar sense; he called those corresponding to a fixed compactification logarithmic deformations. 

Kovalev\cite{kovalevdeformations} studied deformations of Ricci-flat K\"ahler metrics. Given an asymptotically cylindrical Calabi-Yau manifold, he showed that we can define an orbifold of Ricci-flat metrics around its metric, and that locally all such metrics with the same limit are K\"ahler for some logarithmic deformation of the complex structure. However, \cite{kovalevdeformations} did not consider how the complex structure varied in general: in particular, it did not consider variations of complex structure not leading to a change in the metric.

More recently, Conlon, Mazzeo and Rochon\cite{cmraccy} combined \cite[Theorem C]{haskinsheinnordstrom} with \cite{kawamatadeformations} to show (their Theorem C) that the complex deformation theory of asymptotically cylindrical Calabi-Yau manifolds is unobstructed. Conversely to Kovalev, they did not explicitly consider deformations of the K\"ahler form, but observed (Lemma 9.1) that if the first Betti number of the compactifying orbifold is zero, then K\"ahler classes remain K\"ahler under such deformations of complex structure. Combining this with \cite[Theorem D]{haskinsheinnordstrom}, which says in particular that for any K\"ahler class we can find an asymptotically cylindrical Calabi-Yau metric, we essentially expect that every K\"ahler class on the orbifold gives a deformation of the metric corresponding to this complex deformation. 

To get a result comparable to the one we use, we would also have to combine this with \cite{kovalevdeformations} so that we can vary the K\"ahler class without the complex structure and both the K\"ahler class and the complex structure simultaneously. We would also need to extend \cite[Lemma 9.1]{cmraccy} to the case where the first Betti number of the compactifying orbifold is nonzero. 

Nevertheless, the $G_2$ moduli space has already been studied in the asymptotically cylindrical case. The result we use is Theorem \ref{g2slicetheorem}: in the compact case it is due to Bryant and Harvey, but the first published proof was provided by Joyce; in adapting it for the results we need, we follow the simplified proof of Hitchin \cite{hitchinthreeforms} and its elaboration for the asymptotically cylindrical case by Nordstr\"om \cite{nordstromacyldeformations}. Ebin \cite{ebin} also constructed a moduli space of general metrics, and some ideas from his paper have become standard. 

This paper is organised as follows.

In section \ref{sec:preliminaries}, we recall various preliminary definitions. In subsection \ref{ssec:acyl}, we define what it means for a manifold to be asymptotically cylindrical and make various suitably adapted definitions (of diffeomorphisms, isotopy, and so forth). This material is all review from various sources ranging from Lockhart--McOwen \cite{lockhartmcowen} and to some extent Maz$'$ja--Plamenevski\u\i \cite{mazyaplamenevskij} to  Nordstr\"om\cite{nordstromacyldeformations}, except the discussion of asymptotically cylindrical isotopy, which is perhaps implied in some of these sources. Subsection \ref{ssec:sunsetup} gives a definition of the notion of $SU(n)$ (for general $n$) structure, mostly following Hitchin \cite{hitchinthreeforms}. Finally in subsection \ref{ssec:g2setup}, we introduce the notion of a $G_2$ structure following Joyce \cite{joycebook}. 

In section \ref{sec:g2su3}, we then explain Chan's work\cite{chan} on the connection between $S^1$-invariant $G_2$ and $SU(3)$ structures, and how to extend this work to the asymptotically cylindrical case. This proves Theorem C. 

In section \ref{sec:modulispace}, we begin to consider moduli spaces. We define a moduli space of $S^1$-invariant $G_2$ structures in such a way that Chan's arguments extend to give a bijection between it and the product required by Theorem A. By analysing the arguments of Hitchin \cite{hitchinthreeforms} and Nordstr\"om \cite{nordstromacyldeformations} showing that the $G_2$ moduli space is smooth, we then show that this moduli space is smooth; in fact, it is locally diffeomorphic to an open subset of the full $G_2$ moduli space. Using these together, we then argue that the $SU(3)$ moduli space can be assumed to be smooth with a suitably natural smooth structure. 

Finally, section \ref{sec:gluing} falls into five subsections. Firstly, in subsection \ref{ssec:structuregluing}, we show that there is (indeed, are potentially multiple) sensible gluing maps defined on matching $SU(3)$ structures (essentially using the $G_2$ gluing map of \cite{kovalevtwistedconnected}). In subsection \ref{ssec:gluingtomod} we consider how they differ, or not, as maps into moduli space. In subsection \ref{ssec:gluingmodspacesetup} we define moduli spaces of gluing data, following \cite{nordstromgluing}, and show the analogous result to  Theorem A on the relation between these in the $G_2$ and Calabi-Yau setting; in subsection \ref{ssec:removingtildes} we define the gluing map and check it is well-defined. Finally in subsection \ref{ssec:localdiffeoofmodspace}, we combine improved versions of the statements from the first subsection on how our family of gluing maps behaves with the local diffeomorphism of $G_2$ moduli spaces proved in \cite{nordstromgluing} to prove that each gluing map defines a local diffeomorphism of $SU(3)$ moduli spaces also. 

\textbf{Notation.} All cohomology groups are to be understood with real coefficients as de Rham cohomology groups. 

\textbf{Acknowledgement.} I am indebted to Alexei Kovalev for much helpful advice and guidance. 
\section{Preliminaries} 
\label{sec:preliminaries}
In this section, we set up preliminary definitions. In subsection \ref{ssec:acyl}, we define various necessary asymptotically cylindrical objects, which will be used throughout. We then turn to $SU(n)$ and $G_2$ structures. We will follow the definitions and in large part the approach of Hitchin \cite[section 2]{hitchinmodulispace} when dealing with $SU(n)$ in subsection \ref{ssec:sunsetup}, and \cite[section 10.1]{joycebook} when dealing with $G_2$ in subsection \ref{ssec:g2setup}. 
\subsection{Asymptotically cylindrical manifolds}
\label{ssec:acyl}
We first define asymptotically cylindrical structures of various kinds. 
\begin{defin}[cf. \cite{lockhartmcowen, mazyaplamenevskij, giraffe}]
\label{basicacyldefs}
An oriented Riemannian manifold $(M, g)$ is said to have an end if it can be decomposed as a smooth manifold into a compact manifold $M^\cpt$, with compact (oriented) boundary $N$, and the product manifold $N \times [0, \infty)$, with the obvious identification of $N = \partial M^\cpt$ with $N \times\{0\}$. Given such a manifold we can find a global function $t$ that is the coordinate on $[0, \infty)$ on $(1, \infty)$ and zero on $M^\cpt$; throughout, we shall use $t$ for this function. 

$M$ is then said to be asymptotically cylindrical with rate $\delta>0$ if there is some metric $g_N$ on $N$ and constants $C_r$, such that 
\begin{equation}
\label{eq:expdecayestimate}
|D^r (g|_{N \times [0, \infty)} - g_N -dt^2)| < C_re^{-\delta t}
\end{equation}
for all $r=0,1,\ldots$,  $D$ is the Levi-Civita connection induced by $g$, and $|\cdot|$ is the metric induced by $g$ on the appropriate space of tensors. $(M, g)$ is said to be asymptotically cylindrical if there is some $\delta>0$ such that it is asymptotically cylindrical with rate $\delta$. 

Given a bundle $E$ associated to the tangent bundle over $M$, a section $\tilde \alpha$ of $E|_N$ extends to a section of $E|_{N \times [0, \infty)}$ by extending parallel in $t$. A section $\alpha$ of $E$ is said to be asymptotically translation-invariant with rate $0<\delta'<\delta$ if there is a section $\tilde \alpha$ of $E|_N$ and \eqref{eq:expdecayestimate} holds (with $\delta'$) for $|D^r (\alpha -\tilde\alpha)|$ for $t>T$. Throughout, given a section $\alpha$ of such a bundle, $\tilde \alpha$ will be its limit. Note that $dt^2 + g|_N$ is not asymptotically translation invariant for any rate greater than $\delta$, which is why we restrict to $\delta'<\delta$. 
\end{defin}
\begin{rmk}
If $(M, g_1)$ is asymptotically cylindrical with rate $\delta$ and $\alpha$ is asymptotically translation invariant with rate $\delta'<\delta$ and $g_2$ is another asymptotically cylindrical Riemannian metric on $M$ with rate $\epsilon>\delta'$, then $\alpha$ is also asymptotically translation invariant on $(M, g_2)$ with rate at least $\delta'$. Consequently, we may refer to ``asymptotically translation-invariant" tensors without specifying an asymptotically cylindrical metric. 
\end{rmk}
To simplify statements, all fields on asymptotically cylindrical manifolds will be taken to be asymptotically translation invariant. 

Given our smooth function $t$ on a manifold $M$ with an end, we can define cutoff functions. Let 
$\psi_T: \R \to \R$ be a smooth function with
\begin{equation}\label{eq:cutoff}
\psi_T(t) = \begin{cases} 1 &\text{ for } t \geq T-1 \\ 0 &\text{ for } t \leq T-2 \end{cases}
\end{equation}
Then $\psi_T$ evidently extends to $M$. We let $\psi = \psi_2$. 

We have a $C^k$ topology on asymptotically translation invariant forms with a given rate $\delta'$. This is called the extended weighted topology in \cite[section 3]{kovalevdeformations}.  See also the extended $L^2$ spaces of \cite[p.58ff.]{aps1}, although both of these work with extensions of Sobolev spaces whereas we work with H\"older spaces. Specifically, we define a $C^k_{\delta}$-norm on a subset of asymptotically translation invariant fields $\alpha$ by
\begin{defin}
\label{defcktopology}
\begin{equation}
\label{eq:ckdnorm}
\|\alpha\|_{C^k_\delta} = \|(1-\psi)\alpha + \psi e^{\delta t}(\alpha -\tilde \alpha)\|_{C^k(E, g)} + \|\tilde \alpha\|_{C^k(E|_N, \tilde g)}
\end{equation}
\end{defin}
The topology induced by \eqref{eq:ckdnorm} is called the topology with weight $\delta$. Note that if, for some given $\alpha$, there is some fixed $\delta$ so that these norms are finite for all $k=0, 1, \ldots$, $\alpha$ must be asymptotically translation invariant with any rate greater than $\delta$. On the other hand, if $\alpha$ is asymptotically translation invariant, with rate greater than $\delta$, then all these norms are finite. In practice, the major results are proved for each individual $\delta$, and if necessary we then combine the different $\delta$s. 

In the same way we have a H\"older $C^{k, \alpha}_\delta$ topology, and by taking the inverse limit we also have a $C^\infty_\delta$ topology. These topologies (as opposed to the norms) depend only on the decay rate of the metric, since by compactness of $N$ and $M^\cpt$ all metrics with the same decay rate are Lipschitz equivalent. 

In section \ref{sec:modulispace} (e.g. Definition \ref{defin:mg2s1}), we will also need the notion of an asymptotically cylindrical diffeomorphism: general diffeomorphisms obviously do not have to preserve the cylindrical asymptotic, and so do not act by pullback on asymptotically cylindrical metrics. Following Nordstr\"om \cite{nordstromacyldeformations}, we make
\begin{defin}[{\hspace{1sp}\cite[Definition 2.19]{nordstromacyldeformations}}]
\label{acyldiffeo}
A diffeomorphism $\Phi$ of the asymptotically cylindrical manifold $(M, g)$ is asymptotically cylindrical if there is a diffeomorphism $\tilde \Phi$ of $N$ and a parameter $L \in \R$ such that
\begin{equation}
\Phi(n, t) \to (\tilde \Phi(n), t+L)
\end{equation}
exponentially, meaning that on restriction to $N \times (T, \infty)$ for some large $T$, we have $\Phi = \exp V \circ (\tilde \Phi(n), t+L)$ for some vector field $V$ on $M$ decaying exponentially with all derivatives.
\end{defin}
The pullback by an asymptotically cylindrical diffeomorphism of an asymptotically cylindrical metric is asymptotically cylindrical.  Also, whether a diffeomorphism $\Phi$ of $M$ is asymptotically cylindrical does not depend essentially on the asymptotically cylindrical metric $g$ (compare Proposition 6.22 of \cite{nordstromacyldeformations}). Finally we note that the asymptotically cylindrical diffeomorphisms form a group. 

For section \ref{sec:modulispace}, we would like to restrict to the identity component of asymptotically cylindrical diffeomorphisms. We thus need to define a topology. The topology we shall define is essentially used, although not in so many words, by Kovalev \cite[p.148]{kovalevdeformations} in his choice of norm on the generating vector fields, and Nordstr\"om implicitly (some such assumption is required to get the bottom of \cite[p.336]{nordstromacyldeformations}); we have stated it explicitly to avoid confusion about what ``isotopic to the identity" means. 

\begin{defin}
\label{defin:Ddeltatopology}
Let $(M, g)$ be an asymptotically cylindrical manifold. Fix $\delta > 0$, smaller than the decay rate of $g$, and consider the subset of asymptotically cylindrical diffeomorphisms that decay at rate at least $\delta$ with respect to this metric, $\D^g_\delta$. We define a topology on $\D^g_\delta$ by giving neighbourhoods of the identity.  

Let $\Phi$ be an asymptotically cylindrical diffeomorphism. By definition, $\Phi = \exp V \circ (\tilde \Phi(n), t+L)$ far enough along the end. If $(\tilde \Phi, L)$ is not close to the identity, then we do not include $\Phi$ in the neighbourhood. Consequently, we may suppose that $\tilde \Phi = \exp W$, where this exponential map is taken with respect to $\tilde g$, so that $\Phi = \exp_g V \circ \exp_{\tilde g} (W + L\ddt)$. Then $V+W+L\ddt$ is an asymptotically translation invariant vector field. That is, we have identified a subset of $\D^g_\delta$ containing the identity all of whose elements define an asymptotically translation invariant vector field, and where the identity defines the zero vector field. 

To define our neighbourhoods of the identity, we then take the neighbourhoods of zero with respect to the extended weighted topologies described in Definition \ref{defcktopology} on the corresponding asymptotically translation invariant vector fields.
\end{defin}

We may now define a topology on the set $\D$ of all asymptotically cylindrical deformations.
\begin{defin}
\label{defin:ascyldifftopol}
$U \subset \D$ is open if and only if $U \cap \D^g_\delta$ is open in the topology of Definition \ref{defin:Ddeltatopology} for every $\delta>0$. 
\end{defin}

This topology is also independent of our choice of metric $g$. As the map $\Phi \mapsto \tilde\Phi$ is continuous, this definition also automatically gives us a well-defined map from a quotient by $\Diff_0$ on $M$ to a quotient by $\Diff_0$ on $N$. 

We now make
\begin{defin}
\label{defin:acylisotopy}
The asymptotically cylindrical diffeomorphism $\Phi$ is asymptotically cylindrically isotopic to the identity if it lies in the identity component $\Diff_0$ of $\mathcal D$. For simplicity, if $M$ is an asymptotically cylindrical manifold, then if we say ``$\Phi$ is a diffeomorphism of $M$ isotopic to the identity", we shall mean that $\Phi$ is an asymptotically cylindrical diffeomorphism asymptotically cylindrically isotopic to the identity. 

Furthermore, each potential limit $(\tilde \Phi, L)$ of an asymptotically cylindrical diffeomorphism defines a closed subspace of $\D$, as the map to the limit is continuous; we shall say that diffeomorphisms are isotopic with fixed limit if they can be joined by a continuous path in such a subspace (in particular, of course, this implies that they have the same limits). 
\end{defin}
We use isotopy with fixed limit in Definition \ref{defin:spacea}. 

\subsection{\texorpdfstring{$SU(n)$}{SU(n)} structures}
\label{ssec:sunsetup}
We now define $SU(n)$ structures, following Hitchin \cite{hitchinmodulispace}. 
\begin{defin}
\label{defin:sunstructure}
Let $M$ be a $2n$-dimensional manifold. An $SU(n)$ structure on $M$ is induced by a pair $(\Omega, \omega)$ where $\Omega$ is a smooth complex $n$-form on $M$ and $\omega$ is a smooth real $2$-form on $M$ such that at every point $p$ of $M$:
\begin{enumerate}[i)]
\item$\Omega_p = \beta_1 \wedge \cdots \wedge\beta_n$ for some $\beta_i \in T^*_p M \otimes \C$
\item$\Omega_p \wedge \bar \Omega_p \neq 0$
\item$\Omega\wedge\bar\Omega = \frac{(-2)^ni^{n^2}}{n!} \omega^n$
\item$\omega \wedge \Omega = 0$
\item$\omega_p(v, Iv) > 0$ for every $v \in T_pM$ where $I$ is as in the following proposition. 
\end{enumerate}
\end{defin}
The standard definition is actually that an $SU(n)$ structure is a principal $SU(n)$-subbundle of the frame bundle of $M$ (for instance, see \cite[section 2.6]{joycebook}). A pair of forms as in Definition \ref{defin:sunstructure} determines such a subbundle by taking the stabiliser at each point, but evidently the same subbundle could be obtained from different pairs of forms. Nevertheless, we shall abuse notation and say that $(\Omega, \omega)$ is an $SU(n)$ structure. If the two forms have stabiliser $SU(n)$, they determine a complex structure and a hermitian metric. 
\begin{prop}[{\hspace{1sp}\cite[section 2]{hitchinmodulispace}}]
\label{suninduced}
Suppose that $M$ is a $2n$-dimensional manifold and $(\Omega, \omega)$ is an $SU(n)$ structure on it. Then there is a unique almost complex structure $I$ on $M$ with respect to which $\Omega$ is an $(n, 0)$-form and, with respect to $I$, $\omega$ is the fundamental form of an hermitian metric $g$.
\end{prop}
\begin{rmk}
The scaling condition iii) is not used in the proof of this result: it is used for torsion considerations below. In fact, in \cite{joycecones1, joycecones2, joycecones3, joycecones4, joycecones5},  Joyce calls a structure in which iii) may fail an almost Calabi-Yau structure. Our analysis of the relation between $SU(3)$ structures and $G_2$ structures extends to a relation between almost Calabi-Yau structures and $G_2$ structures without much additional work -- we give more details in Remarks \ref{rmk:almostcypointwise} and \ref{rmk:almostcytorsion}.
\end{rmk}
However, we would like $I$ to be a complex structure and $\omega$ a \Kahler\ form. To achieve this, we add further conditions, making
\begin{defin}
\label{defin:suntorsionfreeness}
Suppose that $(\Omega, \omega)$ is an $SU(n)$ structure. It is said to be torsion-free, or a Calabi-Yau structure, if $d\Omega$ and $d\omega$ are both zero.
\end{defin}
Note that Definition \ref{defin:suntorsionfreeness} is equivalent to the vanishing of the intrinsic torsion of the principal $SU(n)$-subbundle.

We indeed have
\begin{prop}
\label{prop:suntorsionfreeness}
Suppose that $(\Omega, \omega)$ is a torsion-free $SU(n)$ structure. The induced almost complex structure $I$ is a complex structure, and the hermitian metric induced from $\omega$ is K\"ahler and Ricci-flat.
\end{prop}

We now have to combine Definition \ref{defin:sunstructure} with Definition \ref{basicacyldefs} to define an asymptotically cylindrical $SU(n)$ structures.
\begin{defin}
\label{defin:acylsunstructure}
An $SU(n)$ structure on a manifold $M$ with an end is said to be asymptotically cylindrical if the induced metric $g$ is asymptotically cylindrical and, with respect to $g$, $\Omega$ and $\omega$ are asymptotically translation invariant.
\end{defin}
Note that if $(\Omega, \omega)$ are asymptotically translation invariant forms for some cylindrical metric, it need not be the case that $\frac\partial{\partial t}$ is orthogonal to $N$. Conversely, any almost complex structure admitting a non vanishing $(n, 0)$ form and such that a cylindrical metric is hermitian can be combined with that cylindrical metric to yield an $SU(n)$ structure inducing a cylindrical metric. This almost complex structure need not be asymptotically translation invariant, so that the fact that the induced metric $g$ is asymptotically cylindrical does not imply that $\Omega$ and $\omega$ are asymptotically translation invariant. That is, the two conditions of Definition \ref{defin:acylsunstructure} are independent of each other. 
\begin{numrmk}
\label{sunacylconndxs}
If we have a torsion-free asymptotically cylindrical $SU(n)$ structure on $M$, then it induces a Ricci-flat metric. If $M$ has disconnected cross-section, it admits a line between different components of the cross-section; by the Cheeger--Gromoll splitting theorem \cite{cheegergromoll1} it is then a product cylinder $N \times \R$, and so not especially interesting. We may thus assume $N$ is connected wherever required. It is only explicitly required in Lemma \ref{ccfcohom}, but that lemma is often used after its proof. 
\end{numrmk}

We will also find it useful for Proposition \ref{spaceanice} to know that diffeomorphisms of a Calabi-Yau manifold isotopic to the identity are isometries if and only if they are automorphisms of the underlying Calabi-Yau structures. The infinitesimal version of this (that Killing fields are holomorphic vector fields and vice versa) is a special case of \cite[Theorem III.5.2]{kobayashitransformations}. 
\begin{lem}
\label{sunauto}
Suppose $M$ is compact or has an end, that $(\Omega, \omega)$ is an (asymptotically cylindrical) torsion-free $SU(n)$ structure, and that $\Phi \in \Diff_0(M)$. Then $\Phi^* g_{(\Omega, \omega)} = g_{\Omega, \omega}$, \ie $\Phi$ is an isometry, if and only if $\Phi^* \Omega = \Omega$ and $\Phi^* \omega = \omega$, \ie $\Phi$ is an automorphism of the $SU(n)$ structure.
\end{lem}
\begin{proof}
If $\Phi$ is an automorphism, it is clearly an isometry, because the metric is obtained from $(\Omega, \omega)$ in a natural fashion, so commuting with the pullback. Conversely, if $\Phi$ is an isometry, it pulls back $\Omega$ and $\omega$ to parallel forms with respect to the induced metric. If $M$ is compact, it also preserves the cohomology classes $[\Omega]$ and $[\omega]$; since an exact parallel form is zero, it preserves $\Omega$ and $\omega$ and so is an automorphism. If $M$ has an end, then the limit cohomology classes are also preserved, and the differences $\Phi^* \Omega - \Omega$ and $\Phi^* \omega -\omega$ are exponentially decaying parallel forms trivial in cohomology, and hence zero. 
\end{proof}

\subsection{\texorpdfstring{$G_2$}{G₂} structures}
\label{ssec:g2setup}
As a technical tool, we will also use $G_2$ structures on $7$-dimensional manifolds. We define
\begin{defin}
\label{defin:g2str}
A $G_2$ structure on a seven-dimensional manifold $M$ is a smooth three-form such that at every point $p$  there is a basis $e_1, \ldots, e_7$ of $T_p^*M$ such that
\begin{equation}
\label{eq:g2structure}
\begin{split}
\phi_p =& e_1 \wedge e_2 \wedge e_3 + e_1 \wedge e_4 \wedge e_5 + e_1 \wedge e_6 \wedge e_7 + e_2 \wedge e_4 \wedge e_6 \\&- e_2 \wedge e_5 \wedge e_7 - e_3 \wedge e_4 \wedge e_7 - e_3 \wedge e_5 \wedge e_6
\end{split}
\end{equation}
\end{defin}
Every $G_2$ structure induces a metric, by taking the corresponding basis to be orthonormal. Thus, $\phi$ induces a $4$-form $*_\phi \phi$. We use this $4$-form to define torsion-freeness of a $G_2$ structure. 
\begin{defin}
\label{defin:g2tfree}
A $G_2$ structure $\phi$ on $M$ is torsion-free if the forms $\phi$ and $*_\phi\phi$ are both closed. 
\end{defin}
As in the $SU(n)$ case, this is equivalent to torsion-freeness of the principal $G_2$-subbundle given by the stabiliser. Torsion-freeness would more naturally be described by $\phi$ being parallel with respect to the metric it induces; that Definition \ref{defin:g2tfree} implies that is a result of Fern\'andez and Gray\cite{fernandezgray}. Further as in the $SU(n)$ case, Definition \ref{defin:g2tfree} implies that the induced metric is Ricci-flat.

Also as in the $SU(n)$ case, we require a notion of asymptotically cylindrical $G_2$ structure. As in Definition \ref{defin:acylsunstructure}, we make
\begin{defin}
\label{defin:precedesg2auto}
\label{acylg2structure}
A $G_2$ structure on a manifold $M$ with an end  is said to be asymptotically cylindrical if the induced metric is asymptotically cylindrical and, with respect to this metric, $\phi$ is asymptotically translation invariant.
\end{defin}
The fact that torsion-free $G_2$ structures induce Ricci-flat metrics gives the analogue of Remark \ref{sunacylconndxs} in the asymptotically cylindrical $G_2$ setting; the fact they are parallel yields the analogue of Lemma \ref{sunauto} in the compact and asymptotically cylindrical $G_2$ settings, with the same proof. 
\section{\texorpdfstring{$SU(3)$}{SU(3)} structures as \texorpdfstring{$S^1$}{S¹}-invariant \texorpdfstring{$G_2$}{G₂} structures}
\label{sec:g2su3}
We now proceed to the relationship between $SU(3)$ and $G_2$ structures induced by the inclusion $SU(3) \subset G_2$, and prove Theorem C. We closely follow ideas in Chan's analysis of the three-dimensional conical gluing problem, \cite{chan}. Only the details of the asymptotically cylindrical case are original. The beginnings of these ideas can be found in \cite[Proposition 11.1.2]{joycebook}, which corresponds to the easier directions of Propositions \ref{g2su3prop1} and \ref{g2su3prop2}: that if we have a Calabi-Yau structure $(\Omega, \omega)$ on a six-manifold $M$ (without any global conditions), we can induce a torsion-free $G_2$ structure on $M \times \R$ by $\phi = \Re \Omega + d\theta \wedge \omega$, with corresponding metric $g_M + d\theta^2$. Our setup, following Chan, is more general. Because in gluing we introduce a perturbation to the $G_2$ structure which is uncontrolled except for being small, it is not at all clear that the $G_2$ structure will remain in the proper subspace $\{\Re \Omega + d\theta \wedge \omega: (\Omega, \omega) \text{ an $SU(3)$ structure}\}$. For instance, and more geometrically, it is not clear that the perturbed $G_2$ structure will either have $\ddth$ orthogonal to $M$ or have $\ddth$ of length one. We introduce $z$, therefore, as a generalisation of $d\theta$: a $1$-form with nonzero coefficient of $d\theta$. In practice, we shall soon assume it has positive coefficient of $d\theta$, and by the end of this section we will assume that $z = Ld\theta + v$ for a constant $L$ and a closed $1$-form $v$ on $M$. In particular, if $b^1(M) = 0$, we could reduce further when discussing moduli spaces, and we may as well assume in section \ref{sec:modulispace} and the relevant parts of section \ref{sec:gluing} that $z = Ld\theta$.

However, we work in full generality to start with. We begin with the vector space case, where $z$ is just a covector.
\begin{prop}
\label{g2su3ptwise}
Suppose given a six-dimensional vector space $V$, and suppose that $z \in (V \oplus \R)^* \cong V^* \oplus \R$ is complementary to $V^*$. Suppose further that $(\Omega, \omega)$ is an $SU(3)$ structure on $V$ (that is, $(\Omega, \omega) \in \bigwedge^3 V^* \otimes \C \oplus \bigwedge^2 V^*$ satisfying the conditions of Definition \ref{defin:sunstructure}) with associated metric $g_{\Omega, \omega}$. Then the three-form
\begin{equation}
\label{eq:g2fromsu3}
\phi = \Re \Omega + z \wedge \omega
\end{equation}
is a $G_2$ structure on $V \oplus \R$ (that is, $\phi \in \bigwedge^3 (V \oplus \R)^*$ satisfying the condition in Definition \ref{defin:g2str}) with associated metric $z\otimes z + g_{\Omega, \omega}$. Moreover, given a $G_2$ structure $\phi$ on $V \oplus \R$, there exist exactly two possible triples $(z, \Omega, \omega)$ and $(z', \Omega', \omega')$, with $(\Omega, \omega)$ and $(\Omega', \omega')$ $SU(3)$ structures on $V$, such that $\phi$ was obtained from this triple as in \eqref{eq:g2fromsu3}. They satisfy $z' = -z$, $\Omega' = \bar \Omega$, and $\omega' = -\omega$. 
\end{prop}
\begin{proof}
First choose a basis $e_2, e_3, \ldots, e_7$ of $V^*$ so that $(\Omega, \omega)$ is the standard $SU(3)$ structure (this can be done similarly to the proof of Proposition \ref{suninduced}):
\begin{equation}
\Omega = (e_2 + i e_3) \wedge (e_4 + i e_5) \wedge (e_6 + i e_7) \qquad \omega = e_2 \wedge e_3 + e_4 \wedge e_5 + e_6 \wedge e_7
\end{equation}
Then $e_1=z, e_2, \ldots, e_7$ is a basis of $(V\oplus\R)^*$ and $\Re \Omega + z\wedge \omega$ is the $G_2$ structure given by \eqref{eq:g2structure} with respect to this basis; hence, this construction always yields a $G_2$ structure. The metric follows by this construction: the dual basis vector to $z$ is orthogonal to $V$ and length $1$, and $V$ has its original metric.

Conversely, suppose given a $G_2$ structure $\phi$ on $V \oplus \R$. We want to choose a basis so that $\phi$ is the $G_2$ structure given by \eqref{eq:g2structure} and $V^*$ is spanned by $e_2, \ldots, e_7$. First choose any basis such that $\phi$ is given by \eqref{eq:g2structure}. We need an element of $G_2$ mapping $\spn \{e_2, \ldots, e_7\}$ to $V^*$. This is equivalent to taking $e_1$ to $V^*$'s unit normal, and this is possible because $G_2$ is transitive on $S^6$. In this basis, we then have the structure of the previous paragraph.

To show that any $\phi$ is given precisely by these two triples, we begin by noting that if $\Re\Omega + z\wedge\omega$ is a $G_2$ structure then $z$ and $V^*$ are orthogonal with respect to the induced metric and $z$ has length $1$. Therefore, $z$ is a unit normal vector to $V^*$ with respect to the metric of $\phi$ and is determined up to sign. Fix a possible $z$, and then consider $SU(3)$ structures $(\Omega, \omega)$ on $V$ such that $\phi = z \wedge \omega + \Re \Omega$; the fact that $z$ is complementary means $\phi$ and $z$ uniquely determine $\omega$ and so $\Re \Omega$. It is easy to check that given any such $z$,$\omega$ and $\Omega$, 
\begin{equation}
*\phi = \frac12 \omega \wedge \omega + z\wedge\Im \Omega
\end{equation}
and so as $\phi$ uniquely determines $*\phi$ it also uniquely determines $\Im \Omega$.

If we reverse the sign of $z$, this merely reverses the signs of $\Im \Omega$ and $\omega$, and so gives the second triple.
\end{proof}
At this point, we have two options. We can either hold on to this 2:1 correspondence throughout, or we can make a uniform choice. We will do the latter, partly for notational simplicity and partly to guarantee that the set of $z$'s is connected (a similar result will be technically useful later.) Since $(\Omega, \omega) \leftrightarrow (\bar\Omega, -\omega)$ is an isomorphism of $SU(3)$ structures, it has no serious effect on the results.

We shall express this choice as an orientation on $\R$ (and later the corresponding manifold $S^1$); this fixes the sign of $z$ by demanding that its relevant component should be positive with respect to a standard form $d\theta$ on $\R$; equivalently, this is defining which orientation $V$ has as a subspace of $V \oplus \R$.
\begin{numrmk}
\label{rmk:almostcypointwise}
If $(\Omega, \omega)$ is only the restriction to a point of an almost Calabi-Yau structure as in Joyce\cite{joycecones1, joycecones2, joycecones3, joycecones4, joycecones5} (that is we drop the normalisation condition on the relative sizes of $\Omega$ and $\omega$), the forward direction clearly gives a $G_2$ structure, as we may imagine rescaling $z$. However, for this reason there are many choices for the backward direction: we may freely scale $\omega$ and $z$. Thus in this case we may assume that the $d\theta$ component of $z$ has coefficient $\pm 1$, to retain our two options.
\end{numrmk}

It is clear that the correspondence of Proposition \ref{g2su3ptwise} extends to global structures.  We require the notion of a structure being $S^1$-invariant to ensure that each $G_2$ structure arises from a single $SU(3)$ structure.  
\begin{defin}
\label{defin:s1inv}
Let $M$ be a six-dimensional manifold. Consider the product $M \times S^1$, and let $\ddth$ be the vector field corresponding to a global function $\theta$ giving a coordinate on the circle. The diffeomorphism $\Theta$ is given by the flow of $\ddth$ for some time. A differential form $\alpha$ on $M \times S^1$ is said to be $S^1$-invariant if its Lie derivative in the $\ddth$ direction is zero, or equivalently it is preserved by pullback by $\Theta$. Any other tensor is said to be $S^1$-invariant if the same conditions hold (since $\Theta$ is a diffeomorphism we can consider pushforward by its inverse). A map of tensors is $S^1$-equivariant if it commutes with the appropriate pullback and pushforward maps induced by $\Theta$. 
\end{defin}
We shall use $\ddth$, $\Theta$, and the notion of $S^1$-invariance throughout. $S^1$-equivariance is primarily used in subsection \ref{ssec:s1invg2modspacesmooth}. 

We may now state
\begin{prop}
\label{g2su3prop1}
Let $M$ be a six-dimensional manifold admitting an $SU(3)$ structure. Let $z$ be an $S^1$-invariant covector field on $M \times S^1$ that is always complementary to the subbundle $T^*M$ and has positive orientation with respect to the circle, \ie $\int_{\{p\} \times S^1} z > 0$ everywhere. Then the construction of Proposition \ref{g2su3ptwise} yields a $G_2$ structure on $M \times S^1$. 

Conversely, if $M \times S^1$ admits a $G_2$ structure, the structure is constructed as in Proposition \ref{g2su3ptwise} from some unique section of the bundle of $SU(3)$ structures on $TM$ over $M \times S^1$ (that is, a structure on $T_pM$ at each point $(p, \theta)$, but potentially varying with $\theta$) and some unique complementary covector field on $M\times S^1$ with suitable orientation at each point. In particular, if the $G_2$ structure is $S^1$-invariant, then both of these sections are $S^1$-invariant, so reduce to an $SU(3)$ structure on $M$ and an $S^1$-invariant complementary covector field with suitable orientation as in the previous paragraph.

Moreover, the maps $(z, \Omega, \omega) \mapsto \phi$ and $\phi \mapsto (z, \Omega, \omega)$ are smooth maps of Fr\'echet spaces.
\end{prop}
\begin{proof}
Proposition \ref{g2su3ptwise} proves the first two paragraphs pointwise; we have to show that the resulting sections are smooth, and the final paragraph. We prove the final paragraph in proving smoothness in the first and second. For the first, if $z$, $\omega$ and $\Omega$ are smooth sections of the relevant bundles then so is $\phi = \Re \Omega + z \wedge \omega$, and as this is multilinear it is clearly a smooth function of $(z, \Omega, \omega)$. 

For the second paragraph, we show directly that $z$, $\Omega$, and $\omega$ are smooth functions of $\phi$, and so in particular are themselves smooth if $\phi$ is. We begin by showing that $z$ is smooth. We observe that $ z= \frac{(\ddth)^\flat}{|\ddth|}$ at all points of $M\times S^1$, where $\theta$ is a positively oriented coordinate (meaning $d\theta$ is positive with respect to our choice of orientation).

Indeed, if $u \in T^*_pM \subset T^*_{(p, \theta)}(M \times S^1)$ we have
\begin{equation}
\left\la u, \left(\ddth\right)^\flat\right\ra = u \left(\ddth\right) = 0
\end{equation}
so $(\ddth)^\flat$ is indeed orthogonal to $T_p^*M$, and $\frac{(\ddth)^\flat}{|\ddth|}$ is clearly unit. Since $(\ddth)^\flat(\ddth) = |\ddth|^2>0 $, the orientation is positive, and so $z$ is indeed $\frac{(\ddth)^\flat}{|\ddth|}$. 

Note that although $\ddth$ is independent of the coordinates on $M$, the metric need not be, so $z$ depends on our position on $M$. We have $(\ddth)^\flat = \iota_\ddth g$. The interior product is linear and continuous; the map $g\mapsto |\ddth|$ is clearly smooth, as the square root of the smooth function $g\mapsto g(\ddth, \ddth)$. It is therefore enough, to prove $z$ is smooth, to prove that
\begin{equation}
\phi \mapsto g_\phi
\end{equation}
is smooth. This essentially follows by the computation in Hitchin \cite{hitchinthreeforms}: both
\begin{equation}
\label{eq:bphi}
B_\phi: (u, v) \mapsto \iota_u \phi \wedge \iota_v \phi \wedge \phi
\end{equation}
and $K_\phi$, the reinterpretation of the bilinear form $B_\phi$ as an endomorphism, are smooth functions of $\phi$. Similarly, so are the determinant of $K_\phi$ and \begin{equation}g_\phi =(\det K_\phi)^{-\frac{1}9} B_\phi\end{equation} using compactness and asymptotic cylindricality to ensure that $\det K_\phi$ is bounded away from zero.

Then the division of $\phi$ and $* \phi$ into ``the $z$ part" and ``the other part" is smooth, because it is linear and continuous; it follows that $\Omega$ and $\omega$ are smooth. 
\end{proof}
We also have to check that the correspondence of Propositions \ref{g2su3ptwise} and \ref{g2su3prop1} respects the notions of asymptotically cylindrical structure that we have defined.
\begin{prop}
\label{g2su3acyl}
Suppose that $M$ is a six-dimensional smooth manifold with an end. Let $(\Omega, \omega)$ be an $SU(3)$ structure on $M$ and $z$ a complementary covector field in the sense of Proposition \ref{g2su3prop1} (appropriately oriented). Suppose that $\phi$ is the corresponding $S^1$-invariant $G_2$ structure on $M \times S^1$. 
Then $\phi$ is asymptotically cylindrical if and only if $(\Omega, \omega)$ is asymptotically cylindrical, $z$ is asymptotically translation-invariant, and $z(\frac{\partial}{\partial t}) \to 0$ exponentially uniformly on $N$. 
\end{prop}
\begin{proof}
It is clear that if $\Omega$, $\omega$, and $z$ are asymptotically translation invariant, then so too is $\phi$. By Proposition \ref{g2su3ptwise}, the induced asymptotically translation invariant metric has limit $\tilde g_{\Omega, \omega} + \tilde z \otimes \tilde z = g_N + dt \otimes dt + \tilde z \otimes \tilde z$ (again we use $\tilde\cdot$ to denote limit). We thus have to show that $\tilde z \otimes \tilde z$ can be taken as a form on only $N \times S^1$. As $z(\frac\partial{\partial t}) \to 0$ , $\tilde z$ has no $dt$ component and this is indeed the case.

For the converse, we observe that both the asymptotically cylindrical $G_2$ structure $\phi$ and its limit $\tilde \phi$ must split in the usual way and therefore we have
\begin{equation}
\phi = \Re \Omega + z\wedge\omega \to \Re \tilde\Omega + \tilde z \wedge \tilde \omega = \tilde \phi
\end{equation}
with respect to the metric induced by either. Since $\phi \to \tilde \phi$, exponentially with all derivatives, and the map $\phi \mapsto (z, \Omega, \omega)$ is a continuous map of Fr\'echet spaces with implicit constants (in continuity arguments) bounded since $\phi$ is asymptotically translation invariant, we must have $z \to \tilde z$ and so on too. So $z$, $\Omega$, and $\omega$ are asymptotically translation invariant. 

Furthermore, 
\begin{equation}
\label{eq:zofddt}
z\left(\frac\partial{\partial t}\right) = g\left(\ddth, \frac\partial{\partial t}\right)\bigg/\left|\ddth\right|
\end{equation}
and since $\phi$ is asymptotically cylindrical the right hand side of \eqref{eq:zofddt} tends to zero uniformly in $N$ and exponentially in $t$. Thus $\tilde z\otimes \tilde z$ has no $dt \otimes dt$ component, and as in the first paragraph $(\Omega, \omega)$ must be asymptotically cylindrical.
\end{proof}

We now find a condition on $z$ which combined with $(\Omega, \omega)$ being torsion-free implies the $G_2$ structure $\Re \Omega + z \wedge \omega$ is torsion-free. Because torsion-freeness, as a differential equation, is a global condition, we restrict to $M$ compact or asymptotically cylindrical.
\begin{prop}
\label{g2su3prop2}
Suppose that $M$ is a compact six-dimensional smooth manifold. Let $(\Omega, \omega)$ be an $SU(3)$ structure on $M$ and $z$ be a covector field on the product $M\times S^1$ complementary to $T^* M$ with positive orientation. Suppose that $\phi$ is the corresponding $S^1$-invariant $G_2$ structure on $M\times S^1$. Then $\phi$ is torsion-free if and only if $z$ is closed and $(\Omega, \omega)$ is torsion-free.

The same holds if $(\Omega, \omega)$ is an asymptotically cylindrical $SU(3)$ structure and $z$ an asymptotically translation invariant complementary covector field, with $\tilde z(\ddt) = 0$, so that $\phi$ is an asymptotically cylindrical $G_2$ structure. 
\end{prop}
\begin{proof}
Firstly, given a torsion-free $SU(3)$ structure and $z$ closed, we have
\begin{equation}
\phi = \Re \Omega + z \wedge \omega
\qquad
*\phi = \frac12 \omega \wedge \omega + z \wedge \Im \Omega
\end{equation}
Since $d$ is a real operator, $\bar\Omega$ is closed, and so both $\phi$ and $*\phi$ are closed.

Conversely, if $\phi$ is a torsion-free $S^1$-invariant $G_2$ structure on $M \times S^1$, we begin by considering the covector field $z$. Let $\ddth$ be the vector field on $S^1$ induced by a standard coordinate $\theta$ (positively orientated and inducing the rotation we use for ``$S^1$-invariance").

Since $\phi$ is $S^1$-invariant, $\ddth$ is a Killing field. A Bochner argument (\eg \cite[Theorem 1.84]{besse}) shows that Killing fields on compact Ricci-flat manifolds are parallel, and we know that the metric associated to the torsion-free $G_2$ structure $\phi$ is Ricci-flat. Consequently, if $M$ is compact, $\ddth$ is parallel. We want to show that $\ddth$ is also parallel in the asymptotically cylindrical case. $\tilde \phi$ defines a translation-invariant $G_2$ structure on the limit $N \times \R \times S^1$. $\ddth$ is a translation-invariant Killing vector field on $N \times \R \times S^1$, and thus we may imagine we work on $N \times S^1 \times S^1$ to deduce that $\ddth$ is parallel with respect to the limit metric. Now we know that $\nabla\ddth$ decays, we may do the integration by parts required by the Bochner argument and deduce that $\ddth$ is parallel. 

Hence, $z = \frac{(\ddth)^\flat}{|\ddth|}$ is parallel and in particular closed.

Then the torsion-freeness of $\phi$ yields from the formulae for $\phi$ and $*_\phi \phi$ in terms of $\omega$ and $\Omega$
\begin{equation}
0 = z \wedge d\omega + d \Re \Omega \qquad 0 = \omega \wedge d\omega + z \wedge d\Im \Omega
\end{equation}
Since $d\omega$, $d\Re \Omega$ and $d\Im \Omega$ are all forms on $M$, it follows since $z$ is complementary to the subbundle of such forms that they are all zero, and so $(\Omega, \omega)$ is torsion-free.
\end{proof}
\begin{numrmk}
\label{rmk:almostcytorsion}
In the almost Calabi-Yau case of Joyce \cite{joycecones1, joycecones2, joycecones3, joycecones4, joycecones5} the torsion-freeness conditions are much weaker: it is only required that $d\omega = 0$. Using the $S^1$-invariance of $\phi$ and Cartan's magic formula, this is equivalent to $\iota_{\ddth} d\phi = 0$. Thus almost Calabi-Yau structures on $M$ are identified, up to a choice of covector field $z$, with $S^1$-invariant $G_2$ structures $\phi$ on $M \times S^1$ satisfying $\iota_{\ddth} d\phi = 0$. 
\end{numrmk}
We now introduce some terminology to simplify the rest of the paper. 
\begin{defin}
\label{defin:twistings}
Let $M$ be a six-dimensional manifold. A closed $S^1$-invariant covector field $z = Ld\theta + v$ for which $L > 0$ is called a twisting. If $M$ is an asymptotically cylindrical manifold, we require also that the limit $\tilde z$ satisfies $\tilde z(\ddt) = 0$. 
\end{defin}

Combining all of this section, and restricting to the torsion-free case, we have the critical Theorem C:
\begin{thm}
\label{g2su3alltold}
\label{theoremc}
If $M$ is a compact six-manifold , there is a homeomorphism
\begin{equation*}
\begin{split}
&\{\text{torsion-free $S^1$-invariant $G_2$ structures on $M \times S^1$}\} \\\leftrightarrow& \{\text{torsion-free $SU(3)$ structures on $M$}\} \times \R_{>0} \times \{\text{closed $1$-forms on $M$}\}
\end{split}
\end{equation*}

If $M$ is a six-manifold with an end, there is a homeomorphism
\begin{equation*}
\begin{split}
&\{\text{torsion-free asymptotically cylindrical $S^1$-invariant $G_2$ structures on $M \times S^1$}\} \\\leftrightarrow& \{\text{torsion-free asymptotically cylindrical $SU(3)$ structures on $M$}\} \times \R_{>0} \\&\times \left\{\text{closed asymptotically translation invariant $1$-forms $v$ on $M$ with $\tilde v\left(\ddt\right) = 0$}\right\}
\end{split}
\end{equation*}
In both cases, this homeomorphism is defined by the map 
\begin{equation}
((\Omega, \omega), L, v) \leftrightarrow \Re \Omega + (L d\theta + v) \wedge \omega
\end{equation}
\end{thm}
\begin{rmk}
The whole section applies equally if the one-dimensional factor is a line instead of a circle, because by invariance we can join the ends to form a circle. Thus the same argument applies on the end of an asymptotically cylindrical $G_2$ manifold, for instance, but asymptotic cylindricality means $v$ that $L$ must be one and $v$ must be zero, or equivalently that the closed $1$-form $z$ must be $dt$.
\end{rmk}
\section{Moduli spaces}
\label{sec:modulispace}
For the remainder of the paper, we shall assume that all $SU(n)$ and $G_2$ structures are torsion-free unless specifically stated otherwise. 

We now want to push the relationship between torsion-free $SU(3)$ structures and torsion-free $G_2$ structures discussed in section \ref{sec:g2su3} and culminating there in Theorem \ref{g2su3alltold} (Theorem C) slightly further. In this section, we define moduli spaces and prove  Theorem A (Theorem \ref{maintheorema}) on how the $SU(3)$ moduli space relates to the $G_2$ moduli space. The section falls into three parts. In subsection \ref{ssec:s1invg2modspacesetup}, we set up a moduli space of $S^1$-invariant torsion-free $G_2$ structures (Definition \ref{defin:mg2s1}). We choose this $S^1$-invariant $G_2$ moduli space so that we have a homeomorphism between it and the product of the Calabi-Yau moduli space with the ``moduli space" $Z$ of potential twistings $z$, using the relationship between Calabi-Yau structures and $G_2$ structures. We then have to use this bijection to show the Calabi-Yau moduli space is a manifold. In subsection \ref{ssec:s1invg2modspacesmooth}, we prove that the $S^1$-invariant $G_2$ moduli space is locally homeomorphic to the moduli space of $G_2$ structures and so a manifold (Theorem \ref{mg2s1mfd}), by closely following the proof that the $G_2$ moduli space itself is a manifold. We also give an idea for an alternative proof of Theorem \ref{mg2s1mfd} and a discussion of where it runs into difficulty. We then discuss the space $Z$ of classes of twistings in subsection \ref{ssec:spacez}, identifying it as the open subset of the cohomology space $H^1(M \times S^1)$ corresponding to positive $H^1(S^1)$ component. Finally, in subsection \ref{ssec:su3modspacesmooth}, we return to the relationship between Calabi-Yau structures and $S^1$-invariant $G_2$ structures. We show that the projection map from the $S^1$-invariant $G_2$ moduli space to the ``moduli space" $Z$ of potential twistings $z$ is a smooth surjective submersion, so that each $SU(3)$ moduli-space fibre is a smooth manifold. To prove Theorem A (Theorem \ref{maintheorema}), it only then remains to show that all these fibres are diffeomorphic and that the product structure obtained in subsection \ref{ssec:s1invg2modspacesetup} is compatible with the manifold structures. 

Suppose $M$ is a (smooth) compact $6$-manifold. We henceforth restrict attention to $6$-manifolds which admit Calabi-Yau structures to avoid having to consider the possibility of empty moduli spaces. On manifolds with ends we will further restrict to manifolds for which these structures can be chosen asymptotically cylindrical. Quotienting by the pullback action of the identity component of the diffeomorphism group, we make
\begin{defin}
\label{defin:cptsu3modspace}
If $M$ is a compact $6$-manifold,
\begin{equation}
\mathcal M_{SU(3)}(M) = \frac {\{\text{Calabi-Yau structures on $M$}\}}{\text{$\Diff_0$ equivalence}}
\end{equation}
\end{defin}
We make a similar definition in the asymptotically cylindrical case. Recall the definition of $\Diff_0$ on such a manifold from Definition \ref{defin:acylisotopy}. 

\begin{defin}
\label{defin:msu3}
If $M$ is a $6$-manifold with an end,
\begin{equation}
\mathcal M_{SU(3)}(M) = \frac {\{\text{asymptotically cylindrical Calabi-Yau structures on $M$}\}}{\text{$\Diff_0$ equivalence}}
\end{equation}
\end{defin}
There is also a natural action by the rescaling $(\Omega, \omega) \mapsto (a^{\frac32} \Omega, a \omega)$ for a fixed constant $a$. We will not quotient by this, as it makes the setup of the moduli spaces slightly more complex: for details of the results we would get, and an example of the resulting complexity, see Remark \ref{rescalingmodspaces} below. 

In the $G_2$ case, similarly, we restrict to $7$-manifolds that admit (asymptotically cylindrical) torsion-free $G_2$ structures. We correspondingly make
\begin{defin}
If $M$ is a $6$-manifold, either compact or with an end,
\begin{equation*}
\mathcal M_{G_2}(M \times S^1) = \frac {\{\text{(asymptotically cylindrical) torsion-free $G_2$ structures on $M \times S^1$}\}}{\text{$\Diff_0$ equivalence}}
\end{equation*}
\end{defin}
Note that $M \times S^1$ has an end if and only if $M$ does. 
\subsection{Setup of the \texorpdfstring{$S^1$}{S¹}-invariant \texorpdfstring{$G_2$}{G₂} moduli space}
\label{ssec:s1invg2modspacesetup}
By Theorem \ref{g2su3alltold}, we have in both the compact and asymptotically cylindrical cases a bijection roughly given by
\begin{equation}
\label{eq:thmcrough}
\begin{split}
&\{\text{Calabi-Yau structures}\} \oplus \R_{>0} \oplus \{\text{closed $1$-forms}\} \leftrightarrow \\&\{\text{$S^1$-invariant torsion-free $G_2$ structures}\}
\end{split}
\end{equation}
In this subsection, we show that this bijection induces a homeomorphism of moduli spaces. Note that since we have not proved that the asymptotically cylindrical $SU(3)$ moduli space is a manifold, we cannot yet ask for a diffeomorphism. 

Therefore, we first define the moduli space of $S^1$-invariant $G_2$ structures and a ``moduli space" of twistings. We do so precisely so that the map induced from Theorem \ref{g2su3alltold} is a well-defined bijection. 

We shall use the following set of diffeomorphisms
\begin{defin}
\label{defin:s1invdiffeo}
Suppose that $M$ is a compact or asymptotically cylindrical manifold with an $S^1$-invariant Ricci-flat metric $g$. Let the space $\Diff_0^{S^1}$ be the identity path-component of
\begin{equation}
\label{eq:diff0s1wholespace}
\left\{\Phi \in \Diff_0: \Phi_* \ddth = \ddth\right\}
\end{equation}
Elements of \eqref{eq:diff0s1wholespace} shall be called $S^1$-invariant diffeomorphisms. 
\end{defin}
In Proposition \ref{s1invdiffeoweak}, we shall show that $\Diff_0^{S^1}$ is equal to the identity path-component of diffeomorphisms satisfying the weaker condition that $\Phi_* \ddth$ is a Killing field for $g$. 

The diffeomorphisms of $M$ extended by the identity certainly define $S^1$-invariant diffeomorphisms, and so we have 
\begin{lem}
\label{easyextensionlemma}
If $\Phi$ is a (asymptotically cylindrical) diffeomorphism of $M^6$ isotopic to the identity then the diffeomorphism $\hat\Phi: (x, \theta) \mapsto (\Phi(x), \theta)$ of $M^6 \times S^1$ lies in $\Diff_0^{S^1}(M \times S^1)$. If $(\Omega, \omega)$ is a (asymptotically cylindrical) Calabi-Yau structure and $Ld\theta +v$ a twisting, the torsion-free $G_2$ structures $\Phi^*(\Re \Omega)) + (Ld\theta + \Phi^*v) \wedge \Phi^*(\omega)$ and $\Re \Omega + (Ld\theta + v) \wedge \omega$ are identified by $\hat \Phi$. 
\end{lem}
We would like $\Phi^*(\Re \Omega)) + (Ld\theta + v) \wedge \Phi^*(\omega)$ and $\Re \Omega + (Ld\theta + v) \wedge \omega$ to be identified in our $S^1$-invariant $G_2$ moduli space, as they correspond to the same element of the $SU(3)$ moduli space with the same twisting. Lemma \ref{easyextensionlemma} says that it is sufficient to choose some more diffeomorphisms so that $\Re \Omega + (Ld\theta + \Phi^*v) \wedge \omega$ and $\Re \Omega + (Ld\theta + v) \wedge \omega$ are identified. More concretely, we shall identify $S^1$-invariant $G_2$ structures where the twisting differs by $df$ for some (asymptotically translation invariant) $f$; $v-\Phi^*v$ is exact and it's clear that the resulting $f$ can be chosen to be asymptotically translation invariant if necessary, by its explicit form as the integral of an asymptotically translation invariant integrand.
\begin{lem}
\label{changingzlemma}
If $(\Omega, \omega)$ is a Calabi-Yau structure on $M$, $Ld\theta + v$ is a twisting, and $f$ is a bounded function on $M$, there is a diffeomorphism $\Phi \in \Diff_0^{S^1}(M \times S^1)$ such that
\begin{equation}
\label{eq:changingzlemma}
\Phi^*(\Re \Omega + (Ld\theta + v)\wedge \omega) = \Re \Omega + (L d\theta + v + df) \wedge \omega
\end{equation}
\end{lem}
\begin{proof}
Consider the curve of diffeomorphisms 
\begin{equation}
\label{eq:changingzdiffeo}
\Phi_s: (x, \theta) \mapsto \left(x, \theta + \frac{sf(x)}L\right)
\end{equation}
Each $\Phi_s$ is clearly smooth and smoothly invertible by $(x, \theta) \mapsto (x, \theta - \frac{sf(x)}L)$. Moreover, it is easy to see that $\Phi_{s*} \ddth = \ddth$ for all $s$: hence $\Phi = \Phi_1$ is in $\Diff_0^{S^1}(M \times S^1)$. 

As all the $M^6$ coordinates are left unchanged, and the structure is invariant by $S^1$ (and so $\Re \Omega$, $\omega$ and $v$ are), $\Phi$ acts as the identity on them. However, by definition
\begin{equation}
\begin{aligned}
\Phi^*(d\theta) &=d(\theta\circ\Phi)\\
		&=d\left(\theta + \frac{f(x)}L\right) \\
		&=d\theta + \frac{df}L
\end{aligned}
\end{equation}
We obtain \eqref{eq:changingzlemma}. 
\end{proof}
Lemmas \ref{easyextensionlemma} and \ref{changingzlemma} show that if we quotient by $\Diff_0^{S^1}(M \times S^1)$ we have a well-defined map from the Calabi-Yau moduli space. We shall thus make

\begin{defin}
\label{defin:mg2s1}
\begin{equation}
\mathcal M_{G_2}^{S^1}(M \times S^1) := \frac{\text{$S^1$-invariant torsion-free $G_2$ structures}}{\Diff^{S^1}_0(M\times S^1)}
\end{equation}
\end{defin}
We shall call $\M_{G_2}^{S^1}$ the $S^1$-invariant $G_2$ moduli space. 

It remains to choose the ``moduli space" $Z$ of twistings $z$. Given Definition \ref{defin:mg2s1}, Lemma \ref{changingzlemma} implies that we have to quotient by the differentials of (asymptotically translation invariant) functions in order to make the induced map an injection. Consequently, we make
\begin{defin}
\label{defin:spacez}
Let the set of twisting-classes $Z$ be the quotient of the twistings of Definition \ref{defin:twistings}
\begin{equation}
\frac{\{\text{closed $S^1$-invariant $1$-forms $z = Ld\theta +v:L>0$ (with $\tilde z(\ddt) = 0$) on $M \times S^1$}\}}{\{\text{differentials of $S^1$-invariant (asymptotically translation invariant) functions}\}}
\end{equation}
\end{defin}
$Z$ has a relatively simple description, which is discussed in subsection \ref{ssec:spacez} leading to Lemma \ref{ccfcohom} below. 

We now verify that the map 
\begin{equation}
\M_{SU(3)} \times Z \to \M_{G_2}^{S^1}
\end{equation}
induced from Theorem \ref{g2su3alltold} is a well-defined bijection. Well-definition follows from Lemmas \ref{easyextensionlemma} and \ref{changingzlemma}; injectivity is
\begin{lem}
\label{s1invdiffstr}
Suppose that $\Phi \in \Diff_0^{S^1}(M \times S^1)$ and that $(\Omega, \omega)$ is a Calabi-Yau structure on $M$ with $z = Ld\theta + v$ a closed complementary covector field with positive orientation (\ie $L>0$.) Then, choosing a point on $S^1$ and so an identification of $S^1$ with $\frac \R \Z$, $\Phi$ is of the form $(p, \theta) \mapsto (f(p), \theta + g(p))$ for some smooth functions $f$ and $g$. Consequently, there exist $\Phi_1 \in \Diff_0(M)$ and $\Phi_2$ the time-$1$ flow of $f \ddth$ (for $f$ asymptotically translation invariant, if necessary) such that
\begin{equation}
\Phi^* \Omega = \Phi_2^* \Phi_1^* \Omega \qquad \Phi^* \omega = \Phi_2^* \Phi_1^* \omega \qquad \Phi^* z = \Phi_2^* \Phi_1^* z
\end{equation}
\end{lem}
\begin{rmk}
In Lemma \ref{s1invdiffstr}, we do not claim that $\Phi = \Phi_1 \Phi_2$, though this will of course be true up to an isometry. 
\end{rmk}
\begin{proof}
Given $0 \in S^1$, write $\Phi(p, 0) = (f(p), g(p))$ for smooth functions $f$ and $g$. Now suppose $(p, \theta') \in S^1$, and consider the curve $(p, s\theta')$ between $(p, 0)$ and $(p, \theta')$ in $M \times S^1$. At all points of this curve its derivative is $\theta'\ddth$. Consequently, the derivative of its image is $\theta' \ddth$ and so its image is $(f(p), g(p) + s\theta')$. Hence $\Phi(p, \theta') = (f(p), g(p) + \theta')$, as required. Hence, $f$ is a diffeomorphism. 

Now let $\Phi_1$ be the extension of the map $f$ as in Lemma \ref{easyextensionlemma}. It is then clear that we have $\Phi_1^* \Omega = \Phi^* \Omega$ and $\Phi_1^* \omega = \Phi^* \omega$. 

The $\theta$ component is preserved by $\Phi_1$, and so $\Phi_1^* (td\theta + v) = td\theta + \Phi^* v$. As before, $\Phi^* v - \Phi_1^* v$ is exact, and in the asymptotically cylindrical case is the differential of an asymptotically translation invariant function. Thus there exists such a $\Phi_2$, by Lemma \ref{changingzlemma}.
\end{proof}
\begin{prop}
\label{setproduct}
\label{topproduct}
The map induced by the bijection of Theorem \ref{g2su3alltold} is a well-defined homeomorphism
\begin{equation}
\label{eq:mg2s1productmap}
\M_{SU(3)}(M) \times Z \to \M_{G_2}^{S^1}(M \times S^1)
\end{equation}
\end{prop}
The fact that \eqref{eq:mg2s1productmap} is a homeomorphism follows immediately as the maps between structures are continuous and we just take the quotient topology. However, we know very little about what the topology on the right hand side looks like. 
\subsection{Smoothness of the \texorpdfstring{$S^1$}{S¹}-invariant \texorpdfstring{$G_2$}{G₂} moduli space}
\label{ssec:s1invg2modspacesmooth}
In order to use the homeomorphism of Proposition \ref{topproduct} to show the moduli space of Calabi-Yau structures $\M_{SU(3)}$ is a manifold, we first have to show that $\M_{G_2}^{S^1}$ is a manifold. The objective of this subsection is to prove Theorem \ref{mg2s1mfd}, which says that $\M_{G_2}^{S^1}$ is a manifold and in fact is locally diffeomorphic to $\M_{G_2}$. The idea of the proof is that the constructions of the $G_2$ moduli space due to Hitchin \cite{hitchinthreeforms} and Nordstr\"om \cite{nordstromacyldeformations} work by, given a $G_2$ structure, constructing a geometrically natural premoduli space of structures and arguing that this premoduli space is locally homeomorphic to the moduli space. Since the construction is geometrically natural, if the $G_2$ structure concerned is $S^1$-invariant, the premoduli space also consists of $S^1$-invariant structures, and the result then follows in the same way. 

In the $G_2$ case, the standard result is 
\begin{thm}[\cite{hitchinthreeforms} for (a), {\cite[Proposition 6.18]{nordstromacyldeformations}} for (b), cf. {\cite[Theorem 7.1]{ebin}} for (c)(i), (c)(ii) by combining the proofs of (a) and (b) with (c)(i)]
\label{g2slicetheorem}
Let $M$ be a six-dimensional manifold. 
\begin{enumerate}[a)]
\item 
If $M$ is compact, the moduli space of torsion-free $G_2$ structures on $M \times S^1$ is a smooth manifold. It is locally diffeomorphic to $H^3(M\times S^1) \cong H^3(M) \oplus H^2(M)$ by the well-defined map that takes a representative of a moduli class to its cohomology class. 

\item 
If $M$ has an end, the moduli space of asymptotically cylindrical torsion-free $G_2$ structures on $M \times S^1$ is a smooth manifold. It is locally diffeomorphic to a submanifold of $H^3(M \times S^1) \oplus H^2(N \times S^1) \cong H^3(M) \oplus H^2(M) \oplus H^2(N) \oplus H^1(N)$, by the well-defined map taking a representative $\phi$ of a moduli class with limit $\tilde \phi = \tilde \phi_1 + dt \wedge \tilde\phi_2$ to the pair $([\phi], [\tilde \phi_2])$. This submanifold is wholly determined by the requirement that there be a Calabi-Yau structure $(\Omega, \omega)$ on $N \times S^1$ with $(\Re \Omega, \omega) = (\tilde \phi_1, \tilde \phi_2)$, which follows from section \ref{sec:g2su3}. 

\item In either case, given a (asymptotically cylindrical) torsion-free $G_2$ structure $\phi$ on $M \times S^1$, we may find a set $U$ of such structures containing $\phi$ and with the following properties. 
\begin{enumerate}[i)]
\item All the structures in the chart have the same group $\Aut$ of (asymptotically cylindrical) automorphisms isotopic to the identity. 

\item There are neighbourhoods $D$ of $[\id] \in \frac{\Diff_0}{\Aut}$ and $V$ of $\phi$ in the set of all (asymptotically cylindrical) torsion-free $G_2$ structures such that the pullback map defines a bijection between the product $D \times U$ and $V$ (in fact, a homeomorphism). 
\end{enumerate}
The set $U$ is called a slice neighbourhood (or just a slice). 
\end{enumerate}
\end{thm}
We concentrate on the proofs of (a) and (b); (c)(i) follows in exactly the same way in the $S^1$-invariant case as in the general case, and then (c)(ii) also follows. We begin with (a), as it is simpler. From the work of Hitchin, we extract Propositions \ref{cptslice} and \ref{cptdetails} which together prove (a). Of course, since we rely on the implicit function theorem, these should properly be stated in terms of suitable Banach spaces. However, the choice of Banach spaces is straightforward and of no relevance to the introduction of $S^1$-invariance; consequently we omit it. 

We first need a slice for the $\Diff_0$ action at some $G_2$ structure $\phi$, that is, essentially, a local cross-section of the quotient. To choose a slice, we need the following standard fact about two-forms on a manifold with a torsion-free $G_2$ structure. 
\begin{lem}[cf. {\cite[Lemma 11.4]{salamonbook}}]
\label{twoformssplitting}
Let the seven-manifold $X$ admit a $G_2$ structure $\phi$. We then have an isomorphism of bundles
\begin{equation}
\bigwedge\nolimits^{\!2} T^*(X) = \bigwedge\nolimits^{\!2}_7 \oplus \bigwedge\nolimits^{\!2}_{14}
\end{equation}
where $\bigwedge^2_7$ is a rank-seven bundle given by contractions of $\phi$ with tangent vectors, $\bigwedge^2_{14}$ is a rank-fourteen bundle, and the fibres of these sub-bundles are orthogonal with respect to the inner product induced on two-forms by $\phi$. 
\end{lem}
We apply Lemma \ref{twoformssplitting} in the case where $X = M \times S^1$. It is virtually sufficient to prove the following proposition. Note that here we make no torsion-freeness assumption on the three-forms other than $\phi$. 
\begin{prop}[cf. {\cite[{bottom of p.23}]{hitchinthreeforms}}]
\label{cptslice}
Let $\phi$ be a torsion-free $G_2$ structure on the compact seven-manifold $M \times S^1$. Let
\begin{equation}
E = \{\alpha \in \Omega^3(M \times S^1):  d\alpha = 0, d^*\alpha \in \Omega^2_{14}\}
\end{equation}
where $\Omega^2_{14}$ is the subspace of $2$-forms which at every point are in the subbundle $\bigwedge\nolimits^{\!2}_{14}$. 

Then $E$ is $L^2$-orthogonal to the space of Lie derivatives $\L_X \phi$ of $\phi$, with respect to the metric at $\phi$, and the sum of these spaces is the set of all three-forms. (We choose the Banach spaces so that the projection onto $E$ is continuous). Consequently, locally $E$ is transverse to the orbits of the identity component of the diffeomorphism group. 
\end{prop}
Now we have to pass to the torsion-free $G_2$ structures in the slice $E$. 
\begin{prop}[{\hspace{1sp}\cite[p. 35]{hitchinthreeforms}}]
\label{cptdetails}
Let $\phi_0$ be a torsion-free $G_2$ structure on $M \times S^1$ as in the previous proposition, and let $E$ be as stated there. Define a map $F$ from a neighbourhood of $\phi_0 \in E$ to exact forms by $F(\phi) = P(*_0 * \phi)$ where $*_0$ is the Hodge star induced by $\phi_0$, $*$ is the Hodge star induced by $\phi$, and $P$ is the orthogonal projection onto exact forms induced by $\phi_0$. 

If $F(\phi) = 0$ for $\phi$ sufficiently close to $\phi_0$, then $\phi$ is itself a torsion-free $G_2$ structure. 

The derivative $DF$ has kernel consisting of harmonic forms and is surjective to the exact forms, proving (a) of Theorem \ref{g2slicetheorem}. 
\end{prop}
We have to transfer Propositions \ref{cptslice} and \ref{cptdetails} and their asymptotically cylindrical analogues from the work of Nordstr\"om \cite{nordstromacyldeformations} to the $S^1$-invariant setting. We begin with the slice in both the compact and asymptotically cylindrical setting. 

We first need to show that $\Diff_0^{S^1}$ has a well-defined tangent space so that we can still work with the space of Lie derivatives. In other words, we must prove
\begin{prop}
\label{diff0s1submfd}
$\Diff_0^{S^1}$ is a manifold. Its tangent space is given by the $S^1$-invariant vector fields.
\end{prop}
\begin{proof}
Since $\Diff_0^{S^1}$ is a group, it suffices to show it is locally a manifold around the identity. A neighbourhood of the identity in $\Diff_0$ is given, using the Riemannian exponential map for some metric $g$, by a neighbourhood of zero in the vector fields on $M$. 

We choose $g$ to be $S^1$-invariant. From Lemma \ref{s1invdiffstr} and a trivial calculation, the $S^1$-invariant diffeomorphisms are the diffeomorphisms of the form $\Phi(p, \theta) = (f(p), g(p) + \theta)$; equivalently, they are precisely the diffeomorphisms that commute with the rotations $\Theta$. We must show that the vector fields in this neighbourhood such that $\exp_v$ commutes with $\Theta$ are a submanifold with the specified tangent space. In fact we shall show that they are precisely the $S^1$-invariant vector fields small enough to be in the neighbourhood; the intersection of a vector subspace with the neighbourhood is clearly a submanifold. 

Suppose given a sufficiently small vector field $v$. We have to compare $\Theta \exp_v(p, \theta)$ and $\exp_v \Theta(p, \theta)$, for some $(p, \theta) \in M \times S^1$. Suppose that $\Theta$ is rotation by $\theta'$. $\Theta\exp_v(p, \theta)$ is the endpoint of the image under $\Theta$ of the geodesic with initial velocity $v_{(p, \theta)}$; since the metric is $S^1$-invariant, $\Theta$ is an isometry, and so $\Theta\exp_v(p, \theta)$ is the endpoint of the geodesic with initial velocity $(\Theta_* v)_{(p, \theta + \theta')}$. On the other hand, $\exp_v\Theta(p, \theta)$ is the endpoint of the geodesic with initial velocity $v_{(p, \theta + \theta')}$. By assumption, $v$ and hence $\Theta_* v$ are sufficiently small that the exponential map is injective, and consequently we have $\exp_v \Theta = \Theta \exp_v$ if and only if $\Theta_* v = v$.
\end{proof}
That is, the tangent space to the orbit of a $G_2$ structure $\phi$ under $\Diff_0^{S^1}$ is given by the Lie derivatives $\L_X \phi$ with $X$ $S^1$-invariant. We shall prove that for $\phi$ $S^1$-invariant, this is equivalent to the $S^1$-invariant Lie derivatives $\L_X \phi$. We begin with the simplest case: an $S^1$-invariant $G_2$ structure on a compact manifold.
\begin{lem}
\label{g2cpts1invardiffeo}
Let $M$ be a compact six-dimensional manifold and let $\phi$ be an $S^1$-invariant torsion-free $G_2$ structure on $M \times S^1$. Suppose that $X$ is a vector field on $M \times S^1$ and $\L_X \phi$ is $S^1$-invariant. Then $X$ is $S^1$-invariant. 
\end{lem}
\begin{proof}
We begin by showing that any Killing field $X$ is $S^1$-invariant. As $M$ is Ricci-flat, a Killing field is parallel; hence $\L_\ddth X = [X, \ddth] = 0$ and this is equivalent to $S^1$-invariance. 

Now suppose that $\L_X \phi$ is $S^1$-invariant, but not necessarily zero. Then we have $\L_\ddth \L_X \phi = 0 = \L_X \L_\ddth \phi$, and so $\L_{[X, \ddth]} \phi = 0$. By the previous paragraph we find that $[X, \ddth]$ is $S^1$-invariant. 

We may now work locally on $M$. Pick some open subset of $M$ on which we have coordinates $x_1, \ldots, x_n$, and on this subset write
\begin{equation}
X = a_0 \ddth + \sum_{i=1}^n a_i \frac{\partial}{\partial x_i} 
\end{equation}
We see by elementary computation that 
\begin{equation}
0 = [[X, \ddth], \ddth] = \frac{\partial^2 a_0} {\partial \theta^2} \ddth + \sum_{i=1}^n \frac{\partial^2 a_i}{\partial \theta^2} \frac{\partial}{\partial x_i}
\end{equation}
It follows that each of the $a_i$ ($i$ possibly zero) is of the form $A_i \theta + B_i$, where $A_i$ and $B_i$ are functions on $M$. But as there is no globally defined function $\theta$, $A_i$ must be identically zero. It follows that $X$ is independent of $\theta$, as required.
\end{proof}
In the asymptotically cylindrical case, we will need a couple of statements very similar to Lemma \ref{g2cpts1invardiffeo}. 
\begin{lem} %label is at bottom of this lemma because stupid pdfTeX not quite bug about references: still not perfect, as end up at the end of the lemma not the beginning, but better than before
\begin{enumerate}[i)]
\item Let $N$ be a compact five-dimensional manifold and let $(\Omega, \omega)$ be an $S^1$-invariant Calabi-Yau structure on $N \times S^1$. Suppose that $X$ is a vector field on $N \times S^1$ and $\L_X \Re \Omega$ is $S^1$-invariant. Then $X$ is $S^1$-invariant, and in particular any Killing field is $S^1$-invariant.
\item Let $M$ be a six-dimensional manifold with an end and let $\phi$ be an asymptotically cylindrical torsion-free $G_2$ structure on $M \times S^1$. Suppose that $X$ is an asymptotically translation invariant vector field on $M \times S^1$ such that $\L_X \phi$ is also $S^1$-invariant. Then $X$ is $S^1$-invariant; in particular, again, any Killing field is $S^1$-invariant.
\end{enumerate}
\label{rests1invardiffeo}
\end{lem}
\begin{proof}
(ii) is essentially identical. The only part of Lemma \ref{g2cpts1invardiffeo} that required compactness of $M$ was showing that Killing fields are parallel, and this was briefly explained in Proposition \ref{g2su3prop2}. 

(i) is slightly more involved, as it is not immediately clear that $\L_X \Re \Omega = 0$ implies that $X$ is a Killing field and so parallel. However, it follows from Hitchin \cite{hitchinthreeforms}, as follows.

Hitchin proves that $\Im \Omega$ can be determined at each point from $\Re \Omega$. Moreover, the map $\Re \Omega \mapsto \Im \Omega$ is smooth. Hence, if $\L_X \Re \Omega = 0$, we have a curve of Calabi-Yau structures corresponding to a curve of diffeomorphisms; at each point $\Im \Omega$ depends smoothly on $\Re \Omega$, and so $\L_X \Im \Omega$ depends linearly on $\L_X \Re \Omega$, and must in turn be zero. Hence, if $\L_X \Re \Omega =0$, then $\L_X \Omega =0$, \ie $X$ is holomorphic. But then it follows by the argument from Kobayashi \cite[Theorem III.5.2]{kobayashitransformations} mentioned before Lemma \ref{sunauto} that $\L_{[X, \ddth]} \omega = 0$. Then $X$ is Killing, and we can apply the previous argument. 
\end{proof}
A similar argument proves a non-infinitesimal version. It does not quite imply Lemma \ref{g2cpts1invardiffeo} as we would need to integrate up: we would need to show that a vector field $X$ with $\L_X \phi$ $S^1$-invariant induces a curve $\Phi_s$ of diffeomorphisms so that $\Phi_s^* \phi$ is $S^1$-invariant for all $s$ sufficiently small. 
\begin{prop}
\label{s1invdiffeoweak}
Let $M \times S^1$ have an $S^1$-invariant Ricci-flat metric $g$. The space $\Diff_0^{S^1}$ defined in Definition \ref{defin:s1invdiffeo} is also the identity path-component of the set
\begin{equation}
\label{eq:diff0s1biggerwholespace}
\{\Phi \in \Diff_0(M \times S^1): \Phi_* \ddth \text{ is Killing}\}
\end{equation}
\end{prop}
\begin{proof}
Because $\ddth$ is certainly a Killing field for the metric $g$, it is clear that \eqref{eq:diff0s1wholespace} is a subspace of \eqref{eq:diff0s1biggerwholespace}, and consequently $\Diff_0^{S^1}(M \times S^1)$ is contained in the identity path-component of \eqref{eq:diff0s1biggerwholespace}. It suffices to show that \eqref{eq:diff0s1wholespace} is open and closed in \eqref{eq:diff0s1biggerwholespace}, for then connectedness of the identity component of \eqref{eq:diff0s1biggerwholespace} implies it is all of $\Diff_0^{S^1}(M \times S^1)$. 

Closedness follows immediately from continuity of the pushforward. 

For openness, we suppose that $\Phi_0$ is in $\Diff_0^{S^1}$; we need to show that there is an open neighbourhood $U$ of $\Phi_0$ in $\Diff_0$ such that if $\Phi \in U$ and $\Phi_* \ddth$ is Killing, then $\Phi_* \ddth$ is $\ddth$. We work around some $p \in M$. We note first that by the argument of Lemma \ref{s1invdiffstr}, $\Phi_0$ carries $\{p\} \times S^1$ onto $\{q\} \times S^1$ for some point $q$ of $M$ depending on $p$. Let $V$ be a small chart around $p$ and let $W$ be a small chart around $q$ such that $\Phi_0(V \times S^1) \subset W \times S^1$. Choose a smaller neighbourhood $V'$ whose closure is compact and contained in $V$; then for $\Phi$ sufficiently close to $\Phi_0$, $\Phi(V' \times S^1) \subset W \times S^1$. 

We may consequently analyse $\Phi$ in terms of the coordinates on $U \times S^1$ and $V \times S^1$; that is, we shall write
\begin{equation}
\Phi(x_1, \ldots, x_n, \theta) = (y_1(x_1, \ldots, x_n, \theta), \ldots, y_n(x_1, \ldots, x_n, \theta), \theta'(x_1, \ldots, x_n, \theta))
\end{equation}
Now $\Phi_* \ddth$ is Killing, and so parallel as in the proof of Proposition \ref{g2su3prop2}, hence $S^1$-invariant. It follows that for each $i$, $\frac{\partial^2 y_i}{\partial \theta^2} = 0$. We deduce that for fixed $x_1, \ldots, x_n$, $y_i = A_i \theta + B_i$, where $A_i$ and $B_i$ depend on $x_1, \ldots, x_n$. It follows immediately that $A_i = 0$, as otherwise we do not have a well-defined map from the circle. Hence, we find that
\begin{equation}
\Phi_* \ddth = A \ddth
\end{equation}
for some function $A(x_1, \ldots, x_n, \theta)$. Taking the second derivatives of $\theta'$ in the same way, $A$ must be independent of $\theta$; since we still need a well-defined map from the circle, we also know that $A \in \mathbb{Z}$ for all points of $M$. Since $\Phi$ is close to $\Phi_0$ in the $C^1$ topology and $\Phi_{0*} \ddth = \ddth$ we obtain $A = 1$ identically, as required.
\end{proof}
\begin{rmk}
It follows immediately from Proposition \ref{s1invdiffeoweak} that the identity component of the isometry group of a manifold $M \times S^1$ with an $S^1$-invariant Ricci-flat metric $g$ is contained in $\Diff_0^{S^1}$. Hence, when we quotient by the automorphism group as in Theorem \ref{g2slicetheorem}(c)(ii), it doesn't matter whether we work with $\Diff_0$ or $\Diff_0^{S^1}$ as in a local neighbourhood of the origin the subgroups of isometries are the same. In practice, of course, the fact that Killing fields are $S^1$-invariant would also prove that the subgroups of isometries are the same locally around the identity. 
\end{rmk}
We can now show that slices such as that in Proposition \ref{cptslice} restrict to slices in the $S^1$-invariant setting. For the asymptotically cylindrical case, we will need analogous results to Proposition \ref{cptslice} for Calabi-Yau structures as well, and so Proposition \ref{generals1invarslice} contains three cases corresponding to Lemma \ref{g2cpts1invardiffeo} and the two cases of Lemma \ref{rests1invardiffeo}. 
\begin{prop}
\label{generals1invarslice}
Suppose that $W \times S^1$ is a six- or seven-dimensional manifold, where $W$ is one of a compact or asymptotically cylindrical six-manifold, or a compact five-manifold, and has either an $S^1$-invariant torsion-free $G_2$ structure $\phi$ or an $S^1$-invariant Calabi-Yau structure $(\Omega, \omega)$ respectively. Let $\alpha$ be $\phi$ or $\Re \Omega$ respectively. Suppose that
\begin{equation}
Y = E \oplus \{L_X \alpha: X \text{ vector field}\}
\end{equation}
is an orthogonal splitting for some vector spaces of three-forms $E$ and $Y$.

Let $E'$ and $Y'$ be the intersection of $E$ and $Y$ respectively with the set of $S^1$-invariant three-forms. Then we have
\begin{equation}\label{s1invslicesecondeq}
Y' = E' \oplus \{L_X \phi: X \text{ $S^1$-invariant vector field}\}
\end{equation}
again as an orthogonal splitting. 
\end{prop}

\begin{proof}
Lemmas \ref{g2cpts1invardiffeo} and \ref{rests1invardiffeo} say that the space of Lie derivatives on the right hand side of \eqref{s1invslicesecondeq} is precisely
\begin{equation}
\{\text{$\L_X \alpha$: $X$ vector field}\} \cap \{\text{$S^1$-invariant forms}\}
\end{equation}
and thus we just need to show that the orthogonal splitting is preserved when we intersect throughout with $S^1$-invariant forms. The two components obviously remain orthogonal, so it suffices to show that the projection to the Lie derivative part is $S^1$-equivariant, \ie commutes with the pullback by any rotation $\Theta$. We observe that $\Theta^* \L_X \phi = \L_{\Theta^{-1}_* X} \Theta^* \phi$ for all $X$, again using the $S^1$-invariance of $\phi$.

Now given a $3$-form $\beta$ its Lie derivative part is the unique $3$-form $\gamma$ such that 
\begin{equation}
\la \beta - \gamma, \L_X \alpha \ra = 0 \qquad \text{ for all vector fields $X$}
\end{equation}
As $\Theta$ is an isometry, we have for any $X$
\begin{equation}
\la \Theta^* \beta - \Theta^* \gamma, \L_X \alpha \ra = \la \beta - \gamma, (\Theta^{-1})^* \L_X \alpha \ra  = \la \beta - \gamma, \L_{\Theta^{-1}_* X} \alpha = 0
\end{equation}
and thus the Lie derivative part of $\Theta^* \beta$ is $\Theta^* \gamma$. In particular, if $\beta$ is $S^1$-invariant its Lie derivative part is $S^1$-invariant. Thus its $E$ part is too, and lies in $E'$. We obtain the orthogonal splitting
\begin{equation}
\{\text{closed $S^1$-invariant $3$-forms}\} = E' \oplus \{\text{$\L_X \phi$: $[X, \ddth]=0$}\} \qedhere
\end{equation}
\end{proof}
Applying Proposition \ref{generals1invarslice} to Proposition \ref{cptslice}, we obtain
\begin{cor}
\label{cor:g2cpts1invarslice}
Let $\phi$ be an $S^1$-invariant $G_2$ structure on $M \times S^1$. With respect to $\phi$, let 
\begin{equation}
E' = \{\alpha \in \Omega^3(M \times S^1): \alpha \text{ $S^1$-invariant}, d\alpha = 0, d^*\alpha \in \Omega^2_{14}\}
\end{equation}
where $\Omega^2_{14}$ is the subbundle of $2$-forms which at every point are in the subbundle corresponding to the $14$-dimensional subrepresentation of $G_2$ on the space of alternating two-forms. Then $E'$ is $L^2$-orthogonal to the space of Lie derivatives $\L_X \phi$ of $\phi$ for $X$ $S^1$-invariant, with respect to the metric at $\phi$, and the sum of these spaces is the set of all three-forms. Consequently, locally $E$ is transverse to the orbits of $\Diff_0$. 
\end{cor}
We now proceed to the $S^1$-invariant version of the implicit function theorem argument for Proposition \ref{cptdetails}. We have to show that when we pass to the $S^1$-invariant setting the derivative $DF$ is still surjective and has the same kernel. Again, we state and prove a more general version that will be used for the asymptotically cylindrical case. We first prove some easy lemmas saying that harmonic forms are $S^1$-invariant. The proof is very similar to the proof of Lemma \ref{sunauto}. 

\begin{lem}
\label{cptharms1invar}
Suppose $M \times S^1$ is a compact manifold with an $S^1$-invariant Riemannian metric, and $\alpha$ a harmonic form on it. Then $\alpha$ is $S^1$-invariant.
\end{lem}
\begin{proof}
Consider a rotation $\Theta$. As $\Theta$ is isotopic to the identity, $[\Theta^* \alpha] = [\alpha]$. As the metric is $S^1$-invariant, $\Theta$ is also an isometry and so $\Theta^* \alpha$ is also a harmonic form. Thus, by Hodge decomposition, $\Theta^*\alpha = \alpha$, \ie $\alpha$ is $S^1$-invariant.
\end{proof}
\begin{lem}
\label{acylharms1invar}
Suppose $M \times S^1$ is a manifold with an end $N \times S^1 \times (0, \infty)$, equipped with an asymptotically cylindrical  $S^1$-invariant Riemannian metric, and $\alpha$ an asymptotically translation invariant harmonic form on it. Then $\alpha$ is $S^1$-invariant.
\end{lem}
\begin{proof}
We apply asymptotically cylindrical Hodge theory, such as in \cite{nordstromacyldeformations}, and the argument in Lemma \ref{cptharms1invar}. By Theorem 5.9 of that paper, $\alpha$ is the sum of a decaying harmonic form, an exact harmonic form, and a coexact harmonic form. We first show that the exact harmonic form $\beta$, say, is $S^1$-invariant. 

By the discussion after Theorem 5.9, the map from the exact harmonic form to its limit is injective. By Lemma \ref{cptharms1invar}, its limit is $S^1$-invariant. Thus, the exact harmonic form $\Theta^* \beta$ has the same limit: so $\Theta^* \beta = \beta$, as required. By the same argument the coexact harmonic form is also $S^1$-invariant. Taking the difference, we now have to show that the decaying harmonic form, $\gamma$ say, is $S^1$-invariant. Again, $\Theta^* \gamma$ is a decaying harmonic form, and it follows from \cite[Theorem 5.9]{nordstromacyldeformations} that the map from such forms to cohomology is injective. Thus $\Theta^* \gamma = \gamma$ exactly as in Lemma \ref{cptharms1invar}. 
\end{proof}

We now prove our general $S^1$-invariance proposition which we will apply to Proposition \ref{cptdetails}. 
\begin{prop}
\label{generals1invardetails}
Let $M$ be a compact or asymptotically cylindrical manifold. Let $F: X \to Y$ be a smooth $S^1$-equivariant nonlinear map between Banach spaces of (asymptotically translation invariant) differential forms on $M \times S^1$ (with  $S^1$-invariant norms). Suppose that $DF$ is surjective and its kernel consists of harmonic forms. Suppose further that in the compact case $X$ is continuously contained in the space of $L^2$ forms and in the asymptotically cylindrical case we can find an $S^1$-invariant complement $W$ to the kernel of $DF$ in $X$. Then the restriction
\begin{equation}
F^{S^1}: X' := X \cap \{S^1\text{-invariant forms}\} \to Y \cap \{S^1\text{-invariant forms}\} =: Y'
\end{equation}
is a well-defined map, with $DF^{S^1}$ surjective and the same kernel as $DF$. 
\end{prop}
\begin{proof}
Since $F$ is $S^1$-equivariant, the image of an $S^1$-invariant form under it is an $S^1$-invariant form. Consequently, $F^{S^1}$ is a well-defined map. 

We now consider the derivative. By hypothesis, its kernel consists of (asymptotically translation invariant) harmonic forms; by Lemmas \ref{cptharms1invar} and \ref{acylharms1invar} we know that these are $S^1$-invariant, and consequently we know that the kernel of $DF^{S^1}$ agrees with the kernel of $DF$. For surjectivity, note first that $DF^{S^1}$ must also be $S^1$-equivariant.

Suppose we have a form $\alpha \in F(X')$, and let $\dot \alpha$ be a tangent to $Y'$ at $\alpha$, so we have to show that $\dot \alpha$ is in the image of $DF^{S^1}$. Since $DF$ is surjective, we can find $\dot \beta$ with $DF(\dot \beta) = \dot \alpha$. There is always an $S^1$-invariant complement $W$ to the finite-dimensional kernel of $DF$; in the asymptotically cylindrical case, this is assumed, and in the compact case we may let $W$ be the $L^2$-orthogonal complement and note that since the metric is $S^1$-invariant, $W$ is also $S^1$-invariant. We suppose that $\dot \beta$ lies in $W$. 

Now, given a  rotation $\Theta$, since $DF^{S^1}$ is $S^1$-equivariant and $\dot \alpha$ is $S^1$-invariant, $\Theta^* \dot \beta$ also maps to $\dot \alpha$. Since $W$ is $S^1$-invariant, $\Theta^* \dot \beta$ is also in $W$. Consequently the difference $\dot \beta - \Theta^* \dot \beta$ is in $W$ and maps to zero; hence it is zero, and $\dot \beta$ is $S^1$-invariant. Hence, $\dot \alpha = DF^{S^1} (\dot \beta)$, and $DF^{S^1}$ is surjective.
\end{proof}
Applying Proposition \ref{generals1invardetails} to Proposition \ref{cptdetails} gives 
\begin{cor}
\label{cor:cpts1invardetails}
Let $\phi_0$ be a torsion-free $S^1$-invariant $G_2$ structure on $M \times S^1$ and let $E$ and $F$ be as in Proposition \ref{cptdetails}. When we restrict all the spaces to be $S^1$-invariant (passing to $E'$ and so forth), the kernel of $DF$ is unchanged and it remains surjective.
\end{cor}
\begin{proof}
We apply Proposition \ref{generals1invardetails}. We have to check that $F$ is $S^1$-equivariant and that its kernel consists of harmonic forms. The fact the kernel consists of harmonic forms is stated in Proposition \ref{cptdetails}. For $S^1$-equivariance, we recall that $F(\phi) = P(*_0 * \phi)$ where $*_0$ and $*$ are the Hodge stars induced by $\phi_0$ and $\phi$ and $P$ is the orthogonal projection onto exact forms induced by $\phi_0$. Since $\phi_0$ is $S^1$-invariant, $*_0$ and $P$ are $S^1$-equivariant ($P$ was essentially proved to be so as part of Proposition \ref{generals1invarslice}). On the other hand, it is clear that the map $\phi \mapsto * \phi$ is $S^1$-equivariant. Consequently, $F$ is $S^1$-equivariant. By Proposition \ref{generals1invardetails}, the result follows. 
\end{proof}
Corollaries \ref{cor:g2cpts1invarslice} and \ref{cor:cpts1invardetails} essentially prove the compact case of Theorem \ref{mg2s1mfd}. 

The asymptotically cylindrical case is similar but more involved. Propositions \ref{generals1invarslice} and \ref{generals1invardetails} will suffice to prove everything, but we have to apply them to the proof of smoothness of the asymptotically cylindrical $G_2$ moduli space in Nordstr\"om \cite{nordstromacyldeformations} which is substantially more complicated than Hitchin's proof for the compact case. We consequently only summarise this case.

We begin by considering the limit as in \cite[section 4]{nordstromacyldeformations}. The limit is a torsion-free Calabi-Yau structure on the cross-section $N \times S^1$, so first of all requires us to define an $S^1$-invariant Calabi-Yau moduli space on $N \times S^1$. 

As in the compact $G_2$ case, the proof is fundamentally that we first choose a slice and then consider the torsion-freeness map from that slice. In our case, we work with a base Calabi-Yau structure $(\Omega, \omega)$ which is $S^1$-invariant. The slice is defined in \cite[Proposition 4.7]{nordstromacyldeformations}, by identifying an orthogonal complement to the Lie derivative of $\Re \Omega$ (essentially symmetrically to that in Proposition \ref{cptslice}. It is then not necessary to consider the Lie derivatives of $\omega$, as each diffeomorphism can be identified by its $\Re \Omega$ part. The Calabi-Yau case of Proposition \ref{generals1invarslice} says that when all forms in these splittings are taken $S^1$-invariant, the reduced orthogonal complement is still an orthogonal complement.

The map defining torsion-freeness is $F$ in \cite[Definition 4.12]{nordstromacyldeformations}, viz.
\begin{equation}
\label{eq:cycrosssectionF}
F(\beta, \gamma) = (P_1(*\hat\beta), P_2(\beta \wedge \gamma), \frac14 \beta \wedge \hat \beta - \frac16 \gamma^3)
\end{equation}
where, for $(\beta, \gamma)$ close to $(\Re \Omega, \omega)$, $\hat \beta$ is the imaginary part of the unique decomposable complex $3$-form of which $\beta$ is the real part, $P_1$ is the orthogonal projection to those three-forms in the slice which are orthogonal to harmonic forms, and $P_2$ is an orthogonal projection on closed five-forms induced by the Calabi-Yau structure (to the harmonic forms and exterior derivatives of $(3, 1) + (1, 3)$ forms). 

Wedge products are clearly $S^1$-equivariant. The map $\beta \mapsto \hat \beta$ is $S^1$-equivariant by uniqueness of the decomposable complex $3$-form, so to show that $F$ is $S^1$-equivariant we only need to show that the two projections are. However, the Calabi-Yau structure with respect to which these splittings are taken is $S^1$-invariant, so the subspaces that these splittings project to and their complements are $S^1$-invariant. It follows as in Proposition \ref{generals1invarslice} that the projection maps are $S^1$-equivariant. 

\cite[Proposition 4.14]{nordstromacyldeformations} says that $DF$ is surjective and \cite[Proposition 4.15]{nordstromacyldeformations} says that the kernel of $DF$ consists of harmonic forms. By Proposition \ref{generals1invardetails} it follows that the kernel is the same and the derivative remains surjective when we pass to the $S^1$-invariant case.   

This proves that the $S^1$-invariant moduli space of Calabi-Yau structures on $N \times S^1$ is a smooth manifold locally diffeomorphic to that for all Calabi-Yau structures on $N \times S^1$. 

Only a subspace of these Calabi-Yau structures might arise as limits of a torsion-free $G_2$ structure: we have to check that this subspace is still a manifold. In the non-$S^1$-invariant case, this is \cite[Proposition 6.2]{nordstromacyldeformations}. Again, the proof is that structures arising as limits correspond to the kernel of two nonlinear maps (taken consecutively). Note that the tangent space to the Calabi-Yau moduli space already consists of harmonic forms, so this hypothesis of Proposition \ref{generals1invardetails} does not need checking. The nonlinear maps concerned are composites of the wedge product, orthogonal projections to $S^1$-invariant subspaces determined by the base Calabi-Yau structure, and the orthogonal projection to the complement of such an $S^1$-invariant subspace with respect to the metric induced by the Calabi-Yau structure we consider. These are clearly all $S^1$-equivariant, and so Proposition \ref{generals1invardetails} shows that the derivatives remain surjective, so that the subspace of the $S^1$-invariant moduli space corresponding to limits of $S^1$-invariant torsion-free $G_2$ structures is indeed a submanifold.

We now must pass to the full asymptotically cylindrical setting. We must, here, restrict to asymptotically cylindrical structures and diffeomorphisms with a fixed decay rate $\delta >0$ to define our Banach spaces, as in \cite{nordstromacyldeformations}. 

\label{s1invarg2slice}
The slice we take in the full asymptotically cylindrical setting is given in \cite[Proposition 6.11]{nordstromacyldeformations}. We restrict to the subspace of asymptotically translation invariant closed $3$-forms with suitable limits, and in particular to vector fields (defining diffeomorphisms) whose limits are Killing fields for the limit structure. We then take $E$ to be the subspace satisfying $d^* \alpha \in \Omega^2_{14}$ as in Proposition \ref{cptslice}. Then we again have that $E$ and the space of Lie derivatives are orthogonal complements. By Proposition \ref{generals1invarslice}, it follows that this orthogonal splitting is preserved when we pass to the $S^1$-invariant setting. 

The final map we consider is $F$ of \cite[Definition 6.13]{nordstromacyldeformations}: $F(\phi) = P(*_0 * \phi)$ exactly as in Proposition \ref{cptdetails}; by \cite[Proposition 6.15]{nordstromacyldeformations} the kernel of $F$ is precisely the torsion-free $G_2$ structures. Exactly as in Corollary \ref{cor:cpts1invardetails}, $F$ is $S^1$-equivariant. \cite[Proposition 6.17]{nordstromacyldeformations} says that the kernel of the derivative, when we restrict to the slice, consists exactly of harmonic forms, and the derivative is surjective. Moreover, the kernel is all the harmonic forms with suitable limits, and by construction every nonzero limit arises as a limit of a harmonic form,  and decaying forms orthogonal to harmonic forms form an $S^1$-invariant complement for the kernel. Consequently by Proposition \ref{generals1invardetails} the derivative has the same kernel and is surjective in the $S^1$-invariant setting too. 

We have now essentially proved
\begin{thm}
\label{mg2s1mfd}
In both the compact and asymptotically cylindrical cases $\M_{G_2}^{S^1}$ is a manifold and is locally diffeomorphic to $\M_{G_2}$.
\end{thm}
The remaining parts of this proof are passing from the reduced-regularity and fixed-decay-rate Banach spaces to smooth and all asymptotically cylindrical structures: $S^1$-invariance is irrelevant to these, which thus follow exactly as in \cite{hitchinthreeforms} and \cite{nordstromacyldeformations}. 

Note that the map $\M_{G_2}^{S^1} \to \M_{G_2}$ is not known to be an inclusion map as it need not be globally injective.
\begin{rmk}
It is tempting to try to prove Theorem \ref{mg2s1mfd} more directly. We outline this ``more direct" proof in the compact case, and explain where it runs into difficulty. We know from Theorem \ref{g2slicetheorem} that the slice neighbourhood for $\M_{G_2}$ around an $S^1$-invariant $G_2$ structure $\phi_0$ must consist of $S^1$-invariant $G_2$ structures, because we know its representatives are preserved by the isometry $\Theta$ of $\phi_0$, and thus we always have a local continuous map $\M_{G_2} \to \M_{G_2}^{S^1}$. As $\Diff_0^{S^1} \subset \Diff_0$, and an $S^1$-invariant $G_2$ structure is a $G_2$ structure, we always also have a well-defined continuous map $\M_{G_2}^{S^1} \to \M_{G_2}$ as in the last proof. If we could show that locally these were inverse to each other, and so defined a homeomorphism, we would immediately get the manifold structure on $\M_{G_2}^{S^1}$. We fix neighbourhoods so that these maps are well-defined, and work entirely on these neighbourhoods (so references to injectivity are only to local injectivity).

It is clear that the composition $\M_{G_2} \to \M_{G_2}^{S^1} \to \M_{G_2}$ is the identity. We need to check that the map $\M_{G_2}^{S^1} \to \M_{G_2}$ is injective; it then follows the other composition is also the identity. We know that every point in the image has a slice representative and so it suffices to check that if $\phi_0$ is in the slice neighbourhood and $\Phi \in \Diff_0$, then $\phi_1 = \Phi^* \phi_0$ is also $\Psi^* \phi_0$ for some $\Psi \in \Diff_0^{S^1}$. 

If we have a curve $\phi_t$ of $S^1$-invariant $G_2$ structures from $\phi_0$ to $\phi_1$ then, at least after passing to a suitable $C^{k, \alpha}$ space, we know by Theorem \ref{g2slicetheorem} that for all $t$, $\phi_t$ is the pullback of an element in the slice and can thus be written as $\Phi_t^* \hat\phi_t$ with $\hat\phi_t$ always in the slice. (In particular, $\hat\phi_1 = \hat\phi_0 = \phi_0$). We know also, as $\hat\phi_t$ is in the slice, that $\hat \phi_t$ is $S^1$-invariant. Theorem \ref{g2slicetheorem} says that isometries of $\hat \phi_t$ are isometries of $\phi_0$, and it follows that $\Phi_t^* \phi_0$ is also $S^1$-invariant, by considering the isometry $\Phi_t \circ \Theta \circ \Phi_t^{-1}$ of $\hat\phi_t$. It then follows from Proposition \ref{s1invdiffeoweak} that $\Phi_t$ is $S^1$-invariant for all $t$; in particular, $\Phi_1$ is $S^1$-invariant, and this proves that the map $\M_{G_2}^{S^1} \to \M_{G_2}$ is (locally) injective. 

Thus this proof reduces to
\begin{claim}
\label{starclaim}
The $S^1$-invariant torsion-free $G_2$ structures form a locally path-connected subset of torsion-free $G_2$ structures.
\end{claim}
Without torsion-freeness this claim is evident because of the openness of $G_2$ structures. The natural thing to do is to take a path of torsion-free $G_2$ structures and average out the rotation, but we cannot take averages because the map defining the torsion is non-linear. Removing the remaining torsion would therefore require some analysis: it should be possible, but is unlikely to be easier than the arguments we have outlined in this subsection.

To complete this extended remark, we will observe where Claim \ref{starclaim} comes in our original proof, or, essentially equivalently, how we have local path-connectedness for all torsion-free $G_2$ structures (before passing to $S^1$-invariant ones). The idea is that we only use the whole set of torsion-free $G_2$ structures as the product of diffeomorphisms and the slice. First, we show that the set of torsion-free $G_2$ structures is locally this product, essentially by using the implicit function theorem on $(\Phi, \phi) \mapsto \Phi^* \phi$; then we use the implicit function theorem again to determine that the torsion-free $G_2$ structures in what's left are also a manifold, and so locally path-connected. Since we already know that the diffeomorphisms are a manifold, so locally path-connected, we know that the set of torsion-free $G_2$ structures is a product of locally path-connected spaces and so locally path-connected. It would perhaps be possible to combine these applications of the implicit function theorem and show Claim \ref{starclaim} directly, by showing that $S^1$-invariant torsion-free $G_2$ structures are themselves a manifold; but the whole point of this ``more direct argument" is to avoid these two applications of the implicit function theorem (which correspond, for instance, to Corollaries \ref{cor:g2cpts1invarslice} and \ref{cor:cpts1invardetails}). 
\end{rmk}
Before turning to the components $\M_{SU(3)}$ and $Z$ of $\M_{G_2}^{S^1}$, we deal with the question promised after Definition \ref{defin:msu3}, of what would happen if we defined our moduli spaces to also quotient by the rescaling action. 
\begin{numrmk}
\label{rescalingmodspaces}
Suppose for simplicity that $M$ is compact; similar arguments will apply in the asymptotically cylindrical case. We know that there is a natural rescaling action on both $SU(3)$ structures and $G_2$ structures. The action induced on $S^1$-invariant $G_2$ structures by rescaling of $SU(3)$ structures is not just rescaling: it maps the $G_2$ structures $a^{\frac32} \Re \Omega + az \wedge \omega$ to $ \Re \Omega + z \wedge \omega$. Consequently, if we quotient by rescaling of $SU(3)$ structures, we have to quotient by this partial rescaling of $S^1$-invariant $G_2$ structures, otherwise Proposition \ref{setproduct} no longer holds. If we also quotient by rescaling the $G_2$ structures, we find in particular that $\Re \Omega + z \wedge \omega$ is identified with $a^{\frac32} \Re \Omega + a^\frac32 z \wedge \omega$ and hence with $\Re \Omega + a^\frac12 z \wedge \omega$; that is, we are also quotienting by rescaling of $Z$. The natural slice to take for $Z$ is to recall that an element $z$ of $Z$ is of the form $L [d\theta] + [v]$ for some $[v] \in H^1(M)$, and merely insist that $L=1$. An easy calculation shows that if $L=1$ and $(\Omega, \omega)$ induces a metric of volume one, then so too does $\Re \Omega + z \wedge \omega$. 

Consequently, we could quotient by these and establish the following analogue of Proposition \ref{setproduct}:
\begin{equation}
\begin{split}
&\{[\Omega, \omega] \in \M_{SU(3)}: \Vol([\Omega, \omega]) = 1\} \times H^1(M) \\=\,& \{[\phi = \Re \Omega + z \wedge \omega] \in \M_{G_2}^{S^1}: \Vol([\phi]) = \Vol([\Omega, \omega]) = 1\}
\end{split}
\end{equation}
Of course, the analogue of Theorem \ref{mg2s1mfd} in this case is completely false: even if we quotient by rescaling of $G_2$ structures, so that we work with $\{[\phi] \in \M_{G_2}: \Vol([\phi]) = 1\}$, $\{[\phi = \Re \Omega + z \wedge \omega] \in \M_{G_2}^{S^1}: \Vol([\phi]) = \Vol([\Omega, \omega]) = 1\}$ must be a proper subspace. Consequently, to prove that this space is smooth we would essentially have to proceed by the same argument as in this subsection and then continue. It is in this sense that we claimed after Definition \ref{defin:msu3} that quotienting by the rescaling action added additional complexity for no practical gain. 
\end{numrmk}

\subsection{The space of twisting classes \texorpdfstring{$Z$}{Z}}
\label{ssec:spacez}
We now know that $\M_{G_2}^{S^1} = Z \times \M_{SU(3)}$ is a manifold; it remains to show that both factors are manifolds and if we take the product manifold structure on the right hand side this identification is a diffeomorphism. We begin with the quotient $Z$ of Definition \ref{defin:spacez}. $Z$ is clearly an open subset of a vector space; the purpose of this subsection is to obtain a description of this vector space intrinsic to $M$. In the compact case, this is straightforward; in the asymptotically cylindrical case, we will use standard asymptotically cylindrical Hodge theory. First of all, we need a standard lemma which we will use later to set up gluing as well. 
\begin{lem}[compare {\cite[(2.5)]{nordstromgluing}}]
\label{endexactness}
Suppose that $\alpha$ is an exponentially decaying (with all derivatives) closed form on the end $N \times (0, \infty)$ of an asymptotically cylindrical manifold $M$ with cross-section $N$. Then there is a form $\eta$ on $N \times (0, \infty)$ such that $d\eta = \alpha|_{N \times (0, \infty)}$. 
\end{lem}
We now claim our intrinsic description of the vector space quotient. This is the only place where we explicitly need Remarks \ref{sunacylconndxs} and its $G_2$ analogue: that all interesting Ricci-flat asymptotically cylindrical manifolds have connected cross-section $N$. 
\begin{lem}
\label{ccfcohom}
Let $M \times S^1$ be compact or asymptotically cylindrical. In the case that $M \times S^1$ is asymptotically cylindrical, suppose further that the cross-section $N$ of $M$ is connected. Then the quotient
\begin{equation}
\label{eq:quotientz}
\frac{\text{closed $S^1$-invariant covector fields $z$ (with $z(\frac{\partial}{\partial t}) \to 0$ exponentially)}}{\text{differentials of $S^1$-invariant (asymptotically translation invariant) functions}}
\end{equation}
is isomorphic to $H^1(M \times S^1) = H^1(M) \times \R$; in particular, if we restrict to $Z$, where the $[d\theta]$ component must be positive, we get the open subset $Z = H^1(M) \times \R_{>0}$.
\end{lem}
\begin{rmk}
Asymptotically cylindrical Hodge theory is well-studied, and Lemma \ref{ccfcohom} is essentially a result between that for bounded harmonic forms and that for arbitrary closed forms with exponential growth, \ie between the generalisations to asymptotically cylindrical manifolds of Propositions 6.13 and 6.18 of Melrose \cite{melrose} -- these say that each of these is given by first cohomology and Lemma \ref{ccfcohom} is that arbitrary closed forms with specified limit over a suitable quotient does as well.
\end{rmk}
\begin{proof}
It is clear that $Z$ corresponds to cohomology classes containing a positive $[d\theta]$ component, so it is enough to show that \eqref{eq:quotientz} is isomorphic to $H^1(M \times S^1)$. In the compact case, \eqref{eq:quotientz} reduces to 
\begin{equation}
\label{eq:cptquotientz}
\frac{\text{closed $S^1$-invariant covector fields $z$}}{\text{differentials of $S^1$-invariant functions}}
\end{equation}
Since a closed $S^1$-invariant covector field is of the form $v + c dt$ where $v$ is a closed $1$-form on $M$ and $c$ is a constant, and an $S^1$-invariant function is just a function on $M$, it is easy to see that \eqref{eq:cptquotientz} is indeed isomorphic to $H^1(M) \oplus \R = H^1(M \times S^1)$. 

In the asymptotically cylindrical case, we will use the following facts of asymptotically cylindrical Hodge theory, for an asymptotically cylindrical manifold $M$ with connected cross-section. 
\begin{enumerate}[i)]
\item $H^1 \cong \H^1_\abs(M) = \H^1_\bdd(M)$, where this isomorphism is given by taking the cohomology class of a bounded harmonic (and so closed) form, $\H^1_\bdd$ is the set of bounded harmonic $1$-forms, and $\H^1_\abs$ is the set of bounded harmonic $1$-forms having the same boundary condition that $v(\ddt) \to 0$ exponentially.
\item An exact decaying $S^1$-invariant $1$-form is the differential of a decaying $S^1$-invariant function.
\end{enumerate}
The first point is simply \cite[Corollary 5.13]{nordstromacyldeformations} with different notation: that says that the map from bounded harmonic $1$-forms to cohomology is an isomorphism, and that no bounded harmonic $1$-form has $dt$ as its limit. The second is easier: it is required to set up the Hodge theory, but here we can do it concretely using Lemma \ref{endexactness}. If $\alpha=df$ is our $1$-form, Lemma \ref{endexactness} says that on the end, we can explicitly find a $g$ with $dg = \alpha$. An examination of the proof shows that $g$ is itself exponentially decaying. On the end, $d(f-g) = 0$, so that $f-g$ is a constant, $c$ say. Then $f-c$ is an exponentially decaying function with $d(f-c) =\alpha$, as required. 

We clearly have a map from closed $S^1$-invariant covector fields $z$ on $M \times S^1$, with $z(\ddt) \to 0$ exponentially, to $H^1(M \times S^1)$. We have to apply (i) and (ii) to show that the induced map from \eqref{eq:quotientz} to $H^1(M \times S^1)$ is a well-defined bijection. Since we quotient by exact forms, it is clearly well-defined. To show that it is injective, we have to show that if $v$ is $S^1$-invariant with appropriate limit and represents the zero cohomology class, then it is the differential of an asymptotically translation invariant $S^1$-invariant function. Since $[v]=0$, and $v$ has appropriate limit, the cohomology class of the limit $[\tilde v] = 0$. Thus $\tilde v = dg$ for some $g$ defined on $N$; we may assume $g$ is $S^1$-invariant, since $v$ is, for instance by averaging $g$ around the circle factor. Then $v-d(\psi g)$, where as in equation \eqref{eq:cutoff} $\psi$ is one for $t$ large and zero for $t$ small, represents the zero cohomology class and has zero limit, so by (ii) we have $v -d(\psi g) = dh$ for some decaying function $h$; since $v-d(\psi g)$ is $S^1$-invariant, we may assume that $h$ is. Then we have $v = d(\psi g+ h)$; $\psi g+h$ is an asymptotically translation invariant $S^1$-invariant function, as required. 

To show surjectivity, by the isomorphism in (i) it is enough to show that every bounded harmonic $1$-form defines a class of \eqref{eq:quotientz}; but, by (i) and Lemma \ref{acylharms1invar}, this is immediate.
\end{proof}
\subsection{Smoothness of the \texorpdfstring{$SU(3)$}{SU(3)} moduli space}
\label{ssec:su3modspacesmooth}
We now turn to the $\M_{SU(3)}$ factor. Since we have it as a subspace of $\M_{G_2}^{S^1}$ by Proposition \ref{topproduct}, and we understand the structure of $\M_{G_2}^{S^1}$ by Theorem \ref{mg2s1mfd}, we have a reasonable knowledge of its structure as a topological space. It remains to understand the $\M_{SU(3)}$ factor as a smooth manifold. We use the projection $\pi_Z$ from $\M_{G_2}^{S^1}$ to $Z$. We will show that $\pi_Z$ is a submersion. Its fibres are precisely $\M_{SU(3)}$'s, and so the implicit function theorem will give a family of manifold structure on $\M_{SU(3)}$. We then have to check that the manifold structure is independent of which fibre we take, and that consequently we indeed have a smooth product; this establishes Theorem A. 

Firstly, we now know that $\M_{SU(3)}(M)$ is locally homeomorphic to a subset of cohomology. 
\begin{prop}
\label{su3cohomrep}
In the compact case, \begin{equation}\label{eq:su3cohomrepcpt}[\Omega, \omega] \mapsto ([\Re \Omega], [\omega]) \in H^3(M) \oplus H^2(M)\end{equation} is a local homeomorphism to its image. In the asymptotically cylindrical case, \begin{equation}\label{eq:su3cohomrepacyl}[\Omega, \omega] \mapsto ([\Re \Omega], [\omega], [\Re \tilde \Omega_2], [\tilde \omega_2]) \in H^3(M) \oplus H^2(M) \oplus H^2(N) \oplus H^1(N)\end{equation} is a local homeomorphism to its image, where $\Re \tilde\Omega_2$ and $\tilde \omega_2$ are, as in Theorem \ref{g2slicetheorem}, the appropriate components of $ \Re \tilde\Omega = \Re \tilde\Omega_1 + dt \wedge \Re \tilde\Omega_2$ and $\tilde \omega = \tilde \omega_1 + dt\wedge \tilde \omega_2$.
\end{prop}
\begin{proof}
We first apply Theorem \ref{mg2s1mfd}, which says that locally $\M_{G_2}^{S^1}$ is homeomorphic to $\M_{G_2}$ and hence to cohomology. Then, in the compact case, the result follows by combining Proposition \ref{topproduct} with the K\"unneth theorem. Specifically, given a point $[\Omega, \omega]$, take a neighbourhood of $[\Re \Omega + d\theta \wedge \omega]$ in $\M_{G_2}^{S^1}$ that is homeomorphic to a neighbourhood in $\M_{G_2}(M\times S^1)$ and so to a neighbourhood in $H^3(M \times S^1)$. Then the map \eqref{eq:su3cohomrepcpt} is given by the composition
\begin{equation*}
\begin{tikzcd}[row sep=0pt]
\M_{SU(3)} \ar{r} &\M_{G_2}^{S^1} \ar{r} &H^3(M \times S^1) \ar{r} &H^3(M) \oplus H^2(M) \\[0pt]
[\Omega', \omega'] \ar[mapsto]{r} & {[\Re \Omega' + d\theta \wedge \omega']} \ar[mapsto]{r} & {[\Re \Omega' + d\theta \wedge \omega']} \ar[mapsto]{r} & ([\Re \Omega'], [\omega'])
\end{tikzcd}
\end{equation*}
Consequently, \eqref{eq:su3cohomrepcpt} is continuous because every individual step is. The inverse can be written in exactly the same way and so is also continuous (the last map being the projection $\M_{G_2}^{S^1} \to \M_{SU(3)}$).

In the asymptotically cylindrical case, the only difficulty is that we have to use the K\"unneth theorem on the cross-section as well. The map from $[\Omega', \omega']$ to $[\phi' = \Re \Omega' + d\theta \wedge \omega']$ is a local homeomorphism to its image exactly as in the compact case. We already know that (with the notation of Theorem \ref{g2slicetheorem}) the map
\begin{equation}
[\phi'] \mapsto ([\phi'], [\tilde\phi'_2]) \in H^3(M\times S^1) \oplus H^2(N\times S^1)
\end{equation}
is a local homeomorphism to its image; finally, the map from $H^3(M\times S^1) \oplus H^2(N\times S^1)$ to $H^3(M) \oplus H^2(M) \oplus H^2(N) \oplus H^1(N)$ is again continuous in both directions, using the K\"unneth theorem for both $H^3(M)$ and $H^2(N)$.
\end{proof}

We now turn to the projection map 
\begin{equation}
\label{eq:pizdef}
\pi_Z: \M_{G_2}^{S^1}(M \times S^1) \to Z
\end{equation}
We will show first that $\pi_Z$ is smooth, then that it is a surjective submersion; the implicit function theorem then implies that the fibres, which are clearly $\M_{SU(3)}$'s, have manifold structures. 

For smoothness, we work locally, and so may assume we have a subset of torsion-free $G_2$ structures (open in some suitable slice). We already know from Proposition \ref{g2su3prop1} that the map taking a $G_2$ structure to the twisting $z$ is smooth (as a map of Fr\'echet spaces). Since the map from $z$ to its cohomology class $[z]$ is linear and continuous, it is evidently smooth. Thus \eqref{eq:pizdef} defines a smooth map between finite-dimensional manifolds. 

To show that $\pi_Z$ is a surjective submersion we will use the following elementary 
\begin{lem}
\label{smoothcurves}
Suppose that $(\Omega, \omega)$ is an (asymptotically cylindrical) Calabi-Yau structure on $M$ and that $z(s)$ is a smooth curve of closed $1$-forms with $[z(s)] \in Z$ for all $s$. Then the curve
\begin{equation}
[\Re \Omega + z(s) \wedge \omega ] \in \M^{S^1}_{G_2}(M \times S^1)
\end{equation}
is smooth.
\end{lem}
\begin{proof}
It is obvious that $\Re \Omega + z(s) \wedge \omega$ is a smooth curve of closed three-forms. The map to $\M^{S^1}_{G_2}$ is given locally by taking certain cohomology classes, by Theorem \ref{g2slicetheorem}. This map is linear and continuous and so is smooth to cohomology classes; hence it is smooth to the image of $\M_{G_2}^{S^1}$ in cohomology. 
\end{proof}

Lemma \ref{smoothcurves} yields
\begin{prop}
\label{surjsubm}
The map $\pi_Z$ of \eqref{eq:pizdef} is a surjective submersion. 
\end{prop}
\begin{proof}
Surjectivity is already done, since $\pi_Z$ is the projection from a product. 

To prove $\pi_Z$ is a submersion, we have to show that given a tangent vector $[y] \in H^1(M \times S^1)$ at $[z] \in Z$, for every $[\phi] \in \pi_Z^{-1}([z])$ there is a tangent vector at $[\phi]$ that maps to $[y]$ under $D\pi_Z$. By Proposition \ref{topproduct}, we know any such $[\phi] = [\Re \Omega + z \wedge \omega]$ for some Calabi-Yau structure $(\Omega, \omega)$ and representative $z$. Pick some representative $y$ for the tangent. $z+sy$ is a smooth curve and, by openness of $Z$, $[z+sy] \in Z$ for $s$ small enough. By Lemma \ref{smoothcurves}, therefore, $\gamma(s) = [\Re \Omega + (z+sy)\wedge \omega]$ is a curve in $\M_{G_2}^{S^1}$ through $[\phi]$. We consider its tangent at $[\phi] = \gamma(0)$.
\begin{equation}
{
\renewcommand{\dot}{\left.\frac{d}{ds}\right|_{0}}
D\pi_Z\left(\dot\gamma\right) = \dot{(\pi\circ\gamma)} = \dot{([s\mapsto z+sy])} = [y]
}
\end{equation}
and so we have a submersion.
\end{proof}
In particular, the implicit function theorem now proves that every $\M_{SU(3)}$ fibre has a smooth structure, possibly different for each fibre. As we already have a topological product, we next  show that all these smooth structures are the same and that the projection map $\M_{G_2}^{S^1} \to \M_{SU(3)}$ is smooth; it is then straightforward to show that we obtain a smooth product.

\begin{prop}
Suppose $[z_1]$ and $[z_2]$ are classes of $Z$. The map $\pi_Z^{-1}([z_1]) \to \pi_Z^{-1}([z_2])$ given using the product structure of Proposition \ref{topproduct} by projection to $\M_{SU(3)}$ and inclusion is a diffeomorphism when these fibres are equipped with their submanifold smooth structures. 
\end{prop}
\begin{proof}
We first prove the case where $M$ is compact. Fix a class $[\phi_1 = \Re \Omega + z_1 \wedge \omega]$ in $\pi_Z^{-1}([z_1])$ and its image $[\phi_2 = \Re \Omega + z_2 \wedge \omega]$ in $\pi_Z^{-1}([z_2])$. Locally around these two, we know by Theorems \ref{mg2s1mfd} and \ref{g2slicetheorem} that $\M_{G_2}^{S^1}$ is locally diffeomorphic to $H^3(M \times S^1)$. Consequently, the fibres $\pi_Z^{-1}([z_1])$ and $\pi_Z^{-1}([z_2])$ are locally submanifolds of $H^3(M \times S^1$). The map in the statement defines a map between these submanifolds, which we want to show is smooth. It suffices to show that there is a smooth map on $H^3(M \times S^1)$ which agrees with this map on $\pi_Z^{-1}[z_1]$. 

Recall from Lemma \ref{ccfcohom} that $Z$ is $\R_{>0} \times H^1(M)$. Suppose that $[z_1] = L_1 [d\theta] + [v_1]$ and $[z_2] = L_2 [d\theta] + [v_2]$. By the K\"unneth theorem, we know that $H^3(M \times S^1) \cong H^3(M) \oplus H^2(M)$. We define a map on $H^3(M \times S^1) \cong H^3(M) \oplus H^2(M)$ by
\begin{equation}
H^3(M) \oplus H^2(M) \ni ([\alpha], [\beta]) \mapsto ([\alpha] + \frac{[v_2]-[v_1]}{L_1} \wedge [\beta], \frac{L_2}{L_1}[\beta])
\end{equation}
This is linear and so certainly smooth. Suppose now that $([\alpha], [\beta]) \in \pi_Z^{-1}([z_1])$ is close to $[\phi_1]$. Then $[\alpha + d\theta \wedge \beta] = [\Re \Omega' + (L_1 d\theta + v_1) \wedge \omega']$. It follows that $[\alpha] = [\Re \Omega' + v_1 \wedge \omega']$ and $[\beta] = [L_1 \omega']$; hence the image of this map is $([\Re \Omega' + v_2 \wedge \omega'], [L_2 \omega']) = [\Re \Omega' + (L_2 d\theta + v_2) \wedge \omega']$. This is precisely the image under the map in the statement, and this proves the result in the compact case.

The asymptotically cylindrical case is very similar: the additional linear map $H^2(N) \oplus H^1(N) \ni ([\tilde \alpha], [\tilde \beta]) \mapsto ([\tilde \alpha], \frac{L_2}{L_1} [\tilde beta])$ for the limit factor behaves identically, and the fact that $\M_{G_2}^{S^1}$ is only diffeomorphic to a submanifold of $H^3(M \times S^1) \oplus H^2(N \times S^1)$ does not affect the argument. 
\end{proof}
We can be more concrete about what the smooth structure on the fibre $\pi_Z^{-1}([z])$ is. To set up our moduli space of gluing data in Proposition \ref{su3ghatcharts}, we will need to know that the moduli spaces have coordinates corresponding to a suitable set of structures (essentially slice coordinates as in the $G_2$ case). 
\begin{prop}
\label{su3slicecoords}
Suppose that $[\Omega, \omega] \in \M_{SU(3)}$, and that $(\Omega, \omega)$ is a Calabi-Yau structure representing it. Then there exists a subset $U$ of Calabi-Yau structures containing $(\Omega, \omega)$ such that $U$ is diffeomorphic to a neighbourhood of $[\Omega, \omega] \in \M_{SU(3)}$, and such that the group of automorphisms $(\Omega', \omega')$ isotopic to the identity is independent of $(\Omega', \omega') \in U$. 
\end{prop}
\begin{proof}
First we write $\phi = \Re \Omega + d\theta \wedge \omega$; $\phi$ is an $S^1$-invariant $G_2$ structure. Consequently, by Theorems \ref{mg2s1mfd} and \ref{g2slicetheorem}, there exists a chart $V$ for $\M_{G_2}^{S^1}$ diffeomorphic to a set of torsion-free $S^1$-invariant $G_2$ structures and such that the group of automorphisms of $\phi'$ isotopic to the identity is independent of $\phi' \in V$. 

By Theorem \ref{g2su3alltold}, such torsion-free $S^1$-invariant $G_2$ structures $\phi'$ are given by a twisting $z'$ and a Calabi-Yau structure $(\Omega', \omega')$. Let $U$ be the set of Calabi-Yau structures 
\begin{equation}
U = \{(\Omega', \omega'): \exists z' \in [d\theta] \text{s.t.\ } \Re \Omega' + z' \theta \wedge \omega' \in V\}
\end{equation}
$U$ is precisely the set $\pi_Z^{-1}([d\theta])$ expressed in the local coordinates provided by $V$: hence $U$ defines a chart for $\M_{SU(3)}$ containing $(\Omega, \omega)$, as needed.

It remains to check that the automorphisms isotopic to the identity don't vary with the Calabi-Yau structure in $U$. We apply the ideas of subsection \ref{ssec:s1invg2modspacesetup}. Suppose that $(\Omega', \omega')$ and $(\Omega'', \omega'')$ are structures in $U'$, and $\Phi$ is an automorphism of $(\Omega', \omega')$ isotopic to the identity. There are $z', z'' \in [d\theta]$ such that $\Re \Omega' + z' \wedge \omega', \Re \Omega'' + z'' \wedge \omega'' \in V$; since $\Phi^* z' - z'$ is exact, by Lemma \ref{changingzlemma} we may find a diffeomorphism $\Psi$ corresponding to a time-$1$ flow in the $\ddth$ direction such that $\Phi \circ \Psi$ is an automorphism of the $G_2$ structure $\Re \Omega' + z' \wedge \omega'$ (clearly isotopic to the identity). Consequently it is an automorphism of $\Re \Omega'' + z'' \wedge \omega''$, and so its $M$ part in the sense of Lemma \ref{s1invdiffstr} is an automorphism of $(\Omega'', \omega'')$. It is easy to see that this $M$ part is precisely $\Phi$, which proves the result.
\end{proof}

Now we have a fixed smooth structure on $\M_{SU(3)}$ we can prove
\begin{prop}
\label{pizsmooth}
The projection map $\M_{G_2}^{S^1} \to \M_{SU(3)}$ is smooth. 
\end{prop}
\begin{proof}
We now know that cohomology classes provide local coordinates for both $\M_{G_2}^{S^1}$ and its submanifold $\M_{SU(3)} = \pi_Z^{-1}([d\theta])$, and therefore it is enough to show the smoothness of the projection map at the level of cohomology classes ($H^3(M)$ in the compact case, and $H^3(M) \oplus H^2(N)$ as in Theorem \ref{g2slicetheorem} in the asymptotically cylindrical case). We take a neighbourhood $U = \M' \times Z'$ with $\M'$ and $Z'$ are both charts, by the fact that $\M^{S^1}_{G_2}$ is a topological product.

For compact manifolds, the projection map becomes
\begin{equation}
[\phi'] = [\Re \Omega' + z' \wedge \omega'] \mapsto [\Re \Omega' + d\theta \wedge \omega']
\end{equation}
that is, it is the addition of $[d\theta-z'] \wedge [\omega']$. We can work on a slice neighbourhood, so $\phi'$ is smooth, and then $[z']$ and $[\omega']$ are smooth by Proposition \ref{g2su3prop1}. In the asymptotically cylindrical case, we note that $\tilde z'(\ddt)$ and $\tilde z(\ddt)$ are zero by the boundary conditions of Definition \ref{defin:twistings}; thus the map on the $H^2(N)$ term is the identity, and certainly smooth.
\end{proof}

Together these yield Theorem A, the culmination of our work on deformations.
\begin{thm}
\label{maintheorema}
\label{smoothproduct}
\label{su3acyldeformation}
\begin{equation}
\M^{S^1}_{G_2} = Z \times \M_{SU(3)}(M)
\end{equation}
where $\M_{SU(3)}(M)$, the Calabi-Yau moduli space, is a manifold and this is a smooth product. 
\end{thm}
\begin{proof}
We now have a single smooth structure on the fibre $\M_{SU(3)}$ such that both projections of the product $\M^{S^1}_{G_2} = \M_{SU(3)} \times Z$ are smooth. We have to show that the combination of these two has an isomorphism for its derivative. 

We work at $[\phi] \in \pi_Z^{-1}([z])$. We know by Proposition \ref{surjsubm} that $D\pi_Z: T \M_{G_2}^{S^1} \to TZ$ is surjective and by construction its kernel is $T\pi_Z^{-1}(z)$. On the other hand, $\pi_{\M_{SU(3)}}: \pi_Z^{-1}([z]) \to \M_{SU(3)}$ is essentially the identity, and so $D\pi_{\M_{SU(3)}}: T\pi_Z^{-1}([z]) \to T\M_{SU(3)}$ is also the identity. It follows immediately that $D\pi_Z \oplus D\pi_{\M_{SU(3)}}: T\M^{S^1}_{G_2} \to TZ \oplus T\M_{SU(3)}$ is an isomorphism, and so that we have a local diffeomorphism for the smooth product structure. Hence, as smoothness is a local property, we have a global smooth product.
\end{proof}
\section{Gluing}
\label{sec:gluing}
We now turn to questions of gluing. The objective of this section is to prove Theorem B (Theorem \ref{maintheoremb}), which states that the gluing map on Calabi-Yau structures induced from the gluing map on $G_2$ structures defines a local diffeomorphism from a moduli space of gluing data to the moduli space of Calabi-Yau structures. 

We must first show that we can induce a gluing map of Calabi-Yau structures from the gluing map of $G_2$ structures. We do this in subsection \ref{ssec:structuregluing}: we analyse the proof of the gluing result for $G_2$ structures found in Kovalev \cite[section 5]{kovalevtwistedconnected} to prove that asymptotically cylindrical $S^1$-invariant $G_2$ structures can be glued to form an $S^1$-invariant $G_2$ structure, and then Theorem \ref{theoremc} (Theorem C) gives us a large family of gluing maps (Theorem \ref{su3structuregluing}). In subsection \ref{ssec:gluingtomod}, we show that this family of gluing maps defines a unique map to the moduli space $\M_{SU(3)}$ of Definition \ref{defin:cptsu3modspace}. In subsection \ref{ssec:gluingmodspacesetup} we set up the moduli space of gluing data. Chiefly we follow \cite{nordstromgluing}, but in a few places this paper was abbreviated from Nordstr\"om's thesis \cite{nordstromthesis} and we need the full version. As this moduli space is induced from the moduli spaces on the asymptotically cylindrical ends, a result analogous to Theorem \ref{smoothproduct} (Theorem A) remains true: this result is Theorem \ref{gluingdataspacesproduct} below. In subsection \ref{ssec:removingtildes}, we then restrict to the data that may be glued, and define the gluing map on the moduli space of gluing data. Finally, in subsection \ref{ssec:finalssec}, we analyse the gluing map of $G_2$ structures in terms of this product structure on the moduli space of gluing data, and identify what deformations in each component correspond to. This analysis enables us to prove Theorem B, by saying that the deformations of $G_2$ gluing data corresponding to deformations of the Calabi-Yau gluing data give Calabi-Yau deformations, but that the deformations corresponding to the twistings do not affect the final Calabi-Yau structure. 
\subsection{Gluing of structures}
\label{ssec:structuregluing}
In this subsection, we show that Calabi-Yau structures can be glued. We briefly review the perturbation argument for $G_2$ gluing. We then show in Theorem \ref{su3structuregluing} that this argument passes to the $S^1$-invariant case, using uniqueness, and so defines a collection of gluing maps for Calabi-Yau structures. 

We can identify suitable pairs of asymptotically cylindrical manifolds. 
\begin{defin}
\label{comdefin:matchingmanifolds}
Suppose that $M_1$ and $M_2$ are manifolds with ends, with corresponding cross-sections $N_1$ and $N_2$. $M_1$ and $M_2$ are said to match with the identification $F$ if we have an orientation-reversing diffeomorphism $F: N_1 \to N_2$. Such an $F$ induces further orientation-preserving maps
\begin{equation}
\label{eq:fextension}
\begin{aligned}
F:N_1 \times S^1 \times (0, 1) &\to N_2 \times S^1 \times (0, 1) & F:N_1 \times  (0, 1) &\to N_2 \times (0, 1)\\
(n, \theta, t) &\mapsto (F(n), \theta, 1-t) & (n, t) &\mapsto (F(n), 1-t))
\end{aligned}
\end{equation}
Fix $T>1$, a ``gluing parameter". In practice $T$ will be taken large enough to provide various analytic estimates. 
Let
\begin{equation}
M_1 \supset M_1^{\tr T} := M_1^{\cpt} \cup N_1 \times (0, T)
\end{equation}
and define $M_2^{\tr T}$ similarly. Using \eqref{eq:fextension}, $F$ defines an orientation-preserving diffeomorphism between $N_1 \times (T-1, T)$ and $N_2 \times (T-1, T)$. Then we consider
\begin{equation}
\label{eq:gluedmanifold}
M^T = \frac{M_1^{\tr T} \cup M_2^{\tr T}}F
\end{equation}
the identification of these two manifolds by $F$. $M^T$ is a closed and oriented manifold. By virtue of our extension of $F$, we also see that if we do the same with $M_1 \times S^1$ and $M_2 \times S^1$ we just get  $M^T \times S^1$. 
\end{defin}
If $F_1$ and $F_2$ are isotopic diffeomorphisms $N_1 \to N_2$ then the manifolds $M^T$ constructed using them are diffeomorphic. Thus our definition depends on the isotopy class of $F$, which is essentially an arbitrary choice: we shall ignore this choice for the time being, though it will re-emerge in Propositions \ref{g2ghatcharts} and \ref{su3ghatcharts}. We may then suppress $F$ and just write $N_1 = N_2$. 

For any $T$ and $T'$,  $M^T$ and $M^{T'}$ are diffeomorphic. $T$ dependence is important when we glue structures, however: we could apply appropriate diffeomorphisms to effectively vary $T$ for the structures without actually varying $T$ for the manifold, but to do so would make the notation more complex for no practical gain. Thus we retain $T$. 

Now, given a pair of structures, they consist of closed forms. We would like to patch them together.
\begin{defin}[{\cite[p.190]{nordstromgluing}}]
\label{comdefin:matchingforms}
Let $M_1$ and $M_2$ be matching manifolds with ends as in Definition \ref{comdefin:matchingmanifolds}, and let $g_1$ and $g_2$ be asymptotically cylindrical metrics on them. Suppose that $\alpha_1$ and $\alpha_2$ are asymptotically translation invariant $p$-forms on $M_1$ and $M_2$ respectively. The diffeomorphism $F$ (extended as in Definition \ref{comdefin:matchingmanifolds}) induces a pullback map $F^*$ from the limiting bundle $\bigwedge^p T^*M_2 |_{N_2}$ to $\bigwedge^p T^*M_1|_{N_1}$. $\alpha_1$ and $\alpha_2$ are said to match if the image of $\tilde \alpha_2$ under $F^*$ is $\tilde \alpha_1$. 

In particular, asymptotically cylindrical Calabi-Yau and $G_2$ structures are said to match if the forms defining them match, and twistings (as in Definition \ref{defin:twistings}) are said to match if they match as forms. 

Let $T$ and $M^T$ be as in Definition \ref{comdefin:matchingmanifolds}. Suppose that $\alpha_1$ and $\alpha_2$ are a pair of matching differential forms, and that $\alpha_1$ and $\alpha_2$ are both closed. Then the limits $\tilde \alpha_i$ are closed on $N_i$, and hence, treated as constants, on the end of $M_i$. 

By Lemma \ref{endexactness} we then have that $\alpha_i - \tilde \alpha_i$ is exact on the end, and so can be written as $d\beta_i$ there. As in equation \eqref{eq:cutoff}, let $\psi_T$ have
\begin{equation}
\psi_T(t) = \begin{cases}1&t>T-1 \\ 0 &t \leq T-2\end{cases}
\end{equation}
and define
\begin{equation}
\alpha'_i = \alpha_i - d(\psi_T \beta_i)
\end{equation}
on the end, and $\alpha_i$ off the end. 

On the overlap of $M_1^{\tr T}$ and $M_2^{\tr T}$, $\alpha'_i = \tilde \alpha_i$, and so the two forms are identified by $F$. Thus they define a global tensor field $\alpha^T$ on $M^T$, and $\alpha^T$ is closed because $\alpha'_i$ are closed. 

Write $\alpha^T = \gamma_T(\alpha_1, \alpha_2)$; that is, $\gamma_T$ is the gluing map giving a closed form on $M^T$ from a closed matching pair.
\end{defin}
Given a pair of matching torsion-free $G_2$ structures, Definition \ref{comdefin:matchingforms} yields a (not necessarily torsion-free) $G_2$ structure $\phi^T$ on $M^T$. By construction, $d\phi^T = 0$ and $d*_{\phi^T}\phi^T$ can be bounded with all derivatives by bounds decaying exponentially in $T$. We can thus perturb $\phi^T$ to find a torsion-free $G_2$ structure.  Proposition \ref{g2structuregluingsetup} and Theorem \ref{g2structuregluinggeneral} below carry out this perturbation. The proposition, which provides the setup, is essentially due to Joyce and the theorem is summarised from Kovalev \cite[section 5]{kovalevtwistedconnected}, though the same result can be obtained by using the work of Joyce. The second paragraph of the theorem is easy to establish from the proof, using lower semi-continuity of the first eigenvalue of the Laplacian in the metric (e.g. \cite[Lemma 5.5]{nordstromacyldeformations}). In Theorem \ref{g2structuregluinggeneral}, we restrict to the case where the seven-manifold is of the form $M^T \times S^1$; the same proof applies for a seven-manifold glued as in Definition \ref{comdefin:matchingmanifolds}. 
\begin{prop}[{\hspace{1sp}\cite[Theorem 10.3.7]{joycebook}}]
\label{g2structuregluingsetup}
Let $X$ be a compact Riemannian seven-manifold whose metric is defined by a closed, but not necessarily torsion-free, $G_2$ structure $\phi$. Let $\la \cdot, \cdot \ra$ be the induced inner product on differential forms. Suppose $\hat \phi$ is a sufficiently small $4$-form such that $d\hat\phi = d*_{\phi}\phi$ (that is, $*_{\phi}\phi - \hat\phi$ is close to $*_\phi \phi$ and closed) and $\eta$ is a sufficiently small $2$-form satisfying a certain equation of the form
\begin{equation}
\label{eq:g2gluingequation}
(dd^* + d^*d)\eta + * d \left(\left(1+ \frac13 \la d\eta, \phi\ra\right)\hat \phi\right) - * dR(d\eta) = 0
\end{equation}
where the remainder term $R$ satisfies $|R(d\eta) - R(d\xi)| \leq \epsilon |d\eta - d\xi|(|\eta| + |\xi|)$ for some constant $\epsilon$. 

Then $\phi + d\eta$ is a torsion-free $G_2$ structure. 
\end{prop}
\begin{thm}
\label{g2structuregluinggeneral}
Let $M^T \times S^1$ be a compact seven-manifold constructed as in Definition \ref{comdefin:matchingmanifolds}, with $\phi^T$ given by gluing asymptotically cylindrical $G_2$ structures as in Definition \ref{comdefin:matchingforms}.  We may choose $\hat\phi^T$ by using the approximate gluing (as in Definition \ref{comdefin:matchingforms}) of the closed forms $*_{\phi_1}\phi_1$ and $*_{\phi_2}\phi_2$ as our closed approximation to $*_\phi \phi$. Then for $T > T_0$ sufficiently large we may find a small $2$-form $\eta$ solving \eqref{eq:g2gluingequation}. $d\eta$ is unique of its size given $\hat \phi^T$. 

Moreover, $\phi^T + d\eta$ can be chosen to be continuous in the structures $\phi_1$ and $\phi_2$ with respect to the extended weighted $C^\infty$ topology defined in Definition \ref{defcktopology}, and $T_0$ can be chosen to be upper semi-continuous in these structures. 
\end{thm}
A straightforward extension of Theorem \ref{g2structuregluinggeneral} yields the theorem that Calabi-Yau structures can be glued.
\begin{thm}
\label{su3structuregluing}
Suppose $M_1$ and $M_2$ are asymptotically cylindrical Calabi-Yau threefolds. Let $(\Omega_1, \omega_1)$ and $(\Omega_2, \omega_2)$ be Calabi-Yau structures on $M_1$ and $M_2$ matching in the sense of Definition \ref{comdefin:matchingforms}, and let $(z_1, z_2)$ be a pair of twistings matching in the same sense. $\phi_1 = \Re \Omega_1 + z_1 \wedge \omega_1$ and $\phi_2 = \Re \Omega_2 + z_2 \wedge \omega_2$ define are $S^1$-invariant torsion-free $G_2$ structures, matching in the same sense. Write $\phi^T$ for the approximate gluing of these torsion-free $G_2$ structures given by Definition \ref{comdefin:matchingforms}. There exists $T_0 > 0$ such that, for all $T> T_0$, $\phi^T$ can be perturbed to give an $S^1$-invariant torsion-free $G_2$ structure on $M^T \times S^1$. In particular, we get a Calabi-Yau structure $(\Omega^T, \omega^T)$ on $M^T$. 

For each choice of matching twistings, this procedure gives a well-defined and continuous map $(\Omega_1, \omega_1, \Omega_2, \omega_2) \mapsto (\Omega^T, \omega^T)$ of Calabi-Yau structures. 
\end{thm}
\begin{proof}
By Propositions \ref{g2su3prop1} and \ref{g2su3prop2}, $\Re \Omega_1 + z_1 \wedge \omega_1$ and $\Re \Omega_2 + z_2 \wedge \omega_2$ are indeed $S^1$-invariant torsion-free $G_2$ structures on $M_i \times S^1$; they obviously match, since taking the real part and the wedge product commute with pullback. The approximate gluing procedure of Definition \ref{comdefin:matchingforms} is clearly invariant under the rotation, and so our approximate gluing $\phi^T$ is $S^1$-invariant. We know by Theorem \ref{g2structuregluinggeneral} that for $T>T_0$ sufficiently large we can perturb $\phi^T$ to a torsion-free $G_2$ structure: we need to check that that structure is $S^1$-invariant. 

To follow the theorem, we need $\hat \phi^T$ to be $S^1$-invariant. If a $G_2$ structure $\phi$ is $S^1$-invariant, then $*_\phi \phi$ is also $S^1$-invariant, because pullback by the isometric rotation $\Theta$ commutes with the Hodge star. Thus, the approximation to $*_\phi \phi$ given by gluing $*_{\phi_1} \phi_1$ and $*_{\phi_2}\phi_2$ is $S^1$-invariant and hence so is $\hat \phi^T$. 

We may now check that the solution $d\eta$, where $\eta$ solves \eqref{eq:g2gluingequation}, is $S^1$-invariant. Since $\phi^T$, and so the metric being used, and $\hat \phi^T$ are both $S^1$-invariant, the operator defining \eqref{eq:g2gluingequation} commutes with $\Theta$, hence if $\eta$ satisfies \eqref{eq:g2gluingequation}, so too does $\Theta^* \eta$. The uniqueness statement then implies that $\Theta^* d\eta = d\Theta^* \eta = d\eta$, \ie that $d\eta$ is $S^1$-invariant. 

Using Propositions \ref{g2su3prop1} and \ref{g2su3prop2} again, $\phi^T + d\eta$ then yields our Calabi-Yau structure $(\Omega^T, \omega^T)$ on $M^T$. 

The claim of continuity on structures is immediate from the claim in Theorem \ref{g2structuregluinggeneral}. 
\end{proof}
We have recovered in more generality the result of Doi--Yotsutani\cite{doiyotsutani}. Their argument proceeds as in the first paragraph of the proof of Theorem \ref{su3structuregluing}, except with slightly more assumptions, to obtain a torsion-free $G_2$ structure on $M^T \times S^1$. They then argue as follows. 
\begin{lem}[{\hspace{1sp}\cite[Lemma 3.14]{doiyotsutani}}]
\label{dylemma314}
Suppose $M^T$ is a simply connected manifold and $M^T \times S^1$ admits a torsion-free $G_2$ structure. Then $M^T$ admits a Ricci-flat K\"ahler metric. 
\end{lem}
\begin{proof}[Sketch of original proof]
Consider the universal cover $M^T \times \R$ of $M^T \times S^1$. $M^T \times \R$ also admits a torsion-free $G_2$ structure, and so a Ricci-flat metric, and by the Cheeger-Gromoll splitting theorem\cite{cheegergromoll1} the metric on $M^T \times \R$ is given by a Riemannian product $N \times \R$. The metric induced on $N$ is Ricci-flat K\"ahler, by holonomy considerations. By classification theory for compact simply connected spin $6$-manifolds, $M^T$ and $N$ are diffeomorphic; hence $M^T$ admits a Ricci-flat K\"ahler metric. 
\end{proof}
Our work makes this argument much more concrete, as well as generalising to the not necessarily simply connected case. We shall assume that the torsion-free $G_2$ structure on $M^T \times S^1$ is $S^1$-invariant; by the proof of Theorem \ref{su3structuregluing}, this assumption requires no further hypotheses on the structures to be glued. 
\begin{proof}[Concrete version of proof of Lemma \ref{dylemma314} if $G_2$ structure is $S^1$-invariant]
Using $S^1$-invariance, we write the torsion-free $G_2$ structure on $M^T \times S^1$ as $\Re \Omega + (Ld\theta + v) \wedge \omega$, where $z = Ld\theta + v$ is a twisting. 

We describe the Riemannian universal cover of $M^T \times S^1$ with the $G_2$ structure $\Re \Omega + (Ld\theta + v) \wedge \omega$. We need to equip the universal cover $M^T \times \R$ with a torsion-free $G_2$ structure, and we take the torsion-free $G_2$ structure $\Re \Omega + (L d\theta) \wedge \omega$, where $\theta$ is the coordinate along $\R$. 

We now need to define a projection $\pi: M^T \times \R \to M^T \times S^1$ such that
\begin{equation}
\label{eq:universalcover}
\pi^* (\Re \Omega + (L d\theta + v) \wedge \omega) = \Re \Omega + Ld\theta \wedge \omega
\end{equation}
Since $M^T$ is simply connected, $b^1(M^T) = 0$ and so we may write $v = df$ for some function $f$ on $M^T$. Define $\pi$ by $(x, \theta) \mapsto (x, [\theta - \frac{f(x)}L])$; it is easy to see that $\pi$ satisfies \eqref{eq:universalcover} and so is a Riemannian covering of $M^T \times S^1$ by $M^T \times \R$; since $M^T$ is simply connected, it is the universal cover. Note that the torsion-free $G_2$ structure $\Re \Omega + (L d\theta) \wedge \omega$ is the product structure that can be obtained by using Cheeger-Gromoll, which shows immediately that $N = M^T$. 

We thus get the Calabi-Yau structure $(\Omega, \omega)$ on $M^T$, and in particular a Ricci-flat K\"ahler metric. 
\end{proof}
The major gain from our concrete approach is that we get an explicit identification of $N$ and $M^T$ (the identity), and an explicit universal cover; in particular, we obtain a relation between the Calabi-Yau structures we glue and the resulting Calabi-Yau structure. %for thesis, put in something about special Lagrangians here. 
\subsection{Gluing to \texorpdfstring{$\M_{SU(3)}$}{moduli space}}
\label{ssec:gluingtomod}
This gluing map is not obviously independent of the twistings $z_1$ and $z_2$. However, we now show that the gluing map is independent of $z_1$ and $z_2$ as a map to the moduli space $\M_{SU(3)}(M^T)$ defined in Definition \ref{defin:cptsu3modspace}: that is, the Calabi-Yau structure may depend on the twistings, but different twistings result in Calabi-Yau structures that are at worst pullbacks of each other. 

We use cohomology. We know that the perturbation made in Theorem \ref{g2structuregluinggeneral} is an exact form, so does not change the cohomology class of the $G_2$ structure. This cohomology class can be decomposed, for instance as in the proof of Proposition \ref{pizsmooth}, into cohomology classes corresponding to the twisting and the Calabi-Yau structure. If it were the case that the cohomology classes corresponding to the Calabi-Yau structure did not change under this perturbation, then since they originally are just given by gluing the Calabi-Yau structures using Definition \ref{comdefin:matchingforms} (which we will prove momentarily), they would be independent of the twistings used. It would then follow from Proposition \ref{su3cohomrep} that the $\M_{SU(3)}$ class is also independent of the twistings used. 

Unfortunately, it is not quite true that the cohomology classes corresponding to the Calabi-Yau structure do not change under the perturbation: whilst it is the case for $[\Re\Omega]$, it is possible we might have to rescale $[\omega]$. In this subsection, we adjust the previous paragraph to provide a correct argument proving the result in a similar way.

We begin with the following simple lemma saying that in cohomology the wedge product of patched forms is the patching of the wedge products.
\begin{lem}
\label{wedgeproductgluing}
Suppose that $\alpha_1$, $\beta_1$ and $\alpha_2$, $\beta_2$ are two pairs of closed matching asymptotically translation invariant forms on $M_1$ and $M_2$, so that $\alpha_1 \wedge \beta_1$ and $\alpha_2 \wedge \beta_2$ is also a pair of closed matching asymptotically translation invariant forms. The approximate gluing of Definition \ref{comdefin:matchingforms} gives well-defined cohomology classes $[\gamma_T(\alpha_1, \alpha_2)]$, $[\gamma_T(\beta_1, \beta_2)]$, and $[\gamma_T(\alpha_1 \wedge \beta_1, \alpha_2 \wedge \beta_2)]$. Then
\begin{equation}
[\gamma_T(\alpha_1 \wedge \beta_1, \alpha_2 \wedge \beta_2)] = [\gamma_T(\alpha_1, \alpha_2)] \wedge [\gamma_T(\beta_1, \beta_2)]
\end{equation}
\end{lem}
\begin{proof}
It is sufficient to show that on $M_1^{\tr T}$ the cutoff of $\alpha_1 \wedge \beta_1$ and the wedge product $(\text{cutoff of } \alpha_1 ) \wedge (\text{cutoff of } \beta_1 ) $ differ by $d\gamma_1$ where $\gamma_1$ is supported away from the identified region of $M_1^{\tr T}$, for then when we identify we get
\begin{equation}
\gamma_T(\alpha_1 \wedge \beta_1, \alpha_2 \wedge \beta_2) - \gamma_T(\alpha_1, \alpha_2)\wedge \gamma_T(\beta_1, \beta_2) = d\gamma_1 + d\gamma_2
\end{equation}
with $\gamma_1$ and $\gamma_2$ having disjoint support and the result follows.

We will therefore drop the subscripts. As in Definition \ref{comdefin:matchingforms}, divide $\alpha$ and $\beta$, on the end, into a limit part and an exact part
\begin{equation}
\alpha = \tilde \alpha + d\alpha' \qquad \beta = \tilde \beta + d\beta'
\end{equation}
We then get
\begin{equation}
\label{eq:alphawedgebeta}
\alpha \wedge \beta = \tilde \alpha \wedge \tilde \beta + d(\alpha' \wedge \tilde \beta + (-1)^{\deg \alpha} \tilde \alpha \wedge \beta' + \alpha' \wedge d\beta')
\end{equation}
To simplify notation, set $\varphi = 1-\psi_T$, so $\varphi= 0$ for $t>T-1$ and $1$ for $t\leq T-2$ ($\psi_T$ is as defined in equation \eqref{eq:cutoff} and used in Definition \ref{comdefin:matchingforms}). The cutoffs of $\alpha$ and $\beta$ are $\tilde \alpha + d(\varphi \alpha')$ and $\tilde \beta + d(\varphi \beta')$ and it follows that the wedge product of the cutoffs is given exactly by replacing $\alpha'$ and $\beta'$ with $\varphi \alpha'$ and $\varphi \beta'$ in \eqref{eq:alphawedgebeta}. Similarly the cutoff of the wedge product is given by introducing a $\varphi$ into the exterior derivative in \eqref{eq:alphawedgebeta}.
On taking the difference, all terms but the last then cancel, to give
\begin{equation}
\label{eq:wedgegluingdifference}
\begin{split}
&d(\varphi \alpha' \wedge d \beta' - \varphi \alpha' \wedge d (\varphi \beta'))\\
={}&d(\varphi \alpha' \wedge d\beta' - \varphi^2 \alpha' \wedge d\beta' - \varphi \alpha' \wedge d\varphi \wedge \beta' )
\end{split}
\end{equation}
Since $\varphi -\varphi^2 = \varphi(1-\varphi)$ and $\varphi d\varphi$ are both supported in $(T-2, T-1)$, \eqref{eq:wedgegluingdifference} is supported in $(T-2, T-1)$; that is, away from the identified region. Thus we have the claimed result.
\end{proof}
Combining Lemma \ref{wedgeproductgluing} with standard results on the cohomology ring of a compact K\"ahler manifold, we obtain our result on how the cohomology classes $[\Re\Omega]$ and $[\omega]$ differ from the approximate gluings of the $\Re\Omega_i$ and $\omega_i$. 
\begin{prop}
\label{gluingcohombehaves}
Suppose that $\Re \Omega_1 + (L d\theta + v_1) \wedge \omega_1$ and $\Re \Omega_2 + (L d\theta +v_2) \wedge \omega_2$ are matching torsion-free asymptotically cylindrical $S^1$-invariant $G_2$ structures obtained from torsion-free asymptotically cylindrical $(\Omega_i, \omega_i)$ and twistings $L_i d\theta + v_i$ by Propositions \ref{g2su3prop1} and \ref{g2su3prop2}. (Note that $L_1 = L_2$ since these two $G_2$ structures match.)

Suppose that for some $T$ these glue as in Theorem \ref{su3structuregluing} to the $S^1$-invariant torsion-free $G_2$ structure $\phi$ and we have $\phi = \Re \Omega + (L' d\theta + v) \wedge \omega$ for a Calabi-Yau structure $(\Omega, \omega)$ and twisting $L' d\theta + v$. Then there exists $c>0$ such that
\begin{equation}\begin{split}
L' = cL \qquad [\Re \Omega] = [\gamma_T(\Re \Omega_1, \Re \Omega_2)] \\ [\omega] = \frac1c[\gamma_T(\omega_1, \omega_2)] \qquad [v] = c [\gamma_T(v_1, v_2)]
\end{split}\end{equation}
\end{prop}
\begin{proof}
Removing the torsion does not affect the cohomology class of $\phi$, so we have
\begin{equation}
[\gamma_T(\Re \Omega_1 + (L d\theta + v_1) \wedge \omega_1, \Re \Omega_2 + (L d\theta + v_2) \wedge \omega_2)] = [\Re \Omega + (L' d\theta + v) \wedge \omega]
\end{equation}
Using Lemma \ref{wedgeproductgluing} and the obvious linearity of $\gamma_T$, we obtain
\begin{equation}
[\gamma_T(\Re \Omega_1, \Re \Omega_2)] + [L d\theta + \gamma_T(v_1, v_2)] \wedge [\gamma_T(\omega_1, \omega_2)] = [\Re \Omega] + [L' d\theta + v] \wedge [\omega]
\end{equation}
Since $\gamma_T$ can be defined on any pair of matching asymptotically cylindrical manifolds and commutes with matching maps of such pairs (provided the cutoff functions are chosen appropriately), and we can choose inclusions so that $(M_1, M_2) \into (M_1 \times S^1, M_2 \times S^1)$ is such a pair, $[\gamma_T(\Re \Omega_1, \Re \Omega_2)]$, $[\gamma_T(v_1, v_2)]$, and $[\gamma_T(\omega_1, \omega_2)]$ are in the subset $H^* (M^T)$ of $H^*(M^T\times S^1)$ (corresponding to having no $d\theta$ terms). Evidently, $[\Re\Omega]$, $[v]$, and $[\omega]$ also lie in the subset $H^* (M^T)$. 

By the K\"unneth theorem, therefore, we have
\begin{equation}
L [\gamma_T(\omega_1, \omega_2)] = L' [\omega]
\end{equation}
Recalling that $L$ and $L'$ are both positive, set $c = \frac{L'}L$. Hence $[\omega] = \frac1c [\gamma_T(\omega_1, \omega_2)]$. 

Taking the other component from the K\"unneth isomorphism, and writing $[\gamma_T(\omega_1, \omega_2)]$ as $c[\omega]$, we have
\begin{equation}\label{eq:star}
[\gamma_T(\Re \Omega_1, \Re \Omega_2)] + [\gamma_T(v_1, v_2)] \wedge c[\omega] = [\Re \Omega] + [v] \wedge [\omega]
\end{equation}
Now $[\omega]$ is a K\"ahler class on $M^T$ and we have 
\begin{equation}
\label{eq:prim3cohom}
[\omega] \wedge [\Re \Omega] = 0 = [\gamma_T(\Re \Omega_1 \wedge \omega_1, \Re \Omega_2 \wedge \omega_2)] = [\gamma_T(\Re \Omega_1, \Re \Omega_2)] \wedge c[\omega]
\end{equation}
again using Lemma \ref{wedgeproductgluing}. Since $c \neq 0$, \eqref{eq:prim3cohom} means that $[\Re \Omega]$ and $[\gamma_T(\Re \Omega_1, \Re \Omega_2)]$ are classes of primitive $3$-cohomology. The remaining two equations now follow from the Lefschetz decomposition as the primitive $3$-cohomology and $1$-cohomology components of \eqref{eq:star}. 
\end{proof}
We now prove that the constant $c$ of Proposition \ref{gluingcohombehaves} doesn't change as we change the twistings. It is clear that if $\Omega$ is known, $c$ is determined for each $\omega$ by condition iii) of Definition \ref{defin:sunstructure}:
\begin{equation}
\Re \Omega \wedge \Im \Omega = \frac{(-2i^{n+1})^{n-1}}{n!} \omega^n
\end{equation}
and that if $[\Omega]$ is known, exactly the same applies for each $[\omega]$. Using torsion-freeness we can say a little more, passing to $[\Re \Omega]$. 
\begin{lem}
\label{ccontrol}
Let $(\Omega, \omega)$ be a Calabi-Yau structure. There is an open neighbourhood $U$ of $(\Omega, \omega)$ in Calabi-Yau structures, and $\epsilon>0$, such that if $(\Omega_1, \omega_1)$ and $(\Omega_2, \omega_2)$ both lie in $U$, with $[\Re \Omega_1] = [\Re \Omega_2]$ and $[\omega_1] = C[\omega_2]$ for $|C-1|<\epsilon$, then $C=1$. 
\end{lem}
\begin{proof}
We work locally around the class of $\M_{G_2}^{S^1}$ corresponding to $\Re \Omega + d\theta \wedge \omega$. We know that there is an open subset $V$ of $\M_{G_2}^{S^1}$ around this class which is homeomorphic to an open subset of $H^3(M \times S^1)$ as in Theorem \ref{g2slicetheorem} and Theorem \ref{mg2s1mfd}. By reducing the open set if necessary, we may assume using Proposition \ref{topproduct} that $V$ is a product of open sets $U'$ and $W$ in $\M_{SU(3)}$ and $Z$ respectively. Let $U$ be the set of structures whose moduli class is in $U'$. $W$ is an open set containing $[d\theta]$ so it contains an interval of the line $\R [d\theta]$; choose $\epsilon<\frac12$ so that it contains $(1-2\epsilon, 1+2\epsilon) [d\theta]$. 

Now suppose given two structures $(\Omega_i, \omega_i)$ as in the statement. Clearly, the $S^1$-invariant torsion-free $G_2$ class $[\Re \Omega_1 + d\theta \wedge \omega_1]$ lies in $V$, and since $|\frac1C -1| < \frac{\epsilon}{1-\epsilon} < 2\epsilon$, the same is true of the class $[\Re \Omega_2 + \frac1Cd\theta \wedge \omega_2]$. Moreover, we have the equality
\begin{equation}
[\Re \Omega_1 + d\theta \wedge \omega_1] = [\Re \Omega_2 + \frac1Cd\theta \wedge \omega_2]
\end{equation}
in $H^3$. Since $V$ is homeomorphic to its image in $H^3(M \times S^1)$, we also have equality in $\M_{G_2}^{S^1}$. Applying Proposition \ref{topproduct} again, we find that $C=1$.
\end{proof}
\begin{rmk}
A natural question for further study is whether the constant $c$ of Proposition \ref{gluingcohombehaves} is necessarily one. This would mean that any such gluing of $S^1$-invariant $G_2$ structures does not fundamentally alter the length of the circle factor, and seems natural if the gluing of the Calabi-Yau structures can be done without reference to the $G_2$ structures. On the other hand, if we regard $c$ instead as a possible rescaling of the holomorphic volume form $\Omega$, it is not obvious that the scaling of the holomorphic volume form should be preserved by gluing. Changing how we regard $c$ would superficially affect much of the below, as $c$ would have to be controlled in different places, but would not make it substantially different. An interesting preliminary question would be whether the scaling of the holomorphic volume form is uniquely determined by the cohomology, that is, whether Lemma \ref{ccontrol} is true globally. 
\end{rmk}
We now use Lemma \ref{ccontrol} to prove our foreshadowed well-definition result.
\begin{prop}
\label{weakwelldef}
Let $(\Omega_1, \omega_1)$ and $(\Omega_2, \omega_2)$ be matching Calabi-Yau structures; let $z_1$ and $z_2$ be two matching twistings, and $z'_1$ and $z'_2$ be another matching pair of twistings. Then, for neck-length parameter $T$ sufficiently large (depending on a curve of $G_2$ structures that will appear in the proof), the $\M_{SU(3)}$ parts of the results of gluing the $G_2$ structures $\Re \Omega_i + z_i \wedge \omega_i$ and $\Re \Omega_i + z'_i \wedge \omega_i$ are equal, and the $Z$ parts are given by the same multiple of the approximate gluing, though the approximate gluing may be different in the two cases. 
\end{prop}
\begin{proof}
The set of matching pairs of twistings is precisely $\R_{>0}$ times the vector space of matching closed $1$-forms (with appropriate limits) on $M_1$ and $M_2$. Therefore it is path-connected, and so there exists a path in it $(z_1(s), z_2(s))$ with $z_i(0) = z_i$ and $z_i(1) = z'_i$. 

For $T$ sufficiently large (by semi-continuity of a minimal $T$ from Theorem \ref{g2structuregluinggeneral} and compactness of $[0, 1]$), the resulting pairs of matching torsion-free $S^1$-invariant $G_2$ structures
\begin{equation}
(\Re \Omega_1 + z_1(s) \wedge \omega_1, \Re \Omega_2 + z_2(s) \wedge \omega_2)
\end{equation}
can be glued to give torsion-free $S^1$-invariant $G_2$ structures $\phi(s)$; $\phi(s)$ is a continuous curve, again by Theorem \ref{g2structuregluinggeneral}. Using the proof of Proposition \ref{g2su3prop1}, we can split these up as
\begin{equation}
\phi(s) = \Re \Omega(s) + z(s) \wedge \omega(s)
\end{equation}
for continuous curves of Calabi-Yau structures and twistings. In the notation of Proposition \ref{gluingcohombehaves}, it follows that $L'$ is continuous and so $c$ is (because these are determined from $z(s)$); hence we know that the cohomology classes satisfy
\begin{align}
[\Re \Omega(s)] &= [\gamma_T(\Re \Omega_1, \Re \Omega_2)]\\
[\omega(s)] &= \frac1{c(s)}[\gamma_T(\omega_1, \omega_2)]
\end{align}
for a continuous positive function $c(s)$. In particular, we see that $(\Omega(s), \omega(s))$ is a continuous curve of Calabi-Yau structures with $[\Re \Omega(s)]$ fixed and $[\omega(s)]$ only varying in a line. It follows from Lemma \ref{ccontrol} that $c(s)$ is locally constant, and hence it is constant; thus, $[\Re \Omega(s)]$ and $[\omega(s)]$ are both fixed, and so so is the moduli class of $(\Omega(s), \omega(s))$, which proves the first claim.  

The second claim follows since we also have $[z(s)] = c(s) [\gamma_T(z_1(s), z_2(s)]$ from Proposition \ref{gluingcohombehaves}.
\end{proof}
Proposition \ref{weakwelldef} essentially says that the images under gluing of a pair of pairs of $G_2$ structures differing just by varying the twistings themselves just differ by varying the twisting and potentially diffeomorphism. 

In subsection \ref{ssec:finalssec}, we shall analyse the gluing map between moduli spaces for $G_2$ closely to work out how the gluing map between Calabi-Yau moduli spaces behaves. Proposition \ref{weakwelldef} will be used to say that a variation corresponding to a twisting glues to a variation corresponding a twisting: we would like to know what happens when we vary the Calabi-Yau structure or the gluing parameter $T$.

In varying the Calabi-Yau structure, there are two complications over varying the twisting. Firstly, it is not at all clear that $\M_{SU(3)}$ is connected, so we will need to assume the existence of the curve used in the proof of Proposition \ref{weakwelldef}. In any case, the factor $c$ may vary. 
\begin{prop}
\label{su3structurebehaves}
Suppose that $z_1$ and $z_2$ are a pair of matching twistings. Let  $(\Omega_1, \omega_1)$ and $(\Omega_2, \omega_2)$ be a pair of matching Calabi-Yau structures, and let $(\Omega'_1, \omega'_1)$ and $(\Omega'_2, \omega'_2)$ be another such pair. Suppose that there exists a continuous curve through matching pairs of Calabi-Yau structures joining these pairs.  Then, for neck-length parameter sufficiently large (depending on this curve), the $Z$ parts of the results of gluing the $G_2$ structures $\Re \Omega_i + z_i \wedge \omega_i$ and $\Re \Omega'_i + z_i \wedge \omega'_i$ are proportional.
\end{prop}
\begin{proof}
By choosing $T$ large, as in the proof of Proposition \ref{weakwelldef}, we get a continuous curve of glued structures; write them as $\Omega(s) + z(s) \wedge \omega(s)$. By Proposition \ref{gluingcohombehaves}, we know that $[z(s)] = c(s)[\gamma_T(z_1, z_2)]$. Hence $[z(0)] = \frac{c(0)}{c(1)}[z(1)]$. 
\end{proof}

The only remaining question is the effect of varying the neck-length parameter $T$: again we obtain
\begin{prop}
\label{lengthstructurebehaves}
Suppose $(\Omega_i, \omega_i)$ are matching Calabi-Yau structures and $z_i$ are matching twistings. Let $T$ and $T'$ be a pair of positive reals exceeding the minimal gluing parameter $T_0$ for the associated matching $G_2$ structures $\Re \Omega_i + z_i \wedge \omega_i$. Then as in Proposition \ref{su3structurebehaves} the $Z$ parts of the glued structures are proportional. 
\end{prop}
\begin{proof}
Choose a curve $T(s)$ from $T$ to $T'$, always greater than $T_0$. As in Propositions \ref{weakwelldef} and \ref{su3structurebehaves}, write the curve of glued structures as $\Omega(s) + z(s) \wedge \omega(s)$ and use Proposition \ref{gluingcohombehaves} to get $[z(s)] = c(s)[\gamma_{T(s)}(z_1, z_2)]$.

The statement follows as in Proposition \ref{su3structurebehaves} if $[\gamma_{T(s)}(z_1, z_2)]$ is independent of $s$. Because the common limit of $z_1$ and $z_2$ has no $dt$ term, the natural diffeomorphism pulls back the gluing with a large $T$ to the smaller $T$ with only a compactly supported error (see \cite[Proposition 3.2]{nordstromgluing}), and so $[\gamma_{T(s)}(z_1, z_2)]$ is indeed independent of $s$. 
\end{proof}
\subsection{Moduli spaces of gluing data}
\label{ssec:gluingmodspacesetup}
By combining Theorem \ref{su3structuregluing} with Proposition \ref{weakwelldef}, we have thus shown that there is a single well-defined gluing map from matching pairs of Calabi-Yau structures to the moduli space $\M_{SU(3)}(M^T)$. We now define a moduli space of gluing data and show that this gluing map induces a well-defined map between these moduli spaces. For the definition, we follow the ideas and notation for the $G_2$ case in \cite{nordstromgluing}. Here, Nordstr\"om restricts to the special case in which the first Betti number of the glued manifold is zero for simplicity, though the result is true in general. In our case, $b^1(M^T \times S^1)$ is clearly nonzero, and though we could similarly argue for the special case when $b^1(M^T) = 0$, we will follow the full generality analysis provided by Nordstr\"om in \cite[subsection 6.3.2]{nordstromthesis} in the relevant place. 

In this subsection, we define a quotient which we expect to define a sensible space of gluing data, and show that this quotient is a manifold. The idea here, which is used in \cite{nordstromgluing}, is to use a sequence of larger and larger spaces, and show each in turn is a manifold. The smallest is the space $\B$ of ``matching moduli classes"; the second is the space $\hat \G$ of ``moduli classes of matching pairs", and finally we end up with the space $\tilde \G$ of ``moduli classes of matching pairs and gluing parameters" which we require. 

We first review the definitions and results in the $G_2$ case: these pass to the $S^1$-invariant $G_2$ case with very little additional work. We show that $\tilde \G_{G_2}^{S^1}$ is a principal $\R$-bundle over $\hat \G_{G_2}^{S^1}$, and that $\hat \G_{G_2}^{S^1}$ is a bundle over $\B_{G_2}^{S^1}$, which defines its coordinates (Proposition \ref{g2ghatcharts}). We provide some detail of the proof that $\hat \G_{G_2}^{S^1}$ is a bundle over $\B_{G_2}^{S^1}$, as this material is not available in \cite{nordstromgluing} and is not fully given even in \cite{nordstromthesis}. 

We then pass simultaneously to the analogous spaces for Calabi-Yau structures and the relationship between the $G_2$ and Calabi-Yau cases. We show that the analogous spaces for Calabi-Yau structures are smooth and the inclusion maps from the Calabi-Yau versions to the $G_2$ versions are smooth. For smoothness of the spaces, we use both similar methods to the $G_2$ case and what we already known about the relationship between $G_2$ and Calabi-Yau structures. The final theorem of this subsection (Theorem \ref{gluingdataspacesproduct}) is a gluing-data version of Theorem A (Theorem \ref{maintheorema}), saying that the space of ``$S^1$-invariant $G_2$ gluing data" is a product of the space of ``Calabi-Yau gluing data" with a suitable space of twistings in the sense of Definition \ref{defin:twistings}. 

We now begin by summarising the $G_2$ case, with the minor changes required to make the results $S^1$-invariant. We will make some of the definitions in greater generality, however, as otherwise we would have to make exactly parallel definitions in the Calabi-Yau case. 

To define a moduli space of gluing data, we need to define an action on matching structures by matching pairs of diffeomorphisms. The type of the structure is irrelevant here. Consequently, we shall write $\upsilon$ for the structures, which we shall use generally in this subsection when giving an argument that applies in both cases.  We recall from Definition \ref{acyldiffeo} that the limit of an asymptotically cylindrical diffeomorphism $\Phi$ is a pair $(\tilde\Phi, L)$ such that the diffeomorphism decays to $(n, t) \mapsto (\tilde\Phi(n), t+L)$. 
\begin{defin}[{cf. \cite[Definition 2.3]{nordstromgluing}}]
\label{comdefin:matchingdiffeo}
Suppose $M_1$, $M_2$, and $F$ are as in Definition \ref{comdefin:matchingmanifolds}. Suppose that $\Phi_1$ and $\Phi_2$ are asymptotically cylindrical diffeomorphisms of $M_1$ and $M_2$ with limits $(\tilde\Phi_i, L_i)$. They are said to match if $F^{-1}\tilde\Phi_2F = \tilde\Phi_1$ as a map $N_1 \to N_1$. Note that we do not require $L_1 = L_2$. 

A matching pair $(\Phi_1, \Phi_2)$ is isotopic to the identity as a matching pair if $\Phi_1$ and $\Phi_2$ are both asymptotically cylindrically isotopic to the identity as in Definition \ref{defin:acylisotopy} and we may choose the isotopies $\Phi_{1, s}$ and $\Phi_{2, s}$ such that the pair of diffeomorphisms $(\Phi_{1, s}, \Phi_{2, s})$ matches for all $s$. As in Definition \ref{defin:acylisotopy}, we shall simply speak of a pair of diffeomorphisms being isotopic to the identity. 

We define an action of matching pairs $(\Phi_1, \Phi_2)$ with limits $(\tilde \Phi_i, L_i)$ on triples $(\upsilon_1, \upsilon_2, T)$ where $\upsilon_1$ and $\upsilon_2$ are (asymptotically cylindrical Calabi-Yau or, potentially $S^1$-invariant, $G_2$) structures on $M_1$ and $M_2$ respectively and $T$ is a real number, which will eventually be the gluing parameter, by
\begin{equation}
\label{eq:matchingaction}
(\Phi_1, \Phi_2)(\upsilon_1, \upsilon_2, T) =  (\Phi_1^*\upsilon_1, \Phi_2^*\upsilon_2, T-\frac12(L_1+L_2))
\end{equation}
\end{defin}
\eqref{eq:matchingaction} is clearly an action on triples of structures. The first thing to show is that \eqref{eq:matchingaction} preserves the subspace of triples where the structures match as in Definition \ref{comdefin:matchingforms}, but this is obvious from the definition of matching diffeomorphisms. Therefore we make
\begin{defin}
\label{defin:g2gtilde}
Let 
\begin{equation}
\tilde \G_{G_2}^{S^1} = \frac{\text{matching pairs of torsion-free $S^1$-invariant $G_2$ structures and parameters $T$}}{\text{matching pairs of $S^1$-invariant diffeomorphisms isotopic to the identity}}
\end{equation}
$S^1$-invariant diffeomorphisms are as defined in Definition \ref{defin:s1invdiffeo}, for consistency of our moduli spaces. 
\end{defin}
The orbit of a matching pair of $G_2$ structures under matching pairs of diffeomorphisms isotopic to the identity is closed: by Theorem \ref{g2slicetheorem}, the orbit of a $G_2$ structure under diffeomorphisms isotopic to the identity is closed (and in fact the diffeomorphisms converge to a diffeomorphism giving the new point of the orbit), so we only have to check that a pair of diffeomorphisms being a matching pair isotopic to the identity is a closed condition, which is obvious. Hence, $\tilde G_{G_2}^{S^1}$ is Hausdorff. The same applies for Calabi-Yau structures, combining Theorem \ref{g2slicetheorem} with Theorem A (Theorem \ref{maintheorema}) and using the closedness of exact forms.

We consider two additional spaces of gluing data, both of which are smaller than $\tilde \G_{G_2}^{S^1}$, from which we can construct $\tilde \G_{G_2}^{S^1}$, and hence infer that it is a manifold. First of all, we define the smallest possible space of gluing data: the subspace of the product of the moduli spaces on each part corresponding to matching moduli classes.
\begin{defin}
\label{defin:g2bspace}
Let $\B_{G_2}^{S^1}$ be the space of matching pairs in the $S^1$-invariant $G_2$ moduli spaces on $M_1$ and $M_2$, that is:
\begin{equation}
([\phi_1], [\phi_2]) \in \M_{G_2}^{S^1}(M_1) \times \M_{G_2}^{S^1}(M_2)
\end{equation}
such that there exist representatives $\phi_1$ and $\phi_2$ matching in the sense of Definition \ref{comdefin:matchingforms}. 
\end{defin}
The final space is the space of pairs of matching classes quotiented by matching diffeomorphisms, defined as for $\tilde \G$ but forgetting the parameter $T$. Because we have to deal with two different kinds of structures the notation is already quite involved, so we give it a specific name. The action by matching pairs of diffeomorphisms is just that restricted from Definition \ref{comdefin:matchingdiffeo}.

\begin{defin}
\label{defin:g2ghat}
Let 
\begin{equation}
\hat \G_{G_2}^{S^1} = \frac{\text{matching pairs of torsion-free $S^1$-invariant $G_2$ structures}}{\text{matching pairs of $S^1$-invariant diffeomorphisms isotopic to the identity}}
\end{equation}
\end{defin}
The notation we are adopting is rather different from that of \cite{nordstromgluing}. In that, our $\B$ is called $\M_y$, following a general principle of using $y$ subscripts to denote matching objects; our space $\tilde \G$ is just denoted $(\mathcal{X}_y \times \R)/\D_y$, and our space $\hat \G$ is denoted $\mathcal{X}_y/\D_y$, or $\B$. It is possible that our using $\B$ for a different space may cause confusion; as in \cite{nordstromgluing}, the reason for this notation is that it is the base space of a suitable bundle. 

It is clear that $\tilde \G_{G_2}^{S^1}$ is a principal $\R$-bundle over $\hat \G_{G_2}^{S^1}$ Therefore, by taking the natural smooth structure on such a bundle, it is enough to show that $\hat \G_{G_2}^{S^1}$ is smooth.

The argument is essentially that $\hat \G_{G_2}^{S^1}$ is a manifold because it is a covering space of $\B_{G_2}^{S^1}$. First, therefore, we have to check that $\B_{G_2}^{S^1}$ is a manifold. We have
\begin{prop}[cf. {\cite[Proposition 4.3]{nordstromgluing}}]
\label{g2btheorem}
$\B_{G_2}^{S^1}$ is a smooth manifold. Moreover, around the classes of any matching pair of structures there exist charts for $\B_{G_2}^{S^1}$ consisting of matching pairs of structures.
\end{prop}
Strictly, of course, Nordstr\"om's argument is for the case of $\B_{G_2}$, defined as in Definition \ref{defin:g2bspace} but removing the constraints on $S^1$-invariance; but we have shown that locally $\M_{G_2}^{S^1}$ is an open subset of $\M_{G_2}$, and therefore locally $\B_{G_2}^{S^1}$ is an open subset of $\B_{G_2}$. As open subsets of a submanifold are submanifolds, the $S^1$-invariant result is immediate. The existence of such charts is not part of \cite[Proposition 4.3]{nordstromgluing} but is clear from its proof.

We now proceed to show $\hat\G_{G_2}^{S^1}$ is a manifold. We give details for this proof, as it is not found in full generality in \cite{nordstromgluing} but only in \cite{nordstromthesis}; even there, not all the details are given. The factor by which $\hat \G_{G_2}^{S^1}$ is bigger than $\B_{G_2}^{S^1}$ appears in \cite[p.140]{nordstromthesis} (except of course for not requiring $S^1$-invariance), though we use a somewhat different setup. As it will reappear in the Calabi-Yau case, in Proposition \ref{su3ghatcharts}, we define it separately. It could be defined in general, as it only depends on the Riemannian metric, but to prove its properties we need slice results analogous to those of Theorem \ref{g2slicetheorem} and Proposition \ref{g2sliceneatprop}. These results could be obtained in the Riemannian case from Ebin\cite{ebin}. 

We first need to weaken the notion of isotopic with fixed limit from Definition \ref{defin:acylisotopy}. We will be interested in isotopies $\Phi_s$ such that for some fixed diffeomorphism $\Psi$, $(\Psi, \Phi_s)$ is an isotopy of matching pairs in the sense of Definition \ref{comdefin:matchingdiffeo}, and thus we make
\begin{defin}
\label{defin:isotopywithfixeddiffn}
Suppose that $\Phi$ and $\Psi$ are asymptotically cylindrical diffeomorphisms of an asymptotically cylindrical manifold $M$, isotopic in the sense of Definition \ref{defin:acylisotopy}. An isotopy is a curve $\Phi_s$ of asymptotically cylindrical diffeomorphisms; taking limits, an isotopy gives us a curve $(\tilde \Phi_s, L_s)$ of diffeomorphisms of $N$ and real numbers. $\Phi$ and $\Psi$ are isotopic with fixed $\Diff(N)$ limit if there is an isotopy such that $\tilde \Phi_s$ is independent of $s$. 
\end{defin}
To have fixed $\Diff(N)$ limit is clearly weaker than having fixed limit in the sense of Definition \ref{defin:acylisotopy}, and is precisely what is needed to obtain an isotopy of matching pairs. 
\begin{defin}
\label{defin:spacea}
Suppose that $(M_1, g_1)$ and $(M_2, g_2)$ are matching asymptotically cylindrical Ricci-flat manifolds, with metrics induced from torsion-free $G_2$ or Calabi-Yau structures. Suppose that the cross-section is $N$ and $g_1$ and $g_2$ induce the metric $\tilde g$ on it, suppressing the diffeomorphism $F$ of Definitions \ref{comdefin:matchingmanifolds} and \ref{comdefin:matchingforms}.  Consider the set of diffeomorphisms
\begin{equation}
\label{eq:bigspacea}
\left\{\begin{gathered}\text{diffeomorphisms of $N \times [0, 3]$ of the form $(n, t) \mapsto (f_t(n), t)$ where} \\ \text{$f_t = \id$ for $t \in [0, 1]$, $f_t = f_2$ for $t \in [2, 3]$, and $f_2$ is an isometry of $\tilde g$}\end{gathered}\right\}
\end{equation}
Let $\tilde A(\tilde g)$ be the quotient of this set by isotopies preserving the diffeomorphism on $N \times ([0, 1] \cup [2, 3])$. 

It is clear that a class of $\tilde A(\tilde g)$ induces diffeomorphisms on $M_1$ and $M_2$, and that taking a different representative gives diffeomorphisms that are isotopic with fixed limit in the sense of Definition \ref{defin:acylisotopy}. Consider the subgroup of $\tilde A(\tilde g)$ consisting of classes whose induced diffeomorphisms on $M_1$ are isotopic with fixed $\Diff(N)$ limit to isometries of $M_1$, and the analogous subgroup for $M_2$. Consider the subgroup $G$ generated by these two subgroups, and let $A(g_1, g_2)$ be the quotient $\frac{\tilde A(\tilde g)}{G}$. 

By a minor abuse of notation, when the metric is induced from Calabi-Yau structures we shall write $\tilde A(\tilde \Omega, \tilde \omega)$ and $A(\Omega_1, \omega_1, \Omega_2, \omega_2)$.

By a slightly larger abuse of notation, when the metric is induced from a $S^1$-invariant torsion-free $G_2$ structures, we will write $\tilde A(\tilde \phi)$ and $A(\phi_1, \phi_2)$ to be the sets given by requiring all the diffeomorphisms above to be $S^1$-invariant in the sense of Definition \ref{defin:s1invdiffeo}.
\end{defin}
Elements of $\tilde A(\tilde g)$ give an ``action" on pairs of structures, thus:
\begin{defin}
\label{defin:theamap}
Suppose that $g_1$ and $g_2$ are Ricci-flat metrics on $M_1$ and $M_2$ induced by matching Calabi-Yau or torsion-free $G_2$ structures. Let $\tilde A(\tilde g)$ be as in Definition \ref{defin:spacea}. Suppose that $g'_1$ and $g'_2$ are two other Ricci-flat metrics induced by matching structures. Define an action of $\tilde A(\tilde g)$ on the set of such metrics by
\begin{equation}
\label{eq:theamap}
[\Phi](g_1, g_2) = (g_1, \Phi^* g_2)
\end{equation}
\end{defin}
Note that \eqref{eq:theamap} is not well-defined; however, it is well-defined up to pullback by a matching pair of diffeomorphisms isotopic to the identity, and in practice we will only be using \eqref{eq:theamap} as a map to spaces (such as $\hat\G_{G_2}^{S^1}$) where we have quotiented by pullback by matching pairs isotopic to the identity. 

We hope that $A(\phi_1, \phi_2)$ will define the other smooth component in a local statement ``$\hat \G_{G_2}^{S^1} = A(\phi_1, \phi_2) \times \B_{G_2}^{S^1}$". We thus need to know that $A(\phi_1, \phi_2)$ is smooth. 
\begin{prop}
\label{spaceanice}
With the notation of Definition \ref{defin:spacea}, $\tilde A(\tilde g)$ is a finite-dimensional abelian Lie group. Its tangent space at the class of the identity is the space of Killing fields on $N$. A diffeomorphism defining a class of $\tilde A(\tilde g)$ defines a class of $G$ precisely if its image under the map of Definition \ref{defin:theamap} is in the orbit of $(g_1, g_2)$ by matching pairs isotopic to the identity and  so $G$ is a closed subgroup, with tangent spaces the sums of the subspaces of Killing fields on $N$ that have extensions to Killing fields on $M_1$ and $M_2$. Hence $A(g_1, g_2)$ is also a finite-dimensional manifold, and the map of Definition \ref{defin:theamap} passes to a well-defined map of $A(g_1, g_2)$ on structures up to pullback by matching  pairs of diffeomorphisms isotopic to the identity. 

Locally around the class of the identity, $A(\phi_1, \phi_2)$ as defined is equal to $A(g_1, g_2)$ for the induced metrics. 
\end{prop}
\begin{proof}
To show $\tilde A(\tilde g)$ is a finite-dimensional Lie group, we note that it is equivalent, by careful use of bump functions, to the space of curves in $\Diff_0(N)$ from the identity to isometries of $N$, modulo homotopy with fixed end points. The corresponding group in \cite[Definition 6.3.6]{nordstromthesis} is in fact defined as this space of curves modulo homotopy. As in \cite{nordstromthesis}, standard arguments show that its identity component (in the sense of isotopy through diffeomorphisms in the subset of \eqref{eq:bigspacea}) is the universal cover of the identity component of the isometry group of $N$. The identity component of the isometry group is a compact Lie group, with Lie algebra given by the Killing fields. Hence, the space of Killing fields is the universal cover, and by considering each component in turn we consequently have that the whole of $\tilde A(\tilde g)$ is a finite-dimensional Lie group, with tangent space at the class of the identity given by the Killing fields on $N$. 

We now have to show that a diffeomorphism defining a class of $\tilde A(\tilde g)$ defines a class of $G$ if and only if its image under the map of Definition \ref{defin:theamap} is in the orbit of the matching pair of structures inducing $(g_1, g_2)$ by matching pairs of diffeomorphisms isotopic to the identity. Since this orbit is closed by the remark after Definition \ref{defin:g2gtilde}, and we can find local continuous maps giving representatives from $\tilde A(\tilde g)$ to diffeomorphisms since $\tilde A(\tilde g)$ is finite-dimensional, we may then deduce that $G$ is a closed subgroup and the rest of the first paragraph follows. We note from Lemma \ref{sunauto} and and the remark after Definition \ref{defin:precedesg2auto} that, within $\Diff_0$, an isometry of a metric induced by a Calabi-Yau or torsion-free $G_2$ structure is an automorphism of that structure. We shall, as in Definition \ref{comdefin:matchingdiffeo}, write $(\upsilon_1, \upsilon_2)$ for the matching pair of structures, which is either $(\phi_1, \phi_2)$ or $(\Omega_1, \omega_1, \Omega_2, \omega_2)$. 

Note that if we take a different diffeomorphism representing the same class of $\tilde A(\tilde g)$, the two diffeomorphisms are isotopic, by an isotopy preserving limits. Since such an isotopy changes the result by a matching pair of diffeomorphisms isotopic to the the identity, using this different representative does not affect whether the image by the map of Definition \ref{defin:theamap} lies in the orbit, so the argument will be independent of the representative we choose. 

Firstly, if the diffeomorphism $\Phi$ representing a class of $\tilde A(\tilde g)$ is isotopic, preserving $\Diff(N)$ limit, to an automorphism $\Psi$ of $\upsilon_2$, then we have that $\Psi^{-1}\Phi$ is isotopic, preserving $\Diff(N)$ limit, to the identity. Consequently the matching pair $(id, \Psi^{-1}\Phi)$ is isotopic to the identity, and hence $(\upsilon_1, \Phi^* \upsilon_2) = (\upsilon_1, \Phi^* (\Psi^{-1})^* \upsilon_2)$ lies in the orbit by matching pairs isotopic to the identity. Consequently, representatives of the part of $G$ corresponding to $M_2$ map into the orbit under the map of Definition \ref{defin:theamap}. 

Secondly, if the diffeomorphism $\Phi$ representing a class of $\tilde A(\tilde g)$ is isotopic, preserving $\Diff(N)$ limits, to an automorphism $\Psi$ of $\upsilon_1$, then since $\Phi$ is isotopic to the identity the matching pair $(\Phi, \Phi)$ is isotopic to the identity, and so $(\Psi, \Phi)$ is isotopic to the identity. Consequently, $(\upsilon_1, \Phi^* \upsilon_2) = (\Psi^* \upsilon_1, \Phi^* \upsilon_2)$ lies in the orbit. 

We have now shown that every diffeomorphism representing a class of $G$ maps into the orbit under the map of Definition \ref{defin:theamap}. It remains to show that if $\Phi$ represents a class of $\tilde A(\tilde g)$ and $(\upsilon_1, \Phi^* \upsilon_2)$ is in the orbit, then $\Phi$ represents a class of $G$. We suppose that $(\upsilon_1, \Phi^* \upsilon_2) = (\Psi_1^* \upsilon_1, \Psi_2^* \upsilon_2)$ for some matching pair $(\Psi_1, \Psi_2)$ isotopic to the identity. In particular, the $\Diff(N)$ part of their common limit is isotopic to the identity. This isotopy is a curve as at the start of the current proof, so defines a diffeomorphism $\tilde \Psi$ giving a class of $\tilde A(\tilde g)$. We will show that $\tilde \Psi$ is isotopic with fixed $\Diff(N)$ limit to both $\Psi_1$ and $\Psi_2$. That is, $\tilde \Psi$ is isotopic with fixed $\Diff(N)$ limit to the automorphism $\Psi_1$ of $\upsilon_1$ and also $\Phi\tilde \Psi^{-1}$ is isotopic with fixed $\Diff(N)$ limit to the automorphism $\Phi \Psi_2^{-1}$ of $\upsilon_2$. That is, the classes defined by $\tilde \Psi$ and $\Phi \tilde\Psi^{-1}$ are in $G$; it follows that the class defined by $\Phi$ lies in $G$. 

We note that the extensions of $\tilde \Psi$ to $M_1$ and $M_2$ are also isotopic to the identity and that moreover these isotopies can be chosen to match the original isotopies of $\Psi_1$ and $\Psi_2$ (which, since they match each other, have the same isotopy at the limit). Inverting one of these isotopies, we find an isotopy with fixed $\Diff(N)$ limit between $\id$ and $\tilde \Psi^{-1} \Psi_i$ for either $i$; composing with $\tilde \Psi$ gives the result.

For the tangent space, we use an infinitesimal version of the previous argument: it would be possible, but more complicated, to extract tangent spaces from our description of $G$. Suppose a Killing field $X$ maps, under the derivative of the map of Definition \ref{defin:theamap}, into the tangent space of the orbit extending $X$ to $M_2$. That is, we have $(\upsilon_1, \L_X \upsilon_2) = (\L_{Y_1} \upsilon_1, \L_{Y_2} \upsilon_2)$ for some matching pair of vector fields $Y_1$ and $Y_2$. It follows that $X-Y_2$ is a Killing field for $g_2$ and that $Y_1$ is a Killing field for $g_1$; hence, since $Y_1$ and $Y_2$ match, $X$ is the sum of Killing fields extending to Killing fields for $g_1$ and $g_2$. The converse follows by reversing the argument. 

To prove the claim for $A(\phi_1, \phi_2)$ we just observe that locally around the class of the identity these manifolds are given by their tangent spaces, and all Killing fields on $S^1$-invariant Ricci-flat manifolds are $S^1$-invariant.
\end{proof}
\begin{rmk}
The dimension of $A(g_1, g_2)$ may be established using the Mayer-Vietoris theorem, see for instance \cite[Proof of Proposition 4.2]{nordstromgluing}, which shows that it is zero under the condition $b^1(M^T) = 0$ (of course, in our case we are working with $b^1(M^T \times S^1) >0$). 
\end{rmk}
Before working further with these ideas, we note that we will need to extract elements of $A(g_1, g_2)$ from pairs of matching structures. Consequently, we need to take slightly more care with the slice arguments proving Theorem \ref{g2slicetheorem}, to make sure we can determine diffeomorphisms from structures. Specifically, we require the following result, claimed without proof by Nordstr\"om \cite[second sentence of p.141]{nordstromthesis}. 
\begin{prop}
\label{g2sliceneatprop}
Suppose one of the following holds. 
\begin{enumerate}[i)]
\item Let $N$ be a compact five-dimensional manifold. Let $U$ be a sufficiently small neighbourhood in a subspace around the translation-invariant $S^1$-invariant $G_2$ structure $\phi_0$ on $N \times S^1 \times \R$ such that $U$ is transverse to the orbit of the identity component $\Diff_0^{S^1}(N \times S^1)$ defined in Definition \ref{defin:s1invdiffeo}. 

\item Let $M$ be a six-dimensional manifold with an end. Let $U$ be a sufficiently small neighbourhood in a subspace around the asymptotically cylindrical $S^1$-invariant $G_2$ structure $\phi_0$ on $M \times S^1$, consisting of structures whose limits are torsion-free $S^1$-invariant $G_2$ structures with the same automorphism groups as the limit of $\phi_0$, and such that $U$ is transverse to the orbit of the identity component of diffeomorphisms in $\Diff_0^{S^1}(M \times S^1)$ that have limits automorphisms of the limit structure $\tilde \phi_0$. 
\end{enumerate}
Then, in both cases, on the image $\left\{\frac{\Diff^{S^1}_0}{\Aut(\phi_0) \cap \Diff_0^{S^1}}\right\}^* U$ of the pullback map, the map from a smooth structure to the class of diffeomorphisms required is continuous and smooth. 
\end{prop}
\begin{proof}
It suffices to prove that the analogous map to $\frac{\Diff_0}{\Aut(\phi_0)}$ is smooth and continuous, as $\frac{\Diff^{S^1}_0}{\Aut(\phi_0) \cap \Diff_0^{S^1}}$ is a submanifold of it, by Proposition \ref{diff0s1submfd} and using Proposition \ref{s1invdiffeoweak} to see that the quotient remains locally the same. 

We prove (i); (ii) is entirely analogous. Both are straightforward applications of the inverse function theorem. Note that the forms in $U$ are not constrained to be torsion-free $G_2$ except in their limits, in case (ii). 

Consider the pullback map from $C^{2, \alpha}$ diffeomorphisms and $C^{1, \alpha}$ $3$-forms in $U$ to $C^{1, \alpha}$ $3$-forms. By combining Baier's result on the smoothness of pullbacks in the diffeomorphism \cite[Theorem 2.2.15]{baier} with linearity in the form, the pullback map is smooth. The derivative is an isomorphism, since $U$ is transverse to the derivative orbits and we have removed the automorphisms, and so by the inverse function theorem there is a small neighbourhood of $\phi_0$ in $U$, and a small neighbourhood of the identity, which we call $D$, on which the inverse is continuous and smooth. When we restrict to smooth diffeomorphisms, the pullback map must remain continuous and smooth (as a smooth map to a submanifold). 

We now just have to globalise in diffeomorphisms. Given some diffeomorphism $\Phi$, consider the subset $D \Phi$ and the slice neighbourhood $U$. The image of $D \Phi \times U$ under the pullback map is just the pullback by $\Phi$ of the image of $D \times U$. Consequently, the map from a point of the image $\phi$ to the diffeomorphism class is given by composing the inverse of $\Phi$ with the diffeomorphism class required for $\Phi^{-1, *} \phi$, which depends smoothly on $\phi$ by the previous paragraph. Since composition and pullback by fixed smooth maps are smooth, it follows that the composition depends smoothly on $\phi$. 
\end{proof}
\begin{rmk}
The slice $U$ exists by the proof of Theorem \ref{g2slicetheorem}. The required transition to the $S^1$-invariant setting is carried out on beginning on page \pageref{s1invarg2slice}. The slice required for case (i) occurs in the paragraph immediately before \eqref{eq:cycrosssectionF} and the slice required for case (ii) is the penultimate paragraph before Theorem \ref{mg2s1mfd}. 
\end{rmk}

We may now use $A(\phi_1, \phi_2)$ to find charts for $\hat\G_{G_2}^{S^1}$. 
\begin{prop}[cf. {\cite[p.141]{nordstromthesis}}]
\label{g2ghatcharts}
Suppose that $[\phi_1, \phi_2] \in \hat \G_{G_2}^{S^1}$. We have a chart $U = U_1 \times U_2$ around $([\phi_1], [\phi_2]) \in \B_{G_2}^{S^1} \subset \M_{G_2}^{S^1} \times \M_{G_2}^{S^1}$ consisting of matching pairs of structures, such that all these pairs and their limits have the same identity components of their isometry groups, and that for each $\phi'_i \in U_i$, for $\Phi^* \phi'_i$ sufficiently close to $\phi_i$ there is a continuous map $\Phi^* \phi'_i \mapsto [\Phi] \in \frac{\Diff_0^{S^1}}{\Aut \cap \Diff_0^{S^1}}$. 

Then an open subset of $A(\phi_1, \phi_2) \times U$ is homeomorphic to an open neighbourhood of $[\phi_1, \phi_2]$ in $\hat \G_{G_2}^{S^1}$, by the map from $A(\phi_1, \phi_2) \times U$ to $\hat \G_{G_2}^{S^1}$
\begin{equation}
\label{eq:ghatchartsmap}
([\Phi], (\phi'_1, \phi'_2)) \mapsto [\phi'_1, \Phi^* \phi'_2]
\end{equation}
where we take the extension of $\Phi$ to a diffeomorphism of $M_2$. 
\end{prop}
\begin{proof}
The set $U$ exists by Proposition \ref{g2btheorem} and the properties required are just properties on $M_1$ and $M_2$ so hold by the slice theorem Theorem \ref{g2slicetheorem}. The existence of the required continuous map is given by Proposition \ref{g2sliceneatprop}. 

We first show that \eqref{eq:ghatchartsmap} gives a well-defined element of $\hat \G_{G_2}^{S^1}$.  Since $(\phi'_1, \phi'_2)$ are the representatives of the point of $U$ in the chart, they match, and have the same identity components of their automorphism groups as $(\phi_1, \phi_2)$, as do their limits. It follows immediately from Proposition \ref{spaceanice} that \eqref{eq:ghatchartsmap} is a well-defined map, since we have quotiented by the stabiliser.  

Now we show injectivity. If we have $[\phi'_1, \Phi^* \phi'_2] = [\phi''_1, \Psi^* \phi''_2]$ in $\hat\G$, in particular these define the same class in $\B$. Thus so do $(\phi'_1, \phi'_2)$ and $(\phi''_1, \phi''_2)$, and by hypothesis both of these pairs lie in $U$. Since $U$ is a slice neighbourhood, it follows that $\phi'_1=\phi''_1$ and $\phi'_2 = \phi''_2$.  It remains to show that if $[\phi'_1, \Phi^* \phi'_2] = [\phi'_1, \Psi^* \phi'_2]$, then $[\Phi] = [\Psi]$ in $A(\phi_1, \phi_2)$. Again, as in Proposition \ref{spaceanice} we have shown that we have quotiented by the stabiliser, we indeed have $[\Phi] = [\Psi]$. 

It is clear that \eqref{eq:ghatchartsmap} is continuous, so it only remains to show that it maps to an open subset and its inverse there is continuous. We shall construct the open set and the inverse on it simultaneously, taking a sequence of smaller open sets as required. First of all, the projection $\hat \G_{G_2}^{S^1} \to \B_{G_2}^{S^1}$ is continuous, and so the preimage of $U$ is open. This preimage is our first open set $V_1$. We also have a natural map from an open subset of $\hat \G_{G_2}^{S^1}$ contained in $V_1$ to $A(\phi_1, \phi_2)$, as follows. Suppose given $[\phi''_1, \phi''_2] \in V_1$, which projects to $([\phi''_1], [\phi''_2]) \in U$. By definition, there then exist slice structures $\phi'_1$ and $\phi'_2$ and asymptotically cylindrical diffeomorphisms $\Phi_1$ and $\Phi_2$ such that $\phi''_i = \Phi_i^* \phi_i$. By construction, $\phi'_1$ and $\phi'_2$ match, but $\Phi_1$ and $\Phi_2$ need not; note that $\Phi_1$ and $\Phi_2$ are only defined up to isometries, but changing them by an isometry will have no effect on the final class of $A(\phi_1, \phi_2)$. Since $\Phi_1$ is asymptotically cylindrically asymptotic to the identity, its limit is isotopic to the identity, and hence the $\Diff(N)$ part is. The isotopy from the identity to the $\Diff(N)$ part of its limit defines, as in the proof of Proposition \ref{spaceanice}, a diffeomorphism $\Psi_1$ representing a class of $\tilde A(\tilde \phi)$, such that $(\Phi_1, \Psi_1)$ is isotopic to the identity as a matching pair. On the other hand, $\Phi_2$ is also asymptotically cylindrically asymptotic to the identity, so we have a diffeomorphism $\Psi_2$ such that $(\Psi_2, \Phi_2)$ is isotopic to the identity as a matching pair. Let $\Phi' = \Psi_2 \Psi_1^{-1}$; the diffeomorphism $\Phi'$ defines a class of $A(\phi_1, \phi_2)$. 

On a suitably small open set $V_2$, $\Phi_1$ and $\Phi_2$ depend continuously on $\phi''_1$ and $\phi''_2$, by the hypothesis. Consequently, since the isotopy can clearly be chosen continuously in the diffeomorphism, so do $\Psi_1$ and $\Psi_2$. Since inversion is continuous, the diffeomorphism $\Phi' = \Psi_2 \Psi_1^{-1}$ also depends continuously on $\phi''_1$ and $\phi''_2$, and so in an even smaller open subset $V_3$ $\Phi'$ defines a class of $A(\phi_1, \phi_2)$ depending continuously on $\phi''_1$ and $\phi''_2$. 

We have now constructed an open subset $V_3$ of $\hat \G_{G_2}^{S^1}$ and a map to $U \times A(\phi_1, \phi_2)$ which we hope to be the inverse. It is clearly continuous, by construction. We have to check that it is an inverse, that is that $[\phi'_1,  (\Psi_2 \Psi_1^{-1})^* \phi'_2] = [\phi''_1, \phi''_2]$. We know that the pairs $(\Phi^{-1}_1, \Psi^{-1}_1)$ and $(\id, \Phi^{-1}_2 \Psi_2)$ are both isotopic to the identity as matching pairs. Consequently, we have
\begin{equation}
[\phi''_1, \phi''_2] = [\Phi_1^* \phi'_1, \Phi_2^* \phi'_2] = [\Phi_1^* \phi'_1, \Psi_2^* \phi'_2] = [\phi'_1, (\Psi_2\Psi_1^{-1})^* \phi'_2] \qedhere
\end{equation}
\end{proof}
We now show that the charts obtained in Proposition \ref{g2ghatcharts} form an atlas, and so that $\hat\G_{G_2}^{S^1}$ is a manifold. 
\begin{prop}
\label{g2ghatmanifold}
Suppose given two open subsets of $\hat\G_{G_2}^{S^1}$ as in Proposition \ref{g2ghatcharts}, so homeomorphic by the map in the proposition to the product of an open subset of $U \times A(\phi_1, \phi_2)$ and $U' \times A(\chi_1, \chi_2)$ respectively. Suppose that these subsets intersect; on the intersection, the transition map
\begin{equation}([\phi'_1], [\phi'_2], [\Phi]) \mapsto [\phi'_1, \Phi^* \phi'_2] = [\chi'_1, \Phi'^* \chi'_2] \mapsto ([\chi'_1], [\chi'_2], [\Phi'])\end{equation}
is smooth, where $\chi'_1$ and $\chi'_2$ are the representatives of the classes $[\phi'_1]$ and $[\phi'_2]$ in the other chart, and $[\Phi']$ is the relevant class of $A(\chi_1, \chi_2)$. 
\end{prop}
\begin{proof}
The map $([\phi'_1, \phi'_2]) \mapsto ([\chi'_1], [\chi'_2])$ is the identity in $\B_{G_2}^{S^1}$, and so is smooth. Consequently also the maps to the slice representatives $\chi'_1$ and $\chi'_2$ are smooth. For the map to $\Phi'$, we note that the structures $\phi'_1$ and $\Phi^* \phi'_2$ depend smoothly on $[\phi'_1]$, $[\phi'_2]$ and $[\Phi]$. It is obvious that $[\phi'_1]$ and $[\phi'_2]$ depend smoothly on these classes, by linearity; for $[\Phi]$, we first note that since the components $\tilde A(\tilde \phi)$ are identified with the finite-dimensional space of Killing fields, we may choose a representative $\Phi$ for $[\Phi]$ smoothly, and then the pullback is smooth by \cite[Theorem 2.2.15]{baier}. 

Consequently it suffices to show the map in Proposition \ref{g2ghatcharts} determining $[\Phi']$ from the structures $\phi'_1$, $\chi'_1$, $\Phi^* \phi'_2$ and $\chi'_2$ is smooth. Exactly the same argument works, using the smoothness result of Proposition \ref{g2sliceneatprop}. 
\end{proof}

We now proceed to the Calabi-Yau case. The spaces are set up as in the $S^1$-invariant $G_2$ case, but to show they are manifolds requires some further work. We begin with the definitions. 
We first make, exactly as in Definition \ref{defin:g2gtilde},
\begin{defin}
\label{defin:su3gtilde}
Let
\begin{equation}
\tilde \G_{SU(3)} = \frac{\text{matching pairs of Calabi-Yau structures and parameters $T$}}{\text{pairs of diffeomorphisms isotopic to the identity as matching pairs}}
\end{equation}
where structures match if they match in the sense of Definition \ref{comdefin:matchingforms}, and isotopy as matching pairs and the action are as in Definition \ref{comdefin:matchingdiffeo}.
\end{defin}
We want to show that $\tilde \G_{SU(3)}$ is smooth and that we can include it into $\tilde\G_{G_2}^{S^1}$ as a smooth submanifold. It follows from our earlier analysis (in section \ref{sec:g2su3}) that to define such a map we need a matching pair of twistings $(z_1, z_2)$, with the usual boundary condition $\tilde z_i(\ddt) = 0$. Such a pair of twistings immediately gives a map from matching pairs of Calabi-Yau structures to matching pairs of $S^1$-invariant $G_2$ structures, using Theorem \ref{g2su3alltold} (Theorem C). This map may not induce a well-defined map $\tilde \G_{SU(3)}\to \tilde \G^{S^1}_{G_2}$; if the triples $(\Omega_1, \omega_1, \Omega_2, \omega_2, T)$ and $(\Omega'_1, \omega'_1, \Omega'_2, \omega'_2, T')$ are identified by the isotopic-to-the-identity matching pair $(\Phi_1, \Phi_2)$ then the extension $(\hat\Phi_1, \hat \Phi_2)$ given as in Lemma \ref{easyextensionlemma} need not identify $(\Re \Omega_1 + z_1 \wedge \omega_1, \Re \Omega_2 + z_2 \wedge \omega_2, T)$ with $(\Re \Omega'_1 + z_1 \wedge \omega'_1, \Re \Omega'_2 + z_2 \wedge \omega_2, T')$. Hence, we take a quotient of twistings as in Proposition \ref{setproduct}. We thus make, using notation inspired by Definition \ref{defin:g2bspace}, 
\begin{defin}
\label{bzdef}
Let $M_1$ and $M_2$ be as in Definition \ref{comdefin:matchingmanifolds}, and let $\B_Z$ be the space of matching pairs of twisting classes; that is, $([z_1], [z_2])$ in $Z(M_1) \times Z(M_2)$ such that there are representatives $z_1$ and $z_2$ matching in the sense of Definition \ref{comdefin:matchingforms}. Here $Z(M_i)$ is the open subset of $H^1(M_i \times S^1)$ of Lemma \ref{ccfcohom}. 
\end{defin}
\begin{rmk}
We know that a twisting is of the form $L d\theta + v$ for $v$ a $1$-form on $M$. Thus if two twistings $L_i d\theta + v_i$ match, we have $L_1=L_2$ and $v_1$ matches with $v_2$ (we used the first of these in Proposition \ref{gluingcohombehaves}).  Thus $\B_Z$ is the product of $\R_{>0}$ with the set of matching pairs $([v_1], [v_2])$. Now the forms $v_i$ have no $dt$ component in the limit and so $[\tilde v_1 - \tilde v_2]$ is a well-defined element of $H^1(N \times S^1)$, and it is easy to see using Mayer-Vietoris that a pair $([v_1], [v_2])$ equivalently matches if and only if $[\tilde v_1-\tilde v_2]$ is zero. Thus the set of matching pairs is a vector space, and $\B_Z$ is the product of $\R_{>0}$ with a vector space. It follows that $\B_Z$ is a manifold, and in particular path-connected. Analogously to our use of path-connectedness of the space of twistings in Proposition \ref{weakwelldef}, we will use path-connectedness of $\B_Z$ in Proposition \ref{btheorem} below to show that which element of $\B_Z$ we use is not very important. 
\end{rmk}
These are indeed the classes we need to use to define our inclusion maps. 
\begin{defin}
\label{defin:iotaontildeg}
Suppose that $([z_1], [z_2]) \in \B_Z$. Define the map
\begin{equation}
\iota_{[z_1], [z_2]}: \tilde \G_{SU(3)} \to \tilde \G_{G_2}^{S^1}
\end{equation}
by taking a pair of matching representatives $(z_1, z_2)$ and then mapping
\begin{equation}
(\Omega_1, \omega_1, \Omega_2, \omega_2, T) \mapsto [\Re \Omega_1 + z_1 \wedge \omega_1, \Re \Omega_2 + z_2 \wedge \omega_2, T]
\end{equation}
for any representative quintuple $(\Omega_1, \omega_1, \Omega_2, \omega_2, T)$.
\end{defin}
\begin{prop}
\label{prop:iotaontildeg}
The map $\iota_{[z_1], [z_2]}: \tilde \G_{SU(3)} \to \tilde \G_{G_2}^{S^1}$ of Definition \ref{defin:iotaontildeg} is well-defined and injective. 
\end{prop}
\begin{proof}
First suppose that $(\Omega_1, \omega_1, \Omega_2, \omega_2, T)$ and $(\Omega'_1, \omega'_1, \Omega'_2, \omega'_2, T')$ are two quintuples representing the same class of $\M_{SU(3)}$. Suppose that we use the matching pairs $(z_1, z_2)$ and $(z'_1, z'_2)$, both of which represent $([z_1], [z_2])$, to define the corresponding $S^1$-invariant $G_2$ structures. There is a matching pair of diffeomorphisms $(\Phi_1, \Phi_2)$ isotopic to the identity such that $\Omega_1 = \Phi_1^* \Omega'_1$ and so on as in Definition \ref{comdefin:matchingdiffeo}. It is clear from the definitions that extending these diffeomorphisms to $M_i \times S^1$ as in Lemma \ref{easyextensionlemma} gives a matching pair, still isotopic to the identity in the sense of Definition \ref{comdefin:matchingdiffeo}, that acts on $(\Re \Omega'_1 + z'_1 \wedge \omega'_1, \Re \Omega'_2 + z'_2 \wedge \omega'_2, T')$ to give $(\Re \Omega_1 + \Phi^*_1(z'_1) \wedge \omega_1, \Re \Omega_2 + \Phi^*_2(z'_2) \wedge \omega_2, T)$. To prove the result, it thus suffices to show that $(\Re \Omega_1 + \Phi^*_1(z'_1) \wedge \omega_1, \Re \Omega_2 + \Phi^*_2(z'_2) \wedge \omega_2, T)$ and $(\Re \Omega_1 + z_1 \wedge \omega_1, \Re \Omega_2 + z_2 \wedge \omega_2, T)$ represent the same class of $\G_{G_2}^{S^1}$. Since $(\Phi_1^* z'_1, \Phi^*_2 z'_2)$ is another matching pair of representatives for $([z_1], [z_2])$, by relabelling it suffices to prove the special case in which $(\Omega_1, \omega_1, \Omega_2, \omega_2, T) = (\Omega'_1, \omega'_1, \Omega'_2, \omega'_2, T')$.

Since $z_i - z'_i$ are exact and asymptotically translation invariant, as in the proof of Lemma \ref{ccfcohom} they are $df_i$ for some asymptotically translation invariant $f_i$. Hence, there are asymptotically cylindrical diffeomorphisms $\Phi_i \in \Diff_0^{S^1}$ identifying $\Re \Omega_i + z_i \wedge \omega_i$ and $\Re \Omega_i + z'_i \wedge \omega_i$ on $M_i \times S^1$ by Lemma \ref{changingzlemma}. We have to check that $(\Phi_1, \Phi_2)$ can be chosen to be isotopic to the identity as a matching pair and that the common limit is of the form $(\tilde \Phi_i, 0)$, i.e.\ has no translation component, so that $T$ is unaffected. By the proof of Lemma \ref{changingzlemma}, $\Phi_i$ is the time-$1$ flow of $f_i \ddth$, so its limit certainly has no translation component (which would correspond to a flow by $\ddt$). It then only remains to show that $f_1$ and $f_2$ can be chosen to match, as then the flow defines a matching isotopy. However, the proof of Lemma \ref{ccfcohom} also yields that the limits of the $f_i$ only depend on the limits of the differences $z'_i - z_i$; hence, that $f_1$ and $f_2$ match follows, if we make appropriate choices, from the fact that these differences match.

For injectivity, we apply the proof of Lemma \ref{s1invdiffstr}. If 
\begin{equation}
[\Re\Omega_1 + z_1 \wedge \omega_1, \Re \Omega_2 + z_2 \wedge \omega_2, T] = [\Re \Omega'_1 + z_1 \wedge \omega'_1, \Re \Omega'_2 + z_2 \wedge \omega'_2, T']
\end{equation}
then there is a pair of $S^1$-invariant diffeomorphisms isotopic to the identity as a matching pair pulling back $\Re \Omega_i + z_i \wedge \omega_i$ to $\Re \Omega'_i + z_i \wedge \omega'_i$. Taking the $M_i$ parts of these diffeomorphisms as in Lemma \ref{s1invdiffstr} gives diffeomorphisms of $M_i$ pulling back $\Omega_i$ to $\Omega'_i$ and $\omega_i$ to $\omega'_i$. Evidently these diffeomorphisms are also isotopic to the identity as a matching pair, and as the $M$ part must include the translation part of the limit, the action on $T$ also gives $T'$. Hence
\begin{equation}
[\Omega_1, \omega_1, \Omega_2, \omega_2, T] = [\Omega'_1, \omega'_1, \Omega'_2, \omega'_2, T']
\end{equation}
\ie $\iota_{[z_1], [z_2]}$ is injective. 
\end{proof}
To show that $\tilde \G_{SU(3)}$ is smooth, and the maps $\iota_{[z_1], [z_2]}$ are smooth inclusions, we introduce smaller moduli spaces exactly analogous to those in the $S^1$-invariant $G_2$ case. We begin with the space corresponding to $\B_{G_2}^{S^1}$ defined in Definition \ref{defin:g2bspace}.
\begin{defin}
\label{defin:su3bspace}
Let ${\mathcal{B}}_{SU(3)}$ be the space of matching pairs in the $SU(3)$ moduli spaces on $M_1$ and $M_2$, that is:
\begin{equation}
([\Omega_1, \omega_1], [\Omega_2, \omega_2]) \in \M_{SU(3)}(M_1) \times \M_{SU(3)}(M_2)
\end{equation}
such that there exist representatives $(\Omega_1, \omega_1)$ and $(\Omega_2, \omega_2)$ matching in the sense of Definition \ref{comdefin:matchingforms}. 
\end{defin}
We also need the space corresponding to $\hat \G_{G_2}^{S^1}$ of Definition \ref{defin:g2ghat}. 
\begin{defin}
Let
\begin{equation}
\hat \G_{SU(3)} = \frac{\text{matching pairs of Calabi-Yau structures}}{\text{pairs of diffeomorphisms isotopic to the identity as matching pairs}}
\end{equation}
\end{defin}
To show that $\tilde \G_{SU(3)}$ is a manifold, we argue, just as in the $S^1$-invariant $G_2$ case, that $\B_{SU(3)}$ and $\hat \G_{SU(3)}$ are manifolds. We also have to show that the inclusion maps corresponding to $\iota_{[z_1], [z_2]}$ are all smooth. 

We begin with $\B_{SU(3)}$. Using Theorem A (Theorem \ref{maintheorema}), and essentially arguing as in the proof of that theorem that $\B_{SU(3)}$ is the fibre of a smooth submersion to $\B_{G_2}^{S^1} \to \B_Z$, we prove that $\B_{SU(3)}$ is a manifold. Proposition \ref{btheorem} is the Calabi-Yau analogue of Proposition \ref{g2btheorem}. 
\begin{prop}
\label{btheorem}
$\B_{SU(3)}$ is a smooth manifold. Moreover, there exist charts consisting of matching pairs of structures around every point, with their groups of automorphisms isotopic to the identity independent of the point in the slice. 

We have a diffeomorphism $\B_{SU(3)} \times \B_Z \to \B_{G_2}^{S^1}$. In particular, the inclusion map $\B_{SU(3)} \to \B_{G_2}^{S^1}$ given by any pair of matching cohomology classes $([z_1], [z_2])$ is a well-defined smooth immersion. 
\end{prop}
\begin{proof}
We have shown in Theorem \ref{su3acyldeformation} that 
\begin{equation}
\label{eq:bspaceproduct}
\begin{split}
&\M_{G_2}^{S^1}(M_1 \times S^1) \times \M_{G_2}^{S^1}(M_2 \times S^1)
\\=\, & \M_{SU(3)}(M_1) \times Z(M_1) \times \M_{SU(3)}(M_2) \times Z(M_2)
\end{split}
\end{equation}
That is, given a pair of $S^1$-invariant $G_2$ moduli classes $[\phi_1]$ and $[\phi_2]$ we can express them in terms of Calabi-Yau structures by pairs $([\Omega_1, \omega_1], [z_1])$ and $([\Omega_2, \omega_2], [z_2])$. If moreover $([\phi_1], [\phi_2]) \in \B_{G_2}^{S^1}$, then there exist representatives $\phi_1$ and $\phi_2$ that match. By applying uniqueness in Proposition \ref{g2su3prop1} to the limits, it follows immediately that the corresponding representatives $z_i$, $\omega_i$, and $\Omega_i$ all match. 

Conversely, given a matching pair of Calabi-Yau classes and a matching pair of twisting classes, taking matching representatives for these pairs and then combining them as in Proposition \ref{g2su3prop1} gives a matching pair of $S^1$-invariant $G_2$ structures and hence of classes.

It follows therefore that the submanifold $\B_{G_2}^{S^1}$ can be expressed in terms of the product structure of \eqref{eq:bspaceproduct} as
\begin{equation}
\label{eq:bspaceclearproduct}
\B_{G_2}^{S^1} = \B_{SU(3)} \times \B_Z
\end{equation}
We proceed exactly as in the case of proving that $\M_{SU(3)}$ is a manifold by showing that $\B_{SU(3)}$ is the fibre of a surjective submersion. By the remark after Definition \ref{bzdef}, $\B_Z$ is the product of $\R_{>0}$ with a vector space, and so a manifold. An obvious smooth path of structures and hence of classes (since a path of structures defines a path of cohomology classes) yields that the map $\B_{G_2}^{S^1} \to \B_Z$ is a submersion; it follows that we have a collection of manifold structures on $\B_{SU(3)}$ by the implicit function theorem. 

The natural inclusion map from $\B_{SU(3)}$ to $\M_{SU(3)} \times \M_{SU(3)}$ is a smooth immersion from the smooth structure given on $\B_{SU(3)}$ by the implicit function theorem, because it is the composition
\begin{equation}
\B_{SU(3)} \into \B_{G_2}^{S^1} \into \M_{G_2}^{S^1} \times \M_{G_2}^{S^1} \twoheadrightarrow \M_{SU(3)} \times \M_{SU(3)}
\end{equation}
It follows that the smooth structure on $\B_{SU(3)}$ is independent of the point of $\B_Z$. 

We already know that $\B_{G_2}^{S^1}$ is a product as a topological space, as a subspace of the topological product \eqref{eq:bspaceproduct}. It remains to check that the bijective homeomorphism $\B_{SU(3)} \times \B_Z \to \B_{G_2}^{S^1}$ is a diffeomorphism. We may choose coordinates as in the statement by applying Proposition \ref{su3slicecoords}. Now the maps from an $S^1$-invariant $G_2$ structure to its $Z$ and Calabi-Yau parts are smooth, and combining these we see that the map $\B_{G_2}^{S^1} \to \B_{SU(3)} \times \B_Z$ is smooth; similarly given a Calabi-Yau structure and a twisting the map to an $S^1$-invariant $G_2$ structure is smooth, so the map $\B_{SU(3)} \times \B_Z \to \B_{G_2}^{S^1}$ is smooth, and hence is indeed a diffeomorphism. 
\end{proof}
We now turn to the smoothness of $\hat \G_{SU(3)}$. The coordinate charts are set up as in Proposition \ref{g2ghatcharts}. 
\begin{prop}
\label{su3ghatcharts}
Suppose that $[\Omega_1, \omega_1, \Omega_2, \omega_2] \in \hat \G_{SU(3)}$. We have a chart $U$ for $\B_{SU(3)}$ around $([\Omega_1, \omega_1], [\Omega_2, \omega_2])$ consisting of matching pairs of structures, such that all these pairs and their limits have the same groups of isometries isotopic to the identity, and that for each $(\Omega, \omega)$ on $M_1$ or $M_2$, for $(\Phi^*\Omega_i, \Phi^* \omega_i) $ sufficiently close to $(\Omega_i, \omega_i)$ there is a continuous map $(\Phi^* \Omega_i,\Phi^* \omega_i) \ \mapsto \Phi$ (which is not necessarily unique). 

Then an open subset of $A(\Omega_1, \omega_1, \Omega_2, \omega_2) \times U$ is homeomorphic to an open neighbourhood of $[\Omega_1, \omega_1, \Omega_2, \omega_2]$ in $\hat \G_{SU(3)}$, by the map from $A(\Omega_1, \omega_1, \Omega_2, \omega_2) \times U$ to $\hat \G_{G_2}^{S^1}$
\begin{equation}
([\Phi], (\Omega'_1, \omega'_1, \Omega'_2, \omega'_2)) \mapsto [\Omega'_1, \omega'_1, \Phi^* \Omega'_2, \Phi^* \Omega'_2]
\end{equation}
where $A(\Omega_1, \omega_1, \Omega_2, \omega_2)$ is as in Definition \ref{defin:spacea}, and its action is as described in Definition \ref{defin:theamap}. 
\end{prop}
\begin{proof}
Once the first paragraph of the proposition is established, the rest follows by exactly the same methods as in Proposition \ref{g2ghatcharts}. 

To establish the first paragraph, we note that we have charts with the required property on isometries by Proposition \ref{btheorem}, and a continuous map giving diffeomorphisms between Calabi-Yau structures is the following composition of which every step is continuous. Given a pair of structures $(\Omega, \omega)$ and $(\Omega', \omega')$ close by and representing the same moduli class, we have that the continuous images $\Re \Omega + d\theta \wedge \omega$ and $\Re \Omega' + d\theta \wedge \omega'$ are also close by and represent the same moduli class. In turn, therefore, we have a continuously dependent $S^1$-invariant $\Phi$ pulling back the first to the second, by the result of Proposition \ref{g2ghatcharts}. We know that $(\Omega, \omega)$ and $(\Omega', \omega')$ represent the same class of $\M_{SU(3)}$, so there is a diffeomorphism $\Phi'$ pulling back one to the other. Composing the extension of $\Phi'$ using Lemma \ref{easyextensionlemma} with $\Phi^{-1}$ clearly gives an automorphism of the $G_2$ structure, and hence an isometry of the product metric. Therefore restricting the composition to $M$ as in Lemma \ref{s1invdiffstr} is also an isometry, and it follows by Lemma \ref{sunauto} that it is an automorphism of the Calabi-Yau structure. In particular, we see that $\Phi^* \Omega = \Omega'$ and $\Phi^* \omega = \omega'$. The map of diffeomorphisms given by restricting $\Phi$ to $M$ is continuous, as we see in Lemma \ref{s1invdiffstr} that restricting an $S^1$-invariant diffeomorphism to $M$ is essentially composition with an inclusion and a projection: that is, the final step is continuous, as required.
\end{proof}

We now show that $\hat \G_{SU(3)}$ is smooth and the inclusion maps $\iota_{[z_1], [z_2]}: \hat \G_{SU(3)} \to \hat \G_{G_2}^{S^1}$ induced as in Definition \ref{defin:iotaontildeg} are smooth. 
\begin{prop}
\label{su3ghatmanifold}
Let $(\Omega_1, \omega_1, \Omega_2, \omega_2)$ define a class of $\hat \G_{SU(3)}$ and let $(z_1, z_2)$ define a class of $\B_Z$. Let $(\phi_1 = \Re \Omega_1 + z_1 \wedge \omega_1, \phi_2 = \Re \Omega_2 + z_2 \wedge \omega_2)$ define $\iota_{[z_1], [z_2]}([\Omega_1, \omega_1, \Omega_2, \omega_2]) \in \hat \G_{G_2}^{S^1}$. The manifolds $A(\Omega_1, \omega_1, \Omega_2, \omega_2)$ and $A(\phi_1, \phi_2)$ of Definition \ref{defin:spacea} can be naturally identified in a neighbourhood of the class of the identity. The inclusion map $\iota_{[z_1], [z_2]}$ can thus be examined locally in terms of the local homeomorphisms of Propositions \ref{su3ghatcharts} and \ref{g2ghatcharts} as
\begin{equation}
A(\Omega_1, \omega_1, \Omega_2, \omega_2) \times \B_{SU(3)} \to \hat \G_{SU(3)} \to \hat \G^{S^1}_{G_2} \to A(\phi_1, \phi_2) \times \B_{G_2}^{S^1}
\end{equation}
It is smooth. Moreover, it is an immersion, so the coordinate charts defined on $\hat\G_{SU(3)}$ in Proposition \ref{su3ghatcharts} form an atlas. 
\end{prop}
\begin{proof}
Locally around the identity, both $A(\phi_1, \phi_2)$ and $A(\Omega_1, \omega_1, \Omega_2, \omega_2)$ are manifolds. Consequently, we may work with the tangent spaces at the identity. These are quotients as identified in Proposition \ref{spaceanice}: the quotient of Killing fields on the cross-section by those that extend to Killing fields on the asymptotically cylindrical pieces.

It is clear that a Killing field on $N$ extends to a Killing field on $N \times S^1$. Conversely, given a Killing field on $N \times S^1$, it is (since parallel and so $S^1$-invariant) $X + c \ddth$, with $X$ a Killing field on $N$ and $c$ a constant, which defines a map from Killing fields on $N \times S^1$ to Killing fields on $N$. Hence we have maps between $\tilde A(\tilde \phi)$ and $\tilde A(\tilde \Omega, \tilde \omega)$. We have to show first that the maps induced on the quotients $A(\phi_1, \phi_2)$ and $A(\Omega_1, \omega_1, \Omega_2, \omega_2)$ by these are well-defined. If a Killing field on $N$ extends to a Killing field on $M_1$, say, then clearly the corresponding Killing field on $N \times S^1$ extends as a Killing field to $M_1 \times S^1$. Conversely, if a Killing field $X + c\ddth$ on $N \times S^1$ extends to a Killing field on $M_1 \times S^1$, then since $c\ddth$ is itself a Killing field on $M_1$, $X$ must also so extend. Thus these maps are well-defined. That the maps are inverse to each other also follows easily from the fact that $c \ddth$ extends to an $S^1$-invariant Killing field of $M_1$. Thus, in a sufficiently small subset of the identity, $A(\Omega_1, \omega_1, \Omega_2, \omega_2)$ and $A(\phi_1, \phi_2)$ are naturally identified. 

For the second claim, for notational simplicity setting $\Psi = \Phi^{-1}$, in these coordinates $\iota_{[z_1], [z_2]}$ becomes
\begin{align}
([\Omega_1, \omega_1], [\Omega_2, \omega_2], [\Phi]) &\mapsto [(\Omega_1, \omega_1), ( \Phi^* \Omega_2, \Phi^* \omega_2)]\\ &\mapsto [\Re \Omega_1 + z_1 \wedge \omega_1,  \Phi^* (\Re \Omega_2 + {{\Psi^*}}z_2 \wedge \omega_2)] \\&\mapsto ([\Re \Omega_1 + z_1 \wedge \omega_1], [\Re \Omega_2 + {\Psi^*}z_2 \wedge \omega_2], [\Phi])
\end{align}
The map from $A(\Omega_1, \omega_1, \Omega_2, \omega_2)$ to $A(\phi_1, \phi_2)$ is clearly the identity under the identification of the previous paragraph and so smooth, so it is sufficient to check that the map to the $\B_{G_2}^{S^1}$ component is smooth. We show that the $\B_{G_2}^{S^1}$ component is independent of $\Phi$. Then the $\B_{G_2}^{S^1}$ component is just $([\Re \Omega_1 + z_1 \wedge \omega_1], [\Re \Omega_2 + z_2 \wedge \omega_2])$, which depends smoothly on the class $([\Omega_1, \omega_1], [\Omega_2, \omega_2])$ of $\B_{SU(3)}$ by Proposition \ref{btheorem}. 

So, it is enough to show that as moduli classes we have $[\Re \Omega_2 + z_2 \wedge \omega_2] = [\Re \Omega_2 + {\Psi^*}z_2 \wedge \omega_2]$. On an appropriate neighbourhood of the class of the identity in $A(\Omega_1, \omega_1, \Omega_2, \omega_2)$, possibly reducing the size of the charts, it is sufficient to show the equalities of cohomology classes
\begin{equation}
[\Re \Omega_2 + z_2 \wedge \omega_2] = [\Re \Omega_2 + {\tilde \Psi^*} z_2 \wedge \omega_2] \qquad [\Re \tilde \Omega_{2, 2} + \tilde z_2 \wedge \tilde \omega_{2, 2}] = [\Re \tilde \Omega_{2, 2} + {\tilde \Psi^*} \tilde z_2 \wedge \tilde \omega_{2, 2}]
\end{equation}
where the additional subscript $2$ in the second equation denotes the relevant components of $\tilde \Omega_2 = \tilde \Omega_{2, 1} + dt \wedge \tilde \Omega_{2, 2}$ and $\tilde \omega_2 = \tilde \omega_{2, 1} + dt \wedge \tilde \omega_{2, 2}$ as in Theorem \ref{g2slicetheorem}. Using that theorem, we know that structures that are sufficiently close and have these cohomology classes the same define the same moduli classes. But $\Phi$ is isotopic to the identity, and so so is $\Psi$, and $\omega_2$ and $\omega_{2, 2}$ are closed (since $\omega_2$ is parallel): it follows that the cohomology classes are the same.

$\iota_{[z_1], [z_2]}$ is now obviously an immersion, because the identity is and the inclusion of $\B_{SU(3)}$ into $\B_{G_2}^{S^1}$ is (by Proposition \ref{btheorem} again). Since the manifold structure on $\hat\G_{G_2}^{S^1}$ is fixed, and $\iota_{[z_1], [z_2]}$ is independent of which chart we take on $\hat\G_{SU(3)}$, we find that each chart is a submanifold of $\hat \G_{G_2}^{S^1}$. By uniqueness of the smooth structure on a submanifold, it follows that the transition functions for the charts of Proposition \ref{su3ghatcharts} are smooth. 
\end{proof}

Finally, $\tilde \G_{SU(3)}$ is a principal $\R$-bundle over $\hat\G_{SU(3)}$ exactly as $\tilde \G_{G_2}^{S^1}$ is a principal $\R$-bundle over $\hat\G_{G_2}^{S^1}$; consequently, it is smooth. The   inclusions $\iota_{[z_1], [z_2]}$ of Definition \ref{defin:iotaontildeg} are bundle maps over the corresponding inclusions $\hat \G_{SU(3)} \to \hat \G_{G_2}^{S^1}$: thus, for every pair $([z_1], [z_2]) \in \B_Z$, $\iota_{[z_1], [z_2]}: \tilde \G_{SU(3)} \to \tilde \G_{G_2}^{S^1}$ is a smooth map. 

Consequently we have the final result of this subsection
\begin{thm}
\label{gluingdataspacesproduct}
The spaces $\tilde \G_{SU(3)}$ and $\tilde \G_{G_2}^{S^1}$ defined in Definitions \ref{defin:g2gtilde} and \ref{defin:su3gtilde} are manifolds. With $\B_Z$ as defined in Definition \ref{bzdef}, we have a diffeomorphism
\begin{equation}
\label{eq:gluingdataspacesproduct}
\tilde\G_{SU(3)} \times \B_Z \to \tilde\G_{G_2}^{S^1}
\end{equation}
induced from that of Theorem \ref{maintheorema}. 
\end{thm}
\begin{proof}
Using the maps $\iota_{[z_1], [z_2]}$ of Definition \ref{defin:iotaontildeg}, we have a map $\tilde \G_{SU(3)} \times \B_Z \to \tilde \G_{G_2}^{S^1}$
\begin{equation}
\label{eq:gluingdataspacesprodmap}
([\Omega_1, \omega_1, \Omega_2, \omega_2, T], ([z_1], [z_2])) \mapsto \iota_{[z_1], [z_2]}([\Omega_1, \omega_1, \Omega_2, \omega_2, T])
\end{equation}
\eqref{eq:gluingdataspacesprodmap} is smooth because in local coordinates, by Proposition \ref{su3ghatmanifold}, it reduces to the corresponding map $\B_{SU(3)} \times \B_Z \to \B_{G_2}^{S^1}$, the identity on the $T$ component, and the identity $A(\Omega_1, \omega_1, \Omega_2, \omega_2) \to A(\phi_1, \phi_2)$. Using that the map on $\B$ spaces is a diffeomorphism, it follows that \eqref{eq:gluingdataspacesprodmap} is a smooth local diffeomorphism. 

It is clearly a surjection, as any representative $(\phi_1, \phi_2, T)$ of a class of $\tilde \G_{G_2}^{S^1}$ can be written as $(\Re \Omega_1 + z_1 \wedge \omega_1, \Re \Omega_2 + z_2 \wedge \omega_2, T)$ for some matching pair of Calabi-Yau structures and matching pair of twistings as in Proposition \ref{btheorem}. It is an injection because if 
\begin{equation}
[\Re \Omega_1 + z_1 \wedge \omega_1, \Re \Omega_2 + z_2 \wedge \omega_2, T] = [\Re \Omega'_1 + z'_1 \wedge \omega'_1, \Re \Omega'_2 + z'_2 \wedge \omega'_2, T']
\end{equation}
then there are asymptotically cylindrical diffeomorphisms relating these $S^1$-invariant $G_2$ structures, and in particular we see that $[z_i] = [z'_i]$ as in Lemma \ref{s1invdiffstr}. Injectivity then follows by injectivity in Proposition \ref{prop:iotaontildeg}. 

Thus \eqref{eq:gluingdataspacesproduct} is a global diffeomorphism, as claimed. 
\end{proof}
\subsection{Restricting to data that can be glued}
\label{ssec:removingtildes}
In this subsection, we define the subspaces of the quotient $\tilde \G_{G_2}^{S^1}$ and $\tilde \G_{SU(3)}$ that actually glue and the corresponding gluing maps (Definitions \ref{defin:g2modspacegluingmap} and \ref{defin:su3modspacegluingmap}). We have to define the gluing on Calabi-Yau structures by our inclusions $\tilde \G_{SU(3)} \to \tilde \G_{G_2}^{S^1}$; as there is more than one such inclusion, there is more than one possible such gluing map. Consequently, we must prove that the Calabi-Yau gluing map is independent of the inclusion we consider: this is the final result of this subsection, Proposition \ref{prop:strongwelldef}, and essentially follows by combining Proposition \ref{weakwelldef} with the fact that the gluing map is well-defined on the $G_2$ moduli space (Theorem \ref{g2gluingwelldef}). 

We begin by defining $\G_{G_2}^{S^1} \subset \tilde \G_{G_2}^{S^1}$: our definition, Definition \ref{defin:g2gluingmodspace}, is adapted from Nordstr\"om  \cite[Definition 2.4]{nordstromgluing}. From Theorem \ref{g2structuregluinggeneral}, we know that any pair of $S^1$-invariant $G_2$ structures glues for gluing parameter $T>T_0$ for some $T_0$, and $T_0$ is upper semi-continuous in the pair of structures. \cite[Proposition 4.4]{nordstromgluing} says that the derivative of the gluing map between moduli spaces of $G_2$ structures (and hence of $S^1$-invariant $G_2$ structures) is an isomorphism for $T>T'_0$ for some, possibly larger, $T'_0$. The proof of \cite[Theorem 3.1]{nordstromgluing} enables us to infer that $T'_0$ is also upper semi-continuous in the structures: in order to prove Theorem B, that the gluing map between moduli spaces of Calabi-Yau structures is a local diffeomorphism, we would like to have $T>T'_0$ as well. Therefore we make
\begin{defin}[cf. {\cite[Definition 2.4]{nordstromgluing}}]
\label{defin:g2gluingmodspace}
Let $\G_{G_2}^{S^1} \subset \tilde \G_{G_2}^{S^1}$  be the subset of $G_2$ gluing data classes that have a representative $(\phi_1, \phi_2, T)$ with $T$ large enough that $\phi_1$ and $\phi_2$ can be glued with parameter $T$ in the sense of Theorem \ref{g2structuregluinggeneral} and the derivative of the gluing map is an isomorphism at the triple $(\phi_1, \phi_2, T)$. 
\end{defin}
We see that $\G_{G_2}^{S^1}$ is an open subset of $\tilde \G_{G_2}^{S^1}$. We may then define a gluing map from $\G_{G_2}^{S^1}$ in the obvious way. Note that as $T$ is now varying, we cannot sensibly use $M^T$ for the glued manifold as in Theorem \ref{g2structuregluinggeneral}. We shall call it $M$. 
\begin{defin}
\label{defin:g2modspacegluingmap}
The gluing map $\G_{G_2}^{S^1}$ to $\M_{G_2}^{S^1}(M \times S^1)$ is defined as follows. Given a class in $\G_{G_2}^{S^1}$, by definition it admits a representative $(\phi_1, \phi_2, T)$ that glues in the sense of Theorem \ref{g2structuregluinggeneral} (and by the proof of Theorem \ref{su3structuregluing} the resulting structure is $S^1$-invariant). Take the class of the result in $\M_{G_2}^{S^1}(M \times S^1)$.
\end{defin}
There is likely to be more than one such representative, but Nordstr\"om has shown 
\begin{thm}[cf. {\cite[{Proposition 4.1}]{nordstromgluing}}]
\label{g2gluingwelldef}
The map $\G^{S^1}_{G_2} \to \M^{S^1}_{G_2}$ given by Definition \ref{defin:g2modspacegluingmap} is well-defined.
\end{thm}
Of course, Nordstr\"om proved that the corresponding map $\G_{G_2} \to \M_{G_2}$, with $\G_{G_2}$ defined analogously to $\G_{G_2}^{S^1}$, was well-defined: since both $\G_{G_2}$ and $\M_{G_2}$ are locally diffeomorphic to the corresponding $S^1$-invariant spaces and the map is defined identically, the $S^1$-invariant result immediately follows.

The most natural definition of $\G_{SU(3)}$ would be to take those classes of Calabi-Yau gluing data that have representatives $(\Omega_1, \omega_1, \Omega_2, \omega_2, T)$ that glue using the inclusions $\iota_{[z_1], [z_2]}$ of Definition \ref{defin:iotaontildeg}. However, it is possible that the required $T$ may depend on $z_i$. We will therefore work initially with subsets depending on the class $([z_1], [z_2]) \in \B_Z$, but we will then take the union to define our space $\G_{SU(3)}$, and check that the gluing map is still well-defined. First, we make
\begin{defin}
\label{defin:su3gluingmodspaces}
Let
\begin{align}
\notag\G_{SU(3), ([z_1], [z_2])} &= \iota_{[z_1], [z_2]}^{-1}(\G_{G_2}^{S^1}) \\
&=\{[\Omega_1, \omega_1, \Omega_2, \omega_2, T]: [\Re \Omega_1 + z_1 \wedge \omega_1, \Re \Omega_2 + z_2 \wedge \omega_2, T] \in \G_{G_2}^{S^1}\}
\end{align}
\end{defin}
$\G_{SU(3), ([z_1], [z_2])}$ is the inverse image of an open subset under a continuous map so open. Note that, for any choice of $([z_1], [z_2])$, every class of $\hat G_{SU(3)}$ is included in $\G_{SU(3), ([z_1], [z_2])}$ for sufficiently large $T$,  because every pair of matching $S^1$-invariant $G_2$ structures glues and the derivative is an isomorphism for $T$ sufficiently large. 

We now define a family of gluing maps:
\begin{defin}
\label{defin:su3modspacegluingmaps}
Define the gluing map on the space of gluing data $\G_{SU(3), ([z_1], [z_2])}$ given by Definition \ref{defin:su3gluingmodspaces} by the composition
\begin{equation}
\G_{SU(3), ([z_1], [z_2])} \to \G_{G_2}^{S^1} \to \M_{G_2}^{S^1} \to \M_{SU(3)}
\end{equation}
where the first map is the inclusion $\iota_{[z_1], [z_2]}$, the second map is the gluing map of Definition \ref{defin:g2modspacegluingmap}, and the third map is the appropriate projection of Theorem \ref{smoothproduct}. 
\end{defin}
Rather than a family of spaces of gluing data and corresponding gluing maps, we would like a single space with a single gluing map. We make
\begin{defin}
Let 
\begin{equation}
\G_{SU(3)} = \bigcup_{([z_1], [z_2]) \in \B_Z} \G_{SU(3), ([z_1], [z_2])}
\end{equation}
\end{defin}
$\G_{SU(3)}$ is also an open subset of $\tilde \G_{SU(3)}$, and so a manifold. We can define a gluing map on it in the natural way
\begin{defin}
\label{defin:su3modspacegluingmap}
Define the gluing map $\G_{SU(3)} \to \M_{SU(3)}$ by taking the map of Definition \ref{defin:su3modspacegluingmaps} on each of the open subsets $\G_{SU(3), ([z_1], [z_2])}$. 
\end{defin}
The gluing map of Definition \ref{defin:su3modspacegluingmap} is not a priori well-defined. If we are given a class $[\Omega_1, \omega_1, \Omega_2, \omega_2, T]$ of $\G_{SU(3)}$, there may be a pair of pairs $([z_1], [z_2])$ and $([z'_1], [z'_2])$ such that $\iota_{[z_1], [z_2]}([\Omega_1, \omega_1, \Omega_2, \omega_2, T])$ and $\iota_{[z'_1, z'_2]}([\Omega_1, \omega_1, \Omega_2, \omega_2, T])$ both lie in $\G_{G_2}^{S^1}$. We have to show that under the two maps of Definition \ref{defin:su3modspacegluingmaps} $[\Omega_1, \omega_1, \Omega_2, \omega_2, T]$ has the same image.

We have already proved Proposition \ref{weakwelldef}, which says that the $\M_{SU(3)}$ components of the gluing of the two pairs $(\Omega_1 + z_1 \wedge \omega_1, \Omega_2 + z_2 \wedge \omega_2, T+S)$ and $(\Omega_1 + z'_1 \wedge \omega_1, \Omega_2 + z'_2 \wedge \omega_2, T+S)$ are equal for $S$ large enough. We now have two distinct classes in $\G_{G_2}^{S^1}$ corresponding to our two inclusions. For each of these classes, there exist representatives that glue, but we know nothing about how the representatives for the different classes are related. However, for $S$ large enough, we know that the explicit representatives $(\Omega_1 + z_1 \wedge \omega_1, \Omega_2 + z_2 \wedge \omega_2, T+S)$ and $(\Omega_1 + z'_1 \wedge \omega_1, \Omega_2 + z'_2 \wedge \omega_2, T+S)$ glue, and we may apply Proposition \ref{weakwelldef} to deduce that these have the same $\M_{SU(3)}$ component. We will show that the equality of the $\M_{SU(3)}$ component is independent of increasing the gluing parameter, so that although $(\Omega_1 + z_1 \wedge \omega_1, \Omega_2 + z_2 \wedge \omega_2, T)$ and $(\Omega_1 + z'_1 \wedge \omega_1, \Omega_2 + z'_2 \wedge \omega_2, T)$ may not glue, the result of gluing the classes $[\Omega_1 + z_1 \wedge \omega_1, \Omega_2 + z_2 \wedge \omega_2, T]$ and $[\Omega_1 + z'_1 \wedge \omega_1, \Omega_2 + z'_2 \wedge \omega_2, T]$ must also have the same $\M_{SU(3)}$ component. 
\begin{lem}
\label{reducingslemma}
Suppose that $(\Omega_i, \omega_i, T)$ and $(\Omega'_i, \omega'_i, T')$ are representatives of the same class in $\tilde \G_{SU(3)}$ and for matching pairs of twistings $z_i$ and $z'_i$, possibly defining different twisting classes, the resulting representatives $(\Re \Omega_i + z_i \wedge \omega_i, T)$ and $(\Re \Omega'_i + z'_i \wedge \omega'_i, T')$ for the corresponding classes of $\tilde \G_{G_2}^{S^1}$ glue and they continue to glue if the parameters $T$ and $T'$ are increased. Suppose that there exists $S>0$ such that the results of gluing $(\Re \Omega_i + z_i \wedge \omega_i, T+S)$ and $(\Re \Omega'_i + z'_i \wedge \omega'_i, T'+S)$ have the same $\M_{SU(3)}$ component. Then so too do the results of gluing $(\Re \Omega_i + z_i \wedge \omega_i, T)$ and $(\Re \Omega'_i + z'_i \wedge \omega'_i, T')$. 
\end{lem}
\begin{proof}
We have two curves in $\M_{G_2}^{S^1}$, defined on $[0, S]$ by gluing the curves $s \mapsto (\Re \Omega_1 + z_1 \wedge \omega_1, \Re \Omega_2 + z_2 \wedge \omega_2, T+s)$ and $s \mapsto s(\Re \Omega'_1 + z'_1 \wedge \omega'_1, \Re \Omega'_2 + z'_2 \wedge \omega'_2, T'+s)$. We consider the projection of these to $\M_{SU(3)}$, and call them $([\Omega(s), \omega(s)])$ and $([\Omega'(s), \omega'(s)])$. By Proposition \ref{su3cohomrep}, $\M_{SU(3)}$ is locally represented by the cohomology of $\Re \Omega$ and $\omega$, and so these curves are determined by their values at a point and the corresponding curves of cohomology classes. Now these cohomology classes are
\begin{equation}
[\gamma_{T+s}(\Re \Omega_1, \Re \Omega_2)] \text{ and } c(s)[\gamma_{T+s}(\omega_1, \omega_2)]
\end{equation}
and the same with primes, where $\gamma_T$ is the gluing map of Definition \ref{comdefin:matchingforms}, and the functions $c$ and $c'$ are as in Proposition \ref{gluingcohombehaves}. 

By assumption, $[\Omega(S), \omega(S)] = [\Omega'(S), \omega'(S)]$. In particular, we have the same for the cohomology classes:
\begin{align}
\label{eq:increasingTfirstOm}
[\gamma_{T+S}(\Re \Omega_1, \Re \Omega_2)] &= [\gamma_{T'+S}(\Re \Omega'_1, \Re \Omega'_2)] \\
\label{eq:increasingTfirstom}
c(s)[\gamma_{T+S}(\omega_1, \omega_2)] &= c'(s)[\gamma_{T'+s}(\omega'_1, \omega'_2)]
\end{align}
As $\Omega_i, \omega_i$ and $\Omega'_i, \omega'_i$ define the same Calabi-Yau class, they have agreeing cohomology and agreeing limit cohomology, it follows that $[\gamma_{T+s}(\Re \Omega_1, \Re \Omega_2)] = [\gamma_{T'+s}(\Re \Omega'_1, \Re \Omega'_2)]$ for all $s$ (because they agree at $s=S$ and the change as we reduce $s$ are given by the Mayer-Vietoris sequence and the limit cohomology; see, for example, \cite[Proposition 3.2]{nordstromgluing}). 

The K\"ahler parts are complicated slightly by the functions $c$ and $c'$. Under the restriction map to the cohomology of $M_1$, $[\gamma_{T+S}(\omega_1, \omega_2)]$ and $[\gamma_{T'+S}(\omega'_1, \omega'_2)]$ give $[\omega_1]$ and $[\omega'_1]$, respectively, which are equal by assumption and nonzero by non-degeneracy of the K\"ahler form. Thus from \eqref{eq:increasingTfirstom} we see that 
\begin{equation}
[\gamma_{T+S}(\omega_1, \omega_2)] = [\gamma_{T'+S}(\omega'_1, \omega'_2)] \qquad c(S) = c'(S)
\end{equation}
The same argument as in the previous paragraph then shows the equality of the cohomology classes over the whole curve. Now Lemma \ref{ccontrol} shows that there exists $\epsilon$ such that $c(s) = c'(s)$ for  $s>S-\epsilon$. Since the cohomology represents $\M_{SU(3)}$ locally homeomorphically we get that the curves in $\M_{SU(3)}$ agree for $s>S-\epsilon$; by continuity they also agree at $s=S-\epsilon$. Generalising the above argument, and using the connectedness of $[0, S]$, it follows that the curves in $\M_{SU(3)}$ agree at $s=0$. 
\end{proof}
As envisaged, Lemma \ref{reducingslemma} enables us to prove a far stronger well-definition result than Proposition \ref{weakwelldef}. 
\begin{prop}
\label{prop:strongwelldef}
The gluing map $\G_{SU(3)} \to \M_{SU(3)}$ of Definition \ref{defin:su3modspacegluingmap} is well-defined.
\end{prop}
\begin{proof}
Suppose that $[\Omega_1, \omega_1, \Omega_2, \omega_2, T]$ is a class of $\G_{SU(3)}$ and there exist two twisting class pairs $[z_1, z_2]$ and $[z'_1, z'_2]$ such that the classes $[\Re \Omega_i + z_i \wedge \omega_i, T]$ and $[\Re \Omega_i + z'_i \wedge \omega_i, T]$ both lie in $\G_{G_2}^{S^1}$. That is, there are representatives
\begin{equation}
(\Re \Omega_1 + z_1 \wedge \omega_1, \Re \Omega_2 + z_2 \wedge \omega_2, T) \text{ and } (\Re \Omega'_1 + z'_1 \wedge \omega'_1, \Re \Omega'_2 + z'_2 \wedge \omega'_2, T')
\end{equation}
both of which glue and continue to glue if $T$ and $T'$ are increased.

Now take $(\Phi_1, \Phi_2)$ to be a matching pair of diffeomorphisms of $M_1$ and $M_2$ isotopic to the identity pulling back $(\Omega'_i, \omega'_i, T')$ to $(\Omega_i, \omega_i, T)$, which represents the same $\tilde G_{SU(3)}$ class, by construction. Then there exists $S$ sufficiently large so that
\begin{equation}\label{eq:strongwelldefstar}
(\Re \Omega_1 + \Phi^*_1 z'_1 \wedge \omega_1, \Re \Omega_2 + \Phi^*_2 z'_2 \wedge \omega_2, T+S)
\end{equation}
also glues. Since the action of $(\Phi_1, \Phi_2)$ is affine on the gluing parameter, gluing \eqref{eq:strongwelldefstar} defines the same class of $\M_{G_2}^{S^1}$ as gluing
\begin{equation}\label{eq:strongwelldefdagger}
(\Re \Omega'_1 + z'_1 \wedge \omega'_1, \Re \Omega'_2 + z'_2 \wedge \omega'_2, T'+S)
\end{equation}
By Theorem \ref{g2gluingwelldef}, the results of gluing these two are thus the same. 

Also, however, by Proposition \ref{weakwelldef}, and possibly increasing $S$ some more, the result of gluing \eqref{eq:strongwelldefstar} has the same $\M_{SU(3)}$ component as the result of gluing 
\begin{equation}\label{eq:strongwelldefddagger}
(\Re \Omega_1 + z_1 \wedge \omega_1, \Re \Omega_2 + z_2 \wedge \omega_2, T+S)
\end{equation}
Thus we have that the $\M_{SU(3)}$ component giving by gluing \eqref{eq:strongwelldefddagger} is the same as for gluing \eqref{eq:strongwelldefdagger}. By Lemma \ref{reducingslemma}, we can then reduce to $S$ to zero, which proves the proposition. 
\end{proof}
\begin{rmk}
The proof of Proposition \ref{prop:strongwelldef} is closely allied to Nordstr\"om's proof of the $G_2$ theorem, Theorem \ref{g2gluingwelldef}, which again works by taking a curve for far larger gluing parameter and arguing on cohomology, and then increasing the gluing parameter in a controlled fashion. It would not be that hard to combine the two proofs (essentially proving Proposition \ref{weakwelldef} in a representation-independent way).
\end{rmk}
\subsection{The main theorem}
\label{ssec:finalssec}\label{ssec:localdiffeoofmodspace}
We now turn to Theorem B, which is that the gluing map is a local diffeomorphism of moduli spaces, and is Theorem \ref{localdiffeoofmodspace} below. The idea is that by the previous work we locally have
\begin{equation}
\label{eq:gsplitting}
\begin{aligned}
\G_{SU(3)} =& A(\Omega_1, \omega_1, \Omega_2, \omega_2) \times \R_{>0} \times \B_{SU(3)} \\\subset& A(\Re \Omega_1 + z_1 \wedge \omega_1, \Re \Omega_2 + z_2 \wedge \omega_2) \times \R_{>0} \times \B_{G_2}^{S^1} = \G_{G_2}^{S^1}
\end{aligned}
\end{equation}
we have by Theorem A (Theorem \ref{maintheorema})
that
\begin{equation}
\label{eq:msplitting}
\M_{G_2}^{S^1} = \M_{SU(3)} \times Z
\end{equation}
and by Nordstr\"om's work \cite{nordstromgluing} we understand the gluing map defined (by Definition \ref{defin:g2modspacegluingmap} below) on $\G_{G_2}^{S^1}$ and its derivative, which is an isomorphism. We thus just have to consider what happens in terms of the splittings in \eqref{eq:gsplitting} and \eqref{eq:msplitting}. We consider the $\B_Z$ part, $\B_{SU(3)}$ part, the $A$ part, and the $\R_{>0}$ part in turn, showing that variations in $\B_Z$ lead to variations in the $Z$ component of $\M_{SU(3)} \times Z$ but variations in the other three parts lead to variations in $\M_{SU(3)}$ and perhaps rescaling of the $Z$ component. It then follows easily that the composition defined by Definition \ref{defin:su3modspacegluingmap} also has derivative an isomorphism, and Theorem B then follows from the inverse function theorem. 

Formally, we should first note Nordstr\"om's result for the $G_2$ case, which says (\cite[Theorem 2.3]{nordstromgluing}; see also our Definition \ref{defin:g2gluingmodspace}) that the map
\begin{equation}
\G_{G_2} \to \M_{G_2}
\end{equation}
defined analogously to Definition \ref{defin:g2modspacegluingmap} is a local diffeomorphism. As we know that $\G_{G_2}$ and $\M_{G_2}$ are locally diffeomorphic to $\G_{G_2}^{S^1}$ and $\M_{G_2}^{S^1}$ respectively, and the gluing map commutes with these local diffeomorphisms as it maps $S^1$-invariant structures to $S^1$-invariant structures (proof of Theorem \ref{su3structuregluing}), Nordstr\"om's result implies
\begin{thm}
\label{g2localdiffeoofmodspace}
The gluing map 
\begin{equation}
\G_{G_2}^{S^1} \to \M_{G_2}^{S^1}
\end{equation}
defined in Definition \ref{defin:g2modspacegluingmap} has derivative an isomorphism, and is thus a local diffeomorphism. 
\end{thm}
We show
\begin{thm}
\label{maintheoremb}
\label{localdiffeoofmodspace}
The map
\begin{equation}
\G_{SU(3)} \to \G^{S^1}_{G_2} \to \M^{S^1}_{G_2} \to \M_{SU(3)}
\end{equation}
defined by Definition \ref{defin:su3modspacegluingmap} is also a local diffeomorphism. 
\end{thm}
As the result is local, we work on an open subset $\G_{SU(3), ([z_1], [z_2])}$ as in Definition \ref{defin:su3gluingmodspaces}. We apply the inverse mapping theorem, as Nordstr\"om did to prove the corresponding result for $G_2$. He showed (in the proof of \cite[Proposition 4.4]{nordstromgluing}) that the map between harmonic forms given by the derivative is essentially the gluing map given by Definition \ref{comdefin:matchingforms}, and that whilst this gluing map isn't an isomorphism, it is injective, and a complement of its image can be obtained by varying $T$. We have to show the derivative remains an isomorphism when we pre- and post-compose the inclusion by $\iota_{[z_1], [z_2]}$ and the projection $\M_{G_2}^{S^1} \to \M_{SU(3)}$ of Theorem \ref{maintheorema}. 

We need to consider the tangent spaces of $\G_{SU(3), ([z_1], [z_2])}$, $\G^{S^1}_{G_2}$, $\M^{S^1}_{G_2}$ and $\M_{SU(3)}$ and how they are related to each other. By the work in subsection \ref{ssec:gluingmodspacesetup}, we essentially have the following
\begin{prop}
\label{gluingtangentsplitting}
The tangent space to $\G_{SU(3), ([z_1], [z_2])}$ is the direct sum $T_{SU(3)} \oplus TA \oplus T\R$, where $T_{SU(3)}$ is the inclusion of the tangent space to $\B_{SU(3)}$ and $A = A(\Omega_1, \omega_1, \Omega_2, \omega_2)$ is defined in Definition \ref{defin:spacea}. The tangent space to $\G_{G_2}^{S^1}$ is the direct sum $T_{SU(3)} \oplus T_Z \oplus TA \oplus T\R$, where $T_{SU(3)}$ is as before, $T_Z$ is the inclusion of the tangent space to $\B_Z$, and $A = A(\phi_1, \phi_2)$, and the inclusion of $T\G_{SU(3)}$ is given by the inclusion of the appropriate components (recalling that the two groups $A$ have the same tangent spaces by Proposition \ref{su3ghatmanifold}).

The tangent space to $\M_{G_2}^{S^1}$ is the direct sum of the tangent space to $\M_{SU(3)}$ and the tangent space to $Z$ at the corresponding points, and we write these as $T_{SU(3)}$ and $T_Z$. 
\end{prop}
\begin{proof}
Locally, $\G_{SU(3), [z_1], [z_2]}$ is diffeomorphic to $\tilde \G_{SU(3)}$ as it is an open subset. But $\tilde \G_{SU(3)}$ is locally diffeomorphic to $\hat \G_{SU(3)} \times \R$, as a principal bundle, which in turn is locally diffeomorphic to $\B_{SU(3)} \times A(\Omega_1, \omega_1, \Omega_2, \omega_2) \times \R$ by Proposition \ref{su3ghatcharts}. For $\G_{G_2}^{S^1}$, we again work locally and note that Theorem \ref{gluingdataspacesproduct} says that $\tilde \G_{G_2}^{S^1}$ is the product of $\tilde \G_{SU(3)}$ with $\B_Z$. The inclusion is the identity for $TA \oplus T\R$ because Theorem \ref{gluingdataspacesproduct} preserves these components. For the inclusion $\B_{SU(3)} \to \B_{G_2}^{S^1}$, the fact the inclusion is the identity on the $\B_{SU(3)}$ component follows from the definition of $\B_{SU(3)}$ as a submanifold in Proposition \ref{btheorem}, and the corresponding product structure. 

The compact case is even easier.
\end{proof}

We begin by considering Proposition \ref{weakwelldef}, which almost immediately implies
\begin{prop}
\label{gluingztoz}
Suppose that $[\phi_1, \phi_2, T] = [\Re \Omega_1 + z_1 \wedge \omega_1, \Re \Omega_2 + z_2 \wedge \omega_2, T] \in \G_{G_2}^{S^1}$. Suppose that $([z'_1], [z'_2]) \in T\B_Z$, so that $[z'_1 \wedge \omega_1, z'_2 \wedge \omega_2, 0] \in T_Z \subset T\G_{G_2}^{S^1}$. Applying the derivative of the gluing map of Definition \ref{defin:g2modspacegluingmap} takes $[z'_1 \wedge \omega_1, z'_2 \wedge \omega_2, 0]$ into the $T_Z$ component of $\M_{G_2}^{S^1}$.  
\end{prop}
\begin{proof}
Suppose that we have chosen our representatives such that the triple $(\Re \Omega_1 + z_1 \wedge \omega_1, \Re \Omega_2 + z_2 \wedge \omega_2, T)$ glues. Choose a curve of matching pairs of twistings $z_i(s)$ with $z_i(0) = z_i$ and $z'_i = z'_i$, for some representative. Upper semi-continuity of the minimal parameter $T_0$ implies that $(\Re \Omega_1 + z_1(s) \wedge \omega_1, \Re \Omega_2 + z_2(s) \wedge \omega_2, T)$ will also glue for $s$ small. Note that these triples define a curve through $[\Re \Omega_1 + z_1 \wedge \omega_1, \Re \Omega_2 + z_2 \wedge \omega_2, T]$ in the same way as Lemma \ref{smoothcurves}, with its tangent exactly $[z'_1 \wedge \omega_1, z'_2 \wedge \omega_2, 0]$. Proposition \ref{weakwelldef} says that the curve in $\M_{G_2}^{S^1}$ constructed by gluing $(\Re \Omega_1 + z_1(s) \wedge \omega_1, \Re \Omega_2 + z_2(s) \wedge \omega_2, T)$ has fixed $\M_{SU(3)}$ component, and so its tangent lies in $T_Z$. 
\end{proof}
Using the analogous proposition for Calabi-Yau structures (Proposition \ref{su3structurebehaves}), we obtain
\begin{prop}
\label{gluingsu3tosu3}
Suppose that $[\Omega_i, \omega_i, T] \in \G_{SU(3), ([z_1], [z_2])}$, and that $[\Omega'_i, \omega'_i, 0]$ is a tangent vector in $T\G_{SU(3)}$. Pick representatives of $\Omega_i$ and $\omega_i$, and let $z_1, z_2$ be a pair of matching twistings such that the triple $(\Re \Omega_1 + z_1 \wedge \omega_1, \Re \Omega_2 + z_2 \wedge \omega_2, T)$ is gluable. Then the image of the tangent $[\Re \Omega'_i + z_i \wedge \omega'_i, 0] \in T_{SU(3)} \subset T\G_{G_2}^{S^1}$ under the derivative of the gluing map lies in $T_{SU(3)} \oplus \R [z] \subset T\M_{G_2}^{S^1}$, where $[z]$ is the twisting class given by applying the gluing of Definition \ref{comdefin:matchingforms} to $z_1$ and $z_2$. 
\end{prop}
\begin{proof}
Pick some forms $\Omega'_i$ and $\omega'_i$ representing $[\Omega'_i, \omega'_i, 0]$. Take a curve of matching Calabi-Yau structures $(\Omega_i(s), \omega_i(s))$ such that $\Omega_i(0) = \Omega_i$, $\omega_i(0) = \omega_i$, $\Omega'_i = \Omega'_i$, $\omega'_i$. By upper semi-continuity of the minimal parameter $T_0$, $(\Re \Omega_1(s) + z_1 \wedge \omega_1(s), \Re \Omega_2(s) + z_2 \wedge \omega_2(s), T)$ is gluable for $s$ sufficiently small. 

We know by Proposition \ref{su3structurebehaves} that the image of $(\Re \Omega_1(s) + z_1 \wedge \omega_1(s), \Re \Omega_2(s) + z_2 \wedge \omega_2(s), T)$ in $\M_{G_2}^{S^1}$ is of the form $[\Re \Omega(s) + c(s) z \wedge \omega(s)]$ for a curve of Calabi-Yau structures $(\Omega(s), \omega(s))$; the result clearly follows by differentiating in $s$. 
\end{proof}
In the same way, we have a similar result for the automorphism component $TA$. 
\begin{prop}
\label{gluingautostosu3}
\label{autosbehave}
As in Proposition \ref{gluingsu3tosu3}, let $[\Omega_i, \omega_i, T] \in \G_{SU(3), ([z_1], [z_2])}$. Choose representatives $\Omega_i$ and $\omega_i$ and let $z_1, z_2$ be a pair of matching twistings such that the triple $(\phi_1=\Re \Omega_1 + z_1 \wedge \omega_1, \phi_2=\Re \Omega_2 + z_2 \wedge \omega_2, T)$ is gluable. Consider $X \in TA(\Omega_1, \omega_1, \Omega_2, \omega_2) = TA(\phi_1, \phi_2) \subset T\G_{G_2}^{S^1}$. Its image under the derivative of the gluing map lies in $T_{SU(3)} \oplus \R [z]$.
\end{prop}
\begin{proof}
Let $\Phi_s$ be the curve of diffeomorphisms of $M_2$ generated by curve of Killing fields $sX$ as in Proposition \ref{spaceanice}. By upper semi-continuity of the minimal $T_0$, $(\Omega_1, \omega_1, \Phi_s^* \Omega_2,  \Phi_s^* \omega_2, T)$ is gluable for $s$ sufficiently small. Exactly as in Proposition \ref{gluingsu3tosu3}, by Proposition \ref{su3structurebehaves}, the image of $(\Omega_1, \omega_1, \Phi_s^* \Omega_2,  \Phi_s^* \omega_2, T)$ in $\M_{G_2}^{S^1}$ is of the form $[\Re \Omega(s) + c(s) z \wedge \omega(s)]$; the result follows.
\end{proof}
Finally, we have to consider the effect of varying the neck length $T$. Exactly the same argument as in Propositions \ref{gluingztoz}--\ref{gluingautostosu3} shows that it follows from Proposition \ref{lengthstructurebehaves} that $T\R$ maps into $T_{SU(3)} \oplus \R [z]$. 

In sum, therefore, we have shown that $T_Z$ maps into $T_Z$ but that $T \G_{SU(3), [z_1], [z_2]}$ maps into $T\M_{SU(3)} \oplus \R [z]$. 
The proof of Theorem \ref{localdiffeoofmodspace} is now straightforward linear algebra, as follows.
\begin{proof}[Proof of Theorem \ref{localdiffeoofmodspace}]
We consider the derivative of the gluing map
\begin{equation}
\G_{SU(3)} \to \M_{SU(3)}
\end{equation}
around any given point. Because $\G_{SU(3)}$ is given by a union of open sets, we may suppose that this point lies in $\G_{SU(3), ([z_1], [z_2])}$, and then by Definition \ref{defin:su3modspacegluingmaps} the gluing map is locally given by the composition
\begin{equation}
\G_{SU(3), ([z_1], [z_2])} \to \G_{G_2}^{S^1} \to \M_{G_2}^{S^1} \to \M_{SU(3)}
\end{equation}
Its derivative is therefore given in terms of the decomposition in Proposition \ref{gluingtangentsplitting} by
\begin{equation}
T_{SU(3)} \oplus TA \oplus T\R \hookrightarrow T_{SU(3)} \oplus T_{Z} \oplus TA \oplus T\R \overset{\gamma}\to T_{SU(3)} \oplus T_Z \twoheadrightarrow T_{SU(3)}
\end{equation}
where the middle map $\gamma$ is the derivative of the gluing map on $S^1$ invariant moduli spaces and is therefore an isomorphism by Theorem \ref{g2localdiffeoofmodspace}.

Now we have that the components in the tangent space of $\G_{G_2}^{S^1}$ corresponding to the tangent space of $\G_{SU(3)}$ are mapped into $T_{SU(3)} \oplus \R [z]$ (by Propositions \ref{gluingsu3tosu3}, \ref{gluingautostosu3}, and the following comment) and the additional component $T_Z$ is mapped into $T_Z$ (by Proposition \ref{gluingztoz}).  Moreover, we know that $\R [z]$  is mapped to by $T_Z$, by taking the obvious curve of twistings $(1+s)(z_1, z_2)$: that is, by taking $[z_1 \wedge \omega_1, z_2 \wedge \omega_2, 0] \in T_Z$. It follows that the composition is injective, because if a vector in $T_{SU(3)}$ maps to zero under the whole map, then under the $G_2$ gluing it must map to $(0, c[z])$. But $(0, c[z])$ is also the image of something in $T_Z$, which proves that the $G_2$ gluing map is not injective, a contradiction. The composition is also surjective, simply because anything in $T_{SU(3)}$ is mapped to by some tangent vector in $T\G_{G_2}^{S^1}$, and if we ignore the $T_Z$ component we still get the same $T_{SU(3)}$ component under the composition. Hence the composition of derivatives is an isomorphism and the gluing map of Calabi-Yau structures is a local diffeomorphism. 
\end{proof}
\begin{rmk}
We could also have used a simplified version of \cite[Proposition 3.1]{nordstromgluing} to show that the patching in Definition \ref{comdefin:matchingforms} defines an isomorphism between the two $T_Z$ components, rather than analysing the remaining components separately.
\end{rmk}
Finally, we make a remark on the possibility of complex gluing parameter $T$, which is natural when we are gluing complex manifolds. 
\begin{rmk}
By the Haskins--Hein--Nordstr\"om structure theory (\cite{haskinsheinnordstrom}), any asymptotically cylindrical Calabi-Yau has cross-section a finite quotient of $S^1 \times X$ for some $X$, and so has a rotation map on its asymptotically cylindrical end; we could say that taking gluing parameter $T + iS$ corresponds to a rotation of one asymptotically cylindrical end. Evidently, if we permitted complex gluing parameter, the gluing map would still cover an open subset. Local injectivity depends on how variation of $S$ interacts with the proof of Theorem \ref{localdiffeoofmodspace}, in particular how the automorphism of rotation on the end behaves as a class in $A(\Omega_1, \omega_1, \Omega_2, \omega_2)$. If it defines the trivial class, then the moduli class we obtain on gluing is independent of $S$. If it defines a nontrivial class, then we lose local injectivity, as changing $S$ is equivalent to a change in $T\hat\G_{SU(3)}$. 
\end{rmk}
\bibliographystyle{abbrv}
\bibliography{bibfile}
\end{document}